\documentclass[11pt,reqno]{amsart}

\usepackage{amscd} 
\usepackage{amsfonts} 
\usepackage{amsmath} 
\usepackage{amssymb} 
\usepackage{amsthm} 
\usepackage{latexsym} 
\usepackage{mathrsfs} 
\usepackage{mathtools,empheq} 
\usepackage{slashed}  

\usepackage{verbatim}
\usepackage[shortlabels]{enumitem}
\usepackage{cases}
\usepackage[dvips]{epsfig}
\usepackage{epsf} 
\usepackage[bookmarksnumbered,pdfpagelabels=true,plainpages=false,colorlinks=true,
            linkcolor=black,citecolor=black,urlcolor=black]{hyperref}
\usepackage{url}
\usepackage{color}
\usepackage{xcolor}
\usepackage{graphicx}
\usepackage{enumitem} 

\usepackage{pifont} 
\usepackage{upgreek} 
\usepackage{fancyhdr} 
\usepackage{calligra} 
\usepackage{marvosym} 
\usepackage[percent]{overpic} 
\usepackage[hypcap=false]{caption} 
\usepackage{accents}
\usepackage{bm}

\usepackage[papersize={8.5in,11in}, margin=1in,footskip=0.4in,vmarginratio={1:1},hmarginratio={1:1}]{geometry} 
\makeatletter
\def\l@subsection{\@tocline{2}{0pt}{2.6pc}{5pc}{}}
\makeatother

\theoremstyle{plain}
\newtheorem{theorem}{Theorem}[section]

\newtheorem{lemma}[theorem]{Lemma}
\newtheorem{proposition}[theorem]{Proposition}
\newtheorem{corollary}[theorem]{Corollary}

\newtheorem{assumption}[theorem]{Assumption}

\theoremstyle{definition}
\newtheorem{remark}[theorem]{Remark}

\newtheorem{definition}[theorem]{Definition}

\numberwithin{equation}{section}
\allowdisplaybreaks

\newcommand{\norm}[1]{\Vert#1\Vert}

\newcommand{\nnorm}[1]{\left\Vert#1\right\Vert}

\newcommand{\curl}{\mbox{\upshape curl}\mkern 1mu}

\newcommand{\mss}{\hspace{0.2cm}}
\newcommand{\ms}{\hspace{0.25cm}}

\newcommand{\proj}{\mathsf{\Pi}}

\begin{document}
\title[A priori estimates]{A priori estimates for the linearized relativistic Euler equations with a physical vacuum boundary and an ideal gas equation of state}
\author[Luczak]{Brian B. Luczak$^{*,\#}$
}



\thanks{$^{*}$Vanderbilt University, Nashville, TN, USA.
\texttt{brian.b.luczak@vanderbilt.edu}}

\thanks{$^{\#}$ BBL gratefully acknowledges support from NSF grant DMS-2107701,
}

\begin{abstract}
In this paper, we will provide a result on the relativistic Euler equations for an ideal gas equation of state and a physical vacuum boundary. More specifically, we will prove a priori estimates for the linearized system in weighted Sobolev spaces.   Our focus will be on choosing the correct thermodynamic variables, developing a weighted book-keeping scheme, and then proving energy estimates for the linearized system.
\bigskip


\end{abstract}
\maketitle

\tableofcontents

\newpage






















\section{Introduction\label{S:Introduction}}


The starting point for our work is the relativistic Euler equations which describe the motion of a relativistic fluid in a Minkowski background (see \cite{disconzi2023recent} for historical context). The equations of motion are given by
\begin{subequations}{\label{E:Euler}}
\begin{align}
u^\mu \partial_\mu \varrho + (p+\varrho) \partial_\mu u^\mu &= 0,
\label{E:Energy_conservation}\\
(p+\varrho) u^\mu \partial_\mu u^\alpha + \proj^{\alpha\mu} \partial_\mu p &=0,
\label{E:Momentum_conservation}\\
g_{\mu\nu} u^\mu u^\nu & = -1,
\label{E:Constraint_four_velocity}
\end{align}
\end{subequations}
where $\varrho$ is the fluid's (energy) density, $u$ is the fluid's (four-)velocity, $p$ is the 
fluid's pressure given by an equation of state (whose choice depends on the nature of the fluid, see below),  
$g$ is the Minkowski metric\footnote{This problem can be studied for a general background
metric, but the Minkowski metric already contains all the important mathematical features. Coupling
to Einstein's equations, on the other hand, is an entirely different (and much harder) problem.},
and $\proj_{\alpha\beta}:= g_{\alpha\beta} + u_\alpha u_\beta$ is the projection onto the 
space orthogonal to $u$.
Above and throughout, we employ standard Cartesian coordinates $\{ x^\alpha \}_{\alpha=0}^3$
with $t:= x^0$ denoting a time coordinate and $x := (x^1,x^2,x^3)$ denoting spatial coordinates,
the sum convention is adopted, Greek indices vary from $0$ to $3$ and Latin indices from $1$ to $3$,
and indices are raised and lowered with $g$.

\begin{remark}
We note that \eqref{E:Energy_conservation} is the conservation of energy for the fluid, \eqref{E:Momentum_conservation} is the conservation of momentum, and \eqref{E:Constraint_four_velocity} is a normalization condition where the velocity is assumed to be a forward time-like vector field (and this constraint is propagated by the flow).
\end{remark}

We are interested in the case where the fluid is confined within a domain that is not fixed but is allowed
to move with the fluid motion, such as, e.g., the motion of a star.
Fluids of this type are called free-boundary fluids. We further consider the situation
where the pressure and density vanish on the boundary which is often referred to as the gas case. The liquid case where the density does not vanish on the boundary is a different problem.

Denote the region containing the fluid
at time $t$ by $\Omega_t$. Then, $\Omega_t$ can be described as
\begin{align}
\Omega_t = \{  (\tau,x) \in \mathbb{R}^{1+3} \, | \, \tau = t, \varrho(t,x) > 0 \}.
\nonumber
\end{align}
Equations \eqref{E:Euler} hold in the spacetime region
\begin{equation}
\mathscr{D} := \bigcup_{0\leq t < T} \{ t \} \times \Omega_t,
\label{E:moving_domain_definition}
\end{equation}
for some $T>0$, known as the moving domain.
The fluid's free-boundary at time $t$ is defined as
$\Gamma_t : = \partial \Omega_t$, and the fluid's free boundary, 
also called the moving boundary, free interface, or vacuum boundary, is defined as
\begin{align}
\Gamma : = \bigcup_{0\leq t < T} \{ t \} \times \Gamma_t  = 
\bigcup_{0\leq t < T} \{ t \} \times \partial \Omega_t.
\nonumber
\end{align}
Sometimes we also call $\Omega_t$ and $\Gamma_t$ the moving domain and the free boundary, respectively.
In free-boundary problems, understanding the dynamics of $\Gamma$ is crucial, as it is the fact that $\Gamma_t$
moves with time that distinguishes such problems from a standard initial-boundary value problem
where the boundary of the domain is fixed. Note that, according to the foregoing, we have
\begin{equation}
p = \varrho = 0 \, \text{ on } \, \Gamma.
\label{E:Pressure_and_density_vanish_boundary}
\end{equation}

\subsection{Physical vacuum boundary with a barotropic equation of state\label{S:Physical_vacuum_barotropic}}

Before discussing our work for a non-barotropic equation of state, let us consider the barotropic setting which will motivate our discussion. Consider the equation of state:
\begin{align}
p = p(\varrho) = \varrho^{\kappa+1}, \kappa > 0, 
\label{E:Equation_state_barotropic}
\end{align}
where $\kappa$ is constant. Equation \eqref{E:Momentum_conservation} then becomes
\begin{align}
(\varrho^{\kappa+1}+\varrho) u^\mu \partial_\mu u^\alpha + c_s^2
\proj^{\alpha\mu} \partial_\mu \varrho =0,
\label{E:Momentum_conservation_barotropic}
\end{align}
where
\begin{align}
c_s^2 := p^\prime(\varrho) = (\kappa+1)\varrho^\kappa
\label{E:Sound_speed_barotropic}
\end{align}
is the fluid's sound speed \cite{disconzi2023recent} and the second equality in \eqref{E:Sound_speed_barotropic} is valid
for the equation of state \eqref{E:Equation_state_barotropic}. Since $\varrho$ vanishes on the
free-boundary, we also have
\begin{align}
\left. c_s^2 \right|_{\Gamma} = 0.
\label{E:Sound_speed_vansih_boundary}
\end{align}
\begin{remark}
Observe that \eqref{E:Sound_speed_vansih_boundary} follows immediately from 
\eqref{E:Pressure_and_density_vanish_boundary} and \eqref{E:Sound_speed_barotropic}. But
typically one would still impose \eqref{E:Sound_speed_vansih_boundary} for other equations of state
than \eqref{E:Equation_state_barotropic} when \eqref{E:Pressure_and_density_vanish_boundary} holds,
since sound waves should not be allowed to propagate into vacuum across a region where the 
medium (the density) vanishes.
\end{remark}
The decay rate of $c_s^2$ near the free-boundary plays a key role in this problem.
One needs to consider decay rates that allow for $\Gamma_t$ to move with a bounded non-zero acceleration. We have from \eqref{E:Momentum_conservation_barotropic} 
and \eqref{E:Sound_speed_barotropic} that the fluid's (four-)acceleration is
\begin{align}
a^\alpha := u^\mu \partial_\mu u^\alpha =
-\frac{ (\kappa+1)\varrho^{\kappa-1}}{1+\varrho^\kappa}
\proj^{\alpha\mu} \partial_\mu \varrho.
\nonumber
\end{align}
Since $\varrho \sim 0$ near the free-boundary, we have that $1+\varrho = O(1)$ so,
near the boundary,
\begin{align}
a^\alpha \sim  
\varrho^{\kappa-1}
\proj^{\alpha\mu} \partial_\mu \varrho \sim \bm{\partial} \varrho^\kappa \sim
\bm{\partial} c_s^2,
\nonumber
\end{align}
using $\bm{\partial}$ to denote a generic spacetime derivative. At this point, we need to make some 
assumption on the decay rate of $c_s^2$. In view of the finite speed propagation property,
away from the boundary the motion of the fluid is essentially the same as in the case without
a free-boundary. Thus, a natural scale in this problem which allows us to separate the bulk
and boundary behaviors is the distance to the boundary:
\begin{align}
d \equiv d(t,x) = \, \text {distance from } \, x \, \text{ to } \, \Gamma_t.
\nonumber
\end{align}
Therefore, a natural assumption to make on $c_s^2$ is that it decays like a power of $d$, i.e., 
$c_s^2 \sim d^\beta$ for some $\beta >0$. 
Under this assumption, the acceleration becomes
\begin{align}
a \sim \bm{\partial} d^\beta = \beta d^{\beta-1} \bm{\partial} d \sim  d^{\beta-1},
\nonumber
\end{align}
where we used that $\bm{\partial} d = O(1)$. From this, we see that 
\begin{align}
\left. a \right|_{\Gamma_t} = 
\begin{cases}
0, & \beta > 1,\\
\infty, & 0<\beta < 1,\\
\text{finite}\neq 0, & \beta = 1.
\end{cases}
\nonumber
\end{align}
Hence, we expect a realistic dynamics only in the case $\beta=1$. 
Therefore,
we assume that 
\begin{align}
c_s^2(t,x) \approx d(t,x) \, \text{ for }\, x\, \text{ near } \, \Gamma_t,
\label{E:Physical_boundary}
\end{align}
i.e., $c_s^2$ is comparable to the distance to the boundary. Condition 
\eqref{E:Physical_boundary} is known as the \textbf{physical vacuum boundary condition.} One should
understand \eqref{E:Physical_boundary} as a constraint, i.e., something that is imposed on the initial
data and then propagated by the flow. The local well-posedness of the relativistic Euler
equations with \eqref{E:Equation_state_barotropic} as equation of state and initial data
satisfying the physical vacuum boundary condition was obtained in \cite{DisconziIfrimTataru} using Eulerian coordinates. Additionally, a-priori estimates for the same problem with equation of state \eqref{E:Equation_state_barotropic} were obtained in \cite{Hadzic-Shkoller-Speck-2019} using Lagrangian coordinates. In \cite{Jang-LeFloch-Masmoudi-2016}, the authors proved a priori estimates using an equation of state where the pressure is assumed to be a power law of the fluid's baryon density $n$ (see \eqref{E:Baryon_density_conservation} below). 

\subsection{The case of an ideal gas}
\label{S:Physical_vacuum_non_barotropic}

We would like to extend the results of \cite{DisconziIfrimTataru} to the case of a non-barotropic equation
of state. In addition, we would like, if possible, to treat situations that are of direct relevance to physicists
working in numerical simulations of star evolution. A typical equation of state used by physicists
in their simulations of stars is that of an \textbf{ideal gas} (see \cite{RezzollaZanottiBookRelHydro-2013} Section 2.4), 
\begin{align}
p = p(n,\varepsilon) = n \varepsilon (\gamma -1 ),
\label{E:Equation_state_ideal_gas}
\end{align}
where $n$ is the fluid's baryon density, $\varepsilon$ is the fluid's internal energy, and $\gamma >1$
is a constant. We recall that $n$ satisfies the equation
\begin{align}
u^\mu\partial_\mu n + n \partial_\mu u^\mu = 0,
\label{E:Baryon_density_conservation}
\end{align}
and $\varepsilon$ is defined through the relation
\begin{align}
\varrho = n(1 + \varepsilon).
\label{E:Definition_internal_energy}
\end{align}
Equation \eqref{E:Baryon_density_conservation} is known as the conservation of baryon density \cite{disconzi2023recent}. This provides a closed system in terms of the unknowns $(\varepsilon, n, u^\mu)$.

Since $\varepsilon\geq 0$, equation \eqref{E:Definition_internal_energy} implies that 
$n$ also vanishes on the free boundary. In addition, since we must have $\varepsilon = 0$
in a vacuum, we also need to impose that $\varepsilon = 0$ on $\Gamma_t$.
We need to find conditions that play a similar role to the physical vacuum boundary
condition \eqref{E:Physical_boundary} employed for barotropic fluids. As in Section
\ref{S:Physical_vacuum_barotropic}, we try to determine such conditions by ensuring
a finite non-zero boundary acceleration.

We now adopt $(\varepsilon, n, u)$ as primitive variables. In this situation, 
using \eqref{E:Equation_state_ideal_gas}, equation \eqref{E:Momentum_conservation} becomes
\begin{align}
\frac{1 + \varepsilon \gamma}{\gamma-1} n u^\mu \partial_\mu u^\alpha 
+ \varepsilon \proj^{\alpha\mu} \partial_\mu n + n \proj^{\alpha\mu} \partial_\mu \varepsilon
=0.
\label{E:Momentum_conservation_non_barotropic}
\end{align}
Inspired by the discussion of Section \ref{S:Physical_vacuum_barotropic}, we suppose that
\begin{align}
n \sim d^\beta, \, \varepsilon \sim d^\sigma, \, \beta > 0, \, \sigma > 0.
\nonumber
\end{align}
Using this Ansatz into equation \eqref{E:Momentum_conservation_non_barotropic} and proceeding
as in Section \ref{S:Physical_vacuum_barotropic}, we find
\begin{align}
a \sim d^{\sigma-1}.
\nonumber
\end{align}
An interesting observation is that $d^\beta$ cancels out and we are left with
\begin{align}
\left. a \right|_{\Gamma_t} = 
\begin{cases}
0, & \sigma > 1,\\
\infty, & 0<\sigma < 1,\\
\text{finite}\neq 0, & \sigma = 1.
\end{cases}
\nonumber
\end{align}
Therefore, we conclude that in the case of an equation of state of an ideal gas given
by \eqref{E:Equation_state_ideal_gas}, the physical vacuum boundary conditions should be
\begin{align}
\varepsilon(t,x) \approx d(t,x)
\,\text{ and }\,
n(t,x) \approx (d(t,x))^\beta,\, \beta > 0, 
 \, \text{ for }\, x\, \text{ near } \, \Gamma_t.
\label{E:Physical_boundary_ideal_gas}
\end{align}

Using \eqref{E:Equation_state_ideal_gas} and \eqref{E:Definition_internal_energy}, 
in terms of $(\varepsilon,n,u)$, equations \eqref{E:Euler} and \eqref{E:Baryon_density_conservation} read
\begin{subequations}{\label{E:Euler_ideal_gas}}
\begin{align}
u^\mu \partial_\mu \varepsilon + (\gamma-1) \varepsilon \partial_\mu u^\mu &= 0,
\label{E:Internal_energy_conservation_ideal_gas}\\
\frac{1 + \varepsilon \gamma}{\gamma-1} n u^\mu \partial_\mu u^\alpha 
+ \varepsilon \proj^{\alpha\mu} \partial_\mu n + n \proj^{\alpha\mu} \partial_\mu \varepsilon &=0,
\label{E:Momentum_conservation_ideal_gas}\\
u^\mu\partial_\mu n + n \partial_\mu u^\mu &= 0,
\label{E:Baryon_density_conservation_ideal_gas}\\
g_{\mu\nu} u^\mu u^\nu & = -1.
\label{E:Constraint_four_velocity_ideal_gas}
\end{align}
\end{subequations}

\begin{remark}
We observe that \eqref{E:Physical_boundary_ideal_gas} is what we obtained solely from
an analysis of the boundary acceleration. However, from the thermodynamic relations and the equation of state (see \eqref{E:s_simplification} below), we have that the entropy $s$ is given by  
\begin{equation*}
    s= \frac{1}{\gamma-1} \log\left(\frac{\varepsilon}{n^{\gamma-1}} \right) +s_0.
\end{equation*}
where $s_0$ is a constant.
From the physical boundary conditions, near the free boundary we would have that 
\begin{equation*}
    s \sim \frac{1}{\gamma-1} \log\left(\frac{d}{d^{(\gamma-1) \beta}}\right) + s_0 \sim \log(d^{1-\beta(\gamma-1)}) + s_0. 
\end{equation*}

Thus, as we approach the free boundary, $d \to 0$ and $s$ would behave like

\begin{equation}
\label{E:s_behavior_vaccum_boundary}
    s \to \begin{cases}
                -\infty &, \text{ for }\beta < \frac{1}{\gamma-1} \\
                \text{finite value} &, \text{ for } \beta = \frac{1}{\gamma-1} \\
                \infty & , \text{ for } \beta > \frac{1}{\gamma-1}.
                \end{cases}
\end{equation}  
Thus, $\beta = \frac{1}{\gamma-1}$ is a natural condition for the decay rate of $n$ in our problem. However, in what follows, our analysis will be solely with respect to the variables $s, u$, and $r$, where $r$ is a multiple of $c_s^2$ that we introduce in Section \ref{S:New_variables}. We make this remark to indicate that there is a natural choice for the physical vacuum boundary condition in the case of an ideal gas.  
\end{remark}

\subsection{Main Result}
Here, we summarize our main result which is based on the linearized version of system \eqref{E:Euler_ideal_gas} with variables $s, u$ and $r$ (where $r$ is defined in \eqref{E:r_definition} below). See \eqref{E:System} below for a rewriting of \eqref{E:Euler_ideal_gas} with our new variables, and \eqref{E:Linearized_System_2} for the full linearized system. We remark that Theorem \ref{T:A_priori_estimates_first_version} below is a statement amount the linearized problem, where the non-linear variables $(s,r,u)$ serve as background data.

\begin{theorem}[Sobolev estimates for the linearized system]
\label{T:A_priori_estimates_first_version}
Let $(s,r,u)$ be a smooth solution to \eqref{E:System} that exists on some time interval $[0,T]$, and for which the physical vacuum boundary condition \eqref{E:Physical_boundary_ideal_gas} holds. Let $(\tilde{s}_0, \tilde{r}_0, \tilde{u}_0)$ be initial data to system \eqref{E:Linearized_System_2}. Then, there exists a constant $\mathcal{C}$ depending only on $s, r, u,$ and $T$ such that, if $(\tilde{s}, \tilde{r}, \tilde{u}) \in C^{\infty}(\overline{\mathscr{D}})$ is a solution to \eqref{E:Linearized_System_2} on $[0,T]$, then
\begin{equation}
    \norm{(\tilde{s}, \tilde{r}, \tilde{u})}_{\mathcal{H}^{2k}(\Omega_t)} \lesssim \mathcal{C} \norm{(\tilde{s}_0, \tilde{r}_0, \tilde{u}_0)}_{\mathcal{H}^{2k}(\Omega_0)}
\end{equation}
where $\mathcal{H}^{2k}(\Omega_t)$ is defined in \eqref{E:Higher_order_H^2k_norm}.
    
\end{theorem}
See Theorem \ref{Th:Main_Theorem_Estimates_in_H^2k} in Section \ref{S:Main_Theorem} for a proof of our main result.

\begin{remark}
In Theorem \ref{T:A_priori_estimates_first_version}, we assume smooth solutions for simplicity. Our quantitative bounds depend only on finite regularity norms as indicated in Theorem  \ref{T:A_priori_estimates_first_version}, and in similar theorems.
    
\end{remark}

\begin{remark}

We note that free boundary fluids have been studied in a variety of contexts since the 1930s, see \cite{Tolman:1939jz, Oppenheimer:1939ne, Tolman:1934za} for such examples. For more recent papers, we refer the reader to \cite{Caffarelli-1979, Matu-1997, Shatah-2008, Makino-Ukai-1995-II,LeFloch-Ukai-2009, Trakhinin-2009, Zhang-2008,  DisconziKukavicaTuffaha, Jang-LeFloch-Masmoudi-2016}. For recent work in free-boundary problems from relativity, we refer the reader to \cite{Miao-Shahshahani2024, Miao-Shahshahani-Wu-2020-arxiv}. See \cite{disconzi2023recent} and 
\cite{DisconziIfrimTataru} for a review of recent developments in mathematical aspects of relativistic fluids.

\end{remark}

\begin{remark}{(Schematic notation for derivatives).}
\label{R:notation_for_derivatives}
We will use the following schematic notation when discussing derivatives in the arguments that follow. We will use the symbol 
\begin{itemize}
    \item $\partial $ to denote a generic spatial derivative, i.e. $\partial_1, \partial_2,$ or $\partial_3$, and 
    \item $\bm{\partial}$ to denote a generic \textbf{spacetime} derivative, i.e. $\partial_t, \partial_1, \partial_2,$ or $\partial_3$.
\end{itemize}
In view of subsection \ref{S:Sovling_for_time_derivatives}, it is important to note that there is significant overlap between these two symbols. Namely, one can use our primary system \eqref{E:System} to solve for the time derivative of a given quantity in terms of the spatial derivatives in a way that, as it turns out, does not introduce problematic terms into our estimates. 
\end{remark}

\subsection{The linearized system in terms of entropy and sound speed (squared)}
\label{S:New_variables}

Our goal is to rewrite \eqref{E:Euler_ideal_gas} so that we can apply energy estimate techniques. For our purposes, it is more effective to consider the unknowns $(s,r,u^\mu)$ where $s$ denotes the entropy per particle, $r$ is a multiple of the sound speed (squared), and $u$ is the usual (four)-velocity. We will often abuse language and simply refer to $r$ as the sound speed (squared) or just the sound speed.

We begin by rewriting \eqref{E:Euler_ideal_gas} in terms of $p$, and then introducing our new variables which will be used in the remainder of the work.

First, using \eqref{E:Equation_state_ideal_gas} and \eqref{E:Internal_energy_conservation_ideal_gas}, we observe that
\begin{equation}
\begin{split}
    u^\mu \partial_\mu p &= (\gamma-1)n u^\mu \partial_\mu \varepsilon +  (\gamma-1) \varepsilon u^\mu \partial_\mu n \\
    &= -(\gamma-1)^2 n \varepsilon \partial_\mu u^\mu - (\gamma-1)n \varepsilon \partial_\mu u^\mu \\
    &= - \gamma p \partial_\mu u^\mu.
\end{split}
\end{equation}
Then, rewriting \eqref{E:Momentum_conservation_ideal_gas} in terms of $\varepsilon$ and $p$, our system \eqref{E:Euler_ideal_gas} takes the form
\begin{subequations}
\begin{align}
\label{E:New_eq_1}
u^\mu \partial_\mu \varepsilon + (\gamma-1)\varepsilon \partial_\mu u^\mu &=0\\
\label{E:New_eq_2}
\frac{1+\varepsilon \gamma}{\gamma-1} p u^\mu \partial_\mu u^\alpha + \varepsilon \proj^{\alpha \mu} \partial_\mu p &= 0\\
\label{E:New_eq_3}
u^\mu \partial_\mu p  +\gamma p \partial_\mu u^\mu &=0.
\end{align}
\end{subequations}

In the following lines, we will make use of several identities \eqref{E:enthalpy_definition} \eqref{E:thermodynamic_relations}, and \eqref{E:p=ntheta}, which can be found in Section 2 of \cite{RezzollaZanottiBookRelHydro-2013}.

Recall that the enthalpy per particle $h$ is defined by
\begin{equation}
    \label{E:enthalpy_definition}
    h = \frac{p + \varrho }{n} = \frac{n \varepsilon(\gamma-1) + n(1+\varepsilon))}{n} = \varepsilon\gamma +1
\end{equation}
which has been simplified using \eqref{E:Equation_state_ideal_gas} and \eqref{E:Definition_internal_energy}.
Using the well known thermodynamic relation
\begin{equation}
\label{E:thermodynamic_relations}
    n dh - dp = n \theta ds,
\end{equation}
where $\theta$ is the temperature for the fluid, and the fact that 
\begin{equation}
\label{E:p=ntheta}
    p = n \theta
\end{equation}
for an ideal fluid, we can solve for $s$ by
\begin{equation}
\label{E:s_simplification}
\begin{split}
   & n \gamma d\varepsilon - (\gamma-1)(n d\varepsilon +\varepsilon dn) = n \varepsilon (\gamma-1) ds \\
   & \implies s = \frac{1}{\gamma-1} \log\left(\frac{\varepsilon}{n^{\gamma-1}} \right) + s_0
\end{split}
\end{equation}
where $s_0$ is a constant (and we will often set $s_0 =0$ for convenience). Additionally, a quick computation (such as in \cite{disconzi2023recent}) shows that \eqref{E:Baryon_density_conservation_ideal_gas} along with the thermodynamic relations implies 
\begin{equation}
\label{E:Entropy_conservation}
u^\mu \partial_\mu s = 0.
\end{equation}
which is the conservation of entropy along flow lines.
Moreover, \eqref{E:s_simplification} allows us to solve for $\varepsilon$ in terms of $s$ and $p$. We see that $e^{(\gamma-1) s} = \frac{\varepsilon}{n^{\gamma-1}} = \frac{1}{\gamma-1} \cdot \frac{p}{n^\gamma}$, which leads to
\begin{equation}
\label{E:internal_energy_simplified}
     \varepsilon = \frac{e^{\frac{\gamma-1}{\gamma} s}}{(\gamma-1)^{\frac{\gamma-1}{\gamma}}} p^{\frac{\gamma-1}{\gamma}}
\end{equation}
Now, substituting for $\varepsilon$ in \eqref{E:New_eq_2} and mulitplying by $(\gamma-1)^{\frac{\gamma-1}{\gamma}}$, we get
\begin{equation}
\label{E:Second_with_A}
    C_1 p u^\mu \partial_\mu u^\alpha + e^{\frac{\gamma-1}{\gamma} s} p^{\frac{\gamma-1}{\gamma}} \proj^{\alpha \mu} \partial_\mu p = 0
\end{equation}
where
\begin{equation*}
    C_1 =\frac{1}{(\gamma-1)^{\frac{1}{\gamma}}} + e^{\frac{\gamma-1}{\gamma} s} p^{\frac{\gamma-1}{\gamma}} \frac{\gamma}{\gamma-1}.
\end{equation*}
Note that $C_1$ is $O(1)$ as we approach the free boundary since $p \to 0$.

Before proceeding, we can quickly verify that the acceleration of the free boundary is $O(1)$ using \eqref{E:Physical_boundary_ideal_gas}. Near the free boundary, $\varepsilon \sim d$ where $d$ is the distance to the boundary. Using \eqref{E:internal_energy_simplified}, we get $p \sim d^{\frac{\gamma}{\gamma-1}}$. In this case, 
\begin{equation}
    a^\alpha = u^\mu \partial_\mu u^\alpha \sim p^{-\frac{1}{\gamma}} \bm{\partial} p \sim \left(d^{\frac{\gamma}{\gamma-1}} \right)^{-\frac{1}{\gamma}} d^{\frac{1}{\gamma-1}} = O(1).
\end{equation}
Because, $d \sim p^{\frac{\gamma-1}{\gamma}}$, this motivates the definition of our new variable
\begin{equation}
\label{E:r_definition}
    r := p^{\frac{\gamma-1}{\gamma}}
\end{equation}
with the idea that $r \sim d$ and $r$ will serve as our weight in the energy for the free boundary problem. After computing the sound speed squared given by $c_s^2 = \frac{dp}{d \varrho} \vert_s$, we further note that $r$ is comparable to the sound speed (squared) for the fluid and it will play a similar role as in the work from \cite{DisconziIfrimTataru}.
Using \eqref{E:New_eq_3} above, we can quickly verify that
\begin{equation}
    u^\mu \partial_\mu r = \frac{\gamma-1}{\gamma} p^{-\frac{1}{\gamma}} u^\mu \partial_\mu p = - (\gamma-1) r \partial_\mu u^\mu 
\end{equation}
and thus \eqref{E:New_eq_3} becomes
\begin{equation}
    u^\mu \partial_\mu r + (\gamma -1) r \partial_\mu u^\mu =0
\end{equation}

In order to complete the system in terms of $(s,r,u^\mu)$ we must further rewrite the second equation \eqref{E:Second_with_A}. Observing that $\partial_\mu p = \frac{\gamma}{\gamma-1} p^{\frac{1}{\gamma}} \partial_\mu r$, we get
\begin{equation}
    \begin{split}
        C_1 u^\mu \partial_\mu u^\alpha + e^{\frac{\gamma-1}{\gamma} s} p^{-\frac{1}{\gamma}} \proj^{\alpha \mu} \partial_\mu p =   C_1 u^\mu \partial_\mu u^\alpha + e^{\frac{\gamma-1}{\gamma} s} \frac{\gamma}{\gamma-1} \proj^{\alpha \mu} \partial_\mu r &= 0 
    \end{split}
\end{equation}
Then, calling
\begin{equation}
    C_2 := C_1 \cdot \frac{\gamma -1}{\gamma} e^{-\frac{\gamma-1}{\gamma}s},
\end{equation}
we get the form of the system in terms of the unknowns $(s, r, u^\mu)$:
\begin{subequations}
\begin{align}
\label{E:Final_eq_1}
u^\mu \partial_\mu s =0\\
\label{E:Final_eq_2}
 u^\mu \partial_\mu u^\alpha + \frac{1}{C_2} \proj^{\alpha \mu} \partial_\mu r &= 0\\
\label{E:Final_eq_3}
 u^\mu \partial_\mu r + (\gamma -1) r \partial_\mu u^\mu &=0.
\end{align}
\end{subequations}
where $C_2= \frac{(\gamma-1)^{\frac{\gamma-1}{\gamma}}}{\gamma e^{\frac{\gamma-1}{\gamma}s}} +r$. For simplicity, let's define
\begin{equation}
\label{E:Gamma_definition}
    \Gamma:= \Gamma(s) = \frac{(\gamma-1)^{\frac{\gamma-1}{\gamma}}}{\gamma e^{\frac{\gamma-1}{\gamma}s}}
\end{equation}
so that $C_2 = \Gamma(s) +r$.

Finally, using the convective derivative which we denote as
\begin{equation}
    D_t = u^\mu \partial_\mu,
\end{equation}
our new system in terms of the unknowns $(s,r,u^\alpha)$ takes the form
\begin{subequations}
\label{E:System}
\begin{empheq}[box=\fbox]{align}
\label{E:System_s_equation}
 D_t s &=0\\
\label{E:System_r_equation}
    D_t r + (\gamma-1) r \partial_\mu u^\mu &=0\\
\label{E:System_v_equation}
    D_t u^\alpha + \frac{1}{\Gamma +r} \proj^{\alpha \mu} \partial_\mu r &=0\\
\label{E:v_constraint_equation}
    g_{\alpha \beta} u^\alpha u^\beta &= -1
    \end{empheq}
\end{subequations}
where $\Gamma=\Gamma(s)$ is defined according to \eqref{E:Gamma_definition}.


Now that the form of our system is complete, we will compute the full linearized equations using the standard procedure with linearized variables. 
Consider a one-parameter family of solutions $\{s_\tau, r_\tau, u_\tau \}_{\tau}$ for the main system \eqref{E:System} such that $(s_\tau, r_\tau, u_\tau)\vert_{\tau=0} = (s,r,u)$. We can now formally define the function $\delta = \frac{d}{d\tau} \vert_{\tau =0}$ and linearized variables $(\tilde{s}, \tilde{r}, \tilde{u}) := (\delta s, \delta{r}, \delta u)$. These variables will solve the linearized system that we now compute by taking $\delta$ of the equations in \eqref{E:System}.
After a computation, we can write the linearized equations in the form: 
\begin{subequations}
\label{E:Linearized_System_2}
\begin{align}
\label{E:Linearized_System_2_s_equation}
    D_t \tilde{s}  &= f\\
\label{E:Linearized_System_2_r_equation}
    D_t \tilde{r} + \textcolor{blue}{\tilde{u}^\mu \partial_\mu r} + \textcolor{blue}{(\gamma-1) r \partial_\mu \tilde{u}^\mu}  &= g\\
\label{E:Linearized_System_2_u_equation}
    D_t \tilde{u}^\alpha + \textcolor{blue}{ \frac{1}{\Gamma+r} \proj^{\alpha \mu} \partial_\mu \tilde{r}}  &= h^\alpha \\
\label{E:Linearized_System_2_constraint_equation}
\tilde{u}_\alpha u^\alpha &=0 
\end{align}
\end{subequations}
where $\Gamma$ is defined according to \eqref{E:Gamma_definition}, and
\begin{equation}
\label{E:Linearized_f_g_h}
    \begin{split}
        & f= -\tilde{u}^\mu \partial_\mu s \\
        & g= - (\gamma-1)\tilde{r} \partial_\mu u^\mu\\
        & h^\alpha = -\tilde{u}^\mu \partial_\mu u^\alpha - \frac{1}{\Gamma +r}(\tilde{u}^\alpha u^\mu + u^\alpha \tilde{u}^\mu)\partial_\mu r + \frac{\Gamma'}{(\Gamma+r)^2} \tilde{s} \proj^{\alpha \mu}\partial_\mu r + \frac{1}{(\Gamma+r)^2} \tilde{r} \proj^{\alpha \mu} \partial_\mu r
    \end{split}
\end{equation}
\begin{remark}
We note that the terms in \textcolor{blue}{blue} will later combine to form a perfect derivative that will assist in our energy estimates. Additionally, each of the terms on the RHS contains undifferentiated $(\tilde{s}, \tilde{r}, \tilde{u})$ with coefficients of the form $(\bm{\partial} s, \bm{\partial} r, \bm{\partial} u)$.
\end{remark}

To clarify the exposition, we will be making the following assumption which can be used when handling the relativistic Euler equations with a physical vacuum boundary (see \cite{DisconziIfrimTataru} and \cite{Hadzic-Shkoller-Speck-2019} for similar examples).

\begin{assumption}[Uniform smallness of $r$]
\label{A:Assumption_smallness r}
\nonumber
Without loss of generality, we will assume that for each $t \in [0, T]$, $r$ is uniformly small in $\Omega_t$:
$$
    \norm{r}_{L^\infty(\Omega_t)} \leq \hat{\varepsilon} \ll \frac{1}{2}
$$
where $\hat{\varepsilon}$ is a small positive constant that we fix for the remainder of the paper. By `$\ll \frac{1}{2}$', we mean, in particular, that there exists a large positive constant $C$ such that $C \hat{\varepsilon} < \frac{1}{2}$. We will fix this constant $C$ for the remainder of the paper.
\end{assumption}
Recall from the physical vacuum boundary conditions and \eqref{E:r_definition} that this assumption will be verified in a small neighborhood of the free boundary $\Gamma$. Due to finite speed of propagation, the behavior of the fluid away from the boundary is non-degenerate and a priori estimates in Sobolev spaces can be handled using methods from the standard relativistic Euler equations. The removal of Assumption \ref{A:Assumption_smallness r} requires a partition of unity argument in which one must separate the fluid into its bulk behavior away from the free boundary and its behavior near the free boundary. By using Assumption \ref{A:Assumption_smallness r}, we are able to focus on the essential difficulties posed by the degenerate nature of the free boundary where $\norm{r}_{L^\infty(\Omega_t)}\leq \hat{\varepsilon} \ll \frac{1}{2}$.

\begin{remark}
\label{R:Omega_t_near_free_boundary}
In what follows, we will often refer to the `smallness' of $r$ near the free boundary which will allow certain terms to make much smaller contributions. We will use $\hat{\varepsilon}$ when appropriate to reference Assumption \ref{A:Assumption_smallness r}.

\end{remark}

\begin{remark}
    \label{R:Interplay between epsilons}
    Additionally, when handling these `small' terms (see Remark \ref{R:order_free_boundary_term} in Section 6), we will often use the Cauchy-Schwarz inequality with $\varepsilon$. Thus, it is useful to fix the interplay between $\hat{\varepsilon}$ and $\varepsilon$ so as to not cause any confusion. Given that $\hat{\varepsilon} \ll \frac{1}{2}$ is fixed by Assumption \ref{A:Assumption_smallness r}, and we also have $\frac{1}{C-1} < \frac{1}{2}$ where $C$ is the fixed positive constant from Assumption \ref{A:Assumption_smallness r}.
    We will use $\varepsilon$ in applications of the Cauchy-inequality-with-$\varepsilon$ to be a small positive constant such that
    \begin{equation}
    0 < \frac{1}{C -1} < \varepsilon < \frac{1}{2}
    \end{equation}
    Note that the previous definitions make the following inequalities true 
    \begin{equation}
    \label{E:epsilon_hat_and_epsilon_inequality}
    \begin{split}
        0 < \hat{\varepsilon} < \frac{1}{2C} < \frac{1}{C-1} < \varepsilon &< \frac{1}{2}\\
        \left(1 + \frac{1}{\varepsilon}\right) \hat{\varepsilon} & < C \hat{\varepsilon} \\
    \end{split}
    \end{equation}
\end{remark}

\begin{remark}[The symbol `$\lesssim$']
\label{R:lesssim_symbol_notation}
Throughout our work, the symbol $\lesssim$ will be used in the usual fashion, i.e.
\begin{equation*}
    A \lesssim B \Longleftrightarrow A \leq \mathcal{D} B
\end{equation*}
where $\mathcal{D}$ is a constant depending on the the fixed data of the problem. In our case, $\mathcal{D}$ will depend on the background solution $(s, r, u)$ along with $\gamma$ and $T$. By slight abuse of notation, we will sometimes write
$A \lesssim \mathcal{D} B$ or $A \lesssim \mathcal{D}(\partial^{l}(s,r,u)) B$ if we want to highlight, for example, that $\mathcal{D}$ depends on $l$ derivatives of our fixed data $(s,r,u)$. This is also done to assist the reader in following our estimates where these quantities are often removed from an integral with the $L^\infty$ norm.
    
\end{remark}

\subsection{Function spaces and Energies}
\label{S:Function_spaces}

Here, we define the function spaces and energies that will play a key role in our work. We denote the weighted $L^2$ spaces with weight $h$ as $L^2(h)$ and equip them with the norm
\begin{equation}
    \norm{f}_{L^2(h)}^2 = \int_{\Omega_t} h |f|^2 \ms dx.
\end{equation}

We will assume throughout that $r$ is a positive function on $\Omega_t$ which vanishes simply on the free boundary $\Gamma_t$. Thus, $r$ will be comparable to the distance to $\Gamma_t$.

For pairs of functions in our system defined on $\Omega_t$, we will use the base Hilbert space
\begin{equation}
    \mathcal{H}=  L^2(r^{\frac{2-\gamma}{\gamma-1}}) \times L^2(r^{\frac{1}{\gamma-1}})
\end{equation}
with the usual norm depending only on the weight $r$. However, we will often use an equivalent norm which is compatible with our energies in the problem. Let $G_{\alpha \beta}$ be the following two form 
\begin{equation}
    \label{E:G_metric_definition}
    G_{\alpha \beta} := g_{\alpha \beta} + 2 u_\alpha u_\beta
\end{equation}
which, by \cite{Hadzic-Shkoller-Speck-2019}, has been shown to be a Riemannian metric in spacetime.

Then, we define the norm
\begin{equation}
\label{E:H^0_norm}
    \norm{(\tilde{r}, \tilde{u})}_{\widetilde{\mathcal{H}}}^2 = \int_{\Omega_t} r^{\frac{2-\gamma}{\gamma-1}} \left( \frac{1}{\gamma-1} \tilde{r}^2 + (\Gamma +r) r |\tilde{u}|^2 \right) \mss dx
\end{equation}
where $|\tilde{u}|^2 = |\tilde{u}|^2_G = G_{\alpha \beta} \tilde{u}^\alpha \tilde{u}^\beta \geq 0$.

\begin{remark}
\label{R:weights_in_wave_energy_tilde_H_norm}
We observe that the two norms $\norm{\cdot}_\mathcal{H}$ and $\norm{\cdot}_{\widetilde{\mathcal{H}}}$ are equivalent on the space of pairs $(\tilde{r}, \tilde{u})$ due to a key fact about the weights appearing in the energy. For $x$ sufficiently close to $\partial \Omega_t$, we have that $r \to 0$, $s$ approaches a finite value $C_s$ by \eqref{E:s_behavior_vaccum_boundary}, and the weight
\begin{equation}
\label{E:weight_fact_1}
    (\Gamma+r) \to \frac{(\gamma-1)^{\frac{\gamma-1}{\gamma}}}{\gamma e^{\frac{\gamma-1}{\gamma} C_s }}  >0
\end{equation}
which is a finite positive constant bounded away from $0$ (as $\gamma>1$ is fixed, and $C_s$ fixed). Hence, $\norm{\cdot}_\mathcal{H}$ and $\norm{\cdot}_{\widetilde{\mathcal{H}}}$ are equivalent near the free boundary.
\end{remark}

We also define the higher order weighted Sobolev spaces $H^{j, \sigma}$ with integer $j \geq 0$ and $\sigma > -\frac{1}{2}$ to be the space of all distributions in $\Omega_t$ whose norm
\begin{equation}
    \label{E:Weighted-Sobolev spaces}
    \norm{f}_{H^{j,\sigma}}^2 = \sum_{|\alpha| \leq j } \norm{r^{\sigma} \partial^\alpha f}_{L^2}^2 
\end{equation}
is finite. For higher regularity, we will use the following higher order Sobolev spaces where the powers of $r$ are linked to the number of derivatives.
Similar to \cite{DisconziIfrimTataru}, we will define higher order function spaces $\mathcal{H}^{2k}$ on triplets $(\tilde{s}, \tilde{r}, \tilde{u})$ in $\Omega_t$ with the following norm
\begin{equation}
\label{E:Higher_order_H^2k_norm}
\begin{split}
     \norm{(\tilde{s}, \tilde{r}, \tilde{u})}_{\mathcal{H}^{2k}}^2
     &= \sum_{|\alpha| =0}^{2k} \sum_{\substack{a=0 \\ |\alpha|-a \leq k}}^k \norm{r^{\frac{2-\gamma}{2(\gamma -1)} + \frac{1}{2} +a } \partial^{\alpha } \tilde{s}}_{L^2}^2 +  \norm{r^{\frac{2-\gamma}{2(\gamma -1)} +a } \partial^{\alpha } \tilde{r}}_{L^2}^2 + \norm{r^{\frac{2-\gamma}{2(\gamma -1)} + \frac{1}{2} +a  } \partial^{\alpha } \tilde{u}}_{L^2}^2 \\
\end{split}
\end{equation}
where we note that the $r$ weight in the $\tilde{r}$ norm is $\frac{1}{2}$ less than the weights on $\tilde{s}$ and $\tilde{u}$ (which is a direct consequence of the powers of $r$ in \eqref{E:Linearized_System_2}). Additionally, these higher order function spaces are based around an even number of derivatives due to the underlying wave-like operator that governs the second order evolution of \eqref{E:Linearized_System_2}. Taking $D_t$ of \eqref{E:Linearized_System_2} illustrates this underlying operator which is a variable coefficient version of $D_t^2 - r \Delta$. 

\begin{remark}
\label{R:H^2k_norm_equivalence}
    Similar to \cite{DisconziIfrimTataru} and \cite{Ifrim-Tataru-2020}, we can use embedding theorems to show that the $\mathcal{H}^{2k}$ norm is equivalent to the $H^{2k, \frac{1}{2(\gamma-1)} + k} \times H^{2k, \frac{2-\gamma}{2(\gamma-1)} +k } \times H^{2k, \frac{1}{2(\gamma-1)} + k}$ norm. 
\end{remark}

\begin{remark}
We remark that the $\widetilde{\mathcal{H}}$ norm will be used to control convective derivatives of $\tilde{r}$ and $\tilde{u}$ which satisfy a system of wave equations (see  \eqref{E:E^2k_wave} for the definition of $E^{2k}_{\text{wave}})$. Meanwhile, $\tilde{u}$ will be decomposed into its vorticity part which, alongside $\tilde{s}$, satisfies a transport equation that we can estimate directly  (see  \eqref{E:E^2k_transport} for the definition of $E^{2k}_{\text{transport}})$. 

\end{remark}

Here, we introduce energies that will be used in the higher order analysis. First, we have the wave energy
\begin{equation}
\label{E:E^2k_wave}
    E^{2k}_{\text{wave}}(\tilde{s}, \tilde{r}, \tilde{u}) = \sum_{j=0}^{k} \norm{(D_t^{2j} \tilde{r}, D_t^{2j} \tilde{u}  )}_{\widetilde{\mathcal{H}}}^2
\end{equation}
recalling \eqref{E:H^0_norm}. Note that the definition of $\tilde{\mathcal{H}}$ guarantees that $|D_t^{2j} \tilde{u}|^2 =|D_t^{2j} \tilde{u}|^2_G \geq 0 $.

Additionally, we will use the transport energy
\begin{equation}
\label{E:E^2k_transport}
E^{2k}_{\text{transport}}(\tilde{s}, \tilde{r}, \tilde{u}) = \norm{\hat{\omega}}_{H^{2k-1, k+ \frac{1}{2(\gamma-1)}}}^2 + \norm{\tilde{s}}_{H^{2k, k+ \frac{1}{2(\gamma-1)}}}^2
\end{equation}
where $\hat{\omega}$ is the reduced linearized vorticity defined in \eqref{E:reduced_lin_vort_evolution_eq}, and this will be addressed in Section \ref{S:Vorticity_estimates}.
Then, the final linearized energy is given by
\begin{equation}
\label{E:E^2k_full_linearized_energy}
 E^{2k}(\tilde{s}, \tilde{r}, \tilde{u}) =  E^{2k}_{\text{wave}}(\tilde{s}, \tilde{r}, \tilde{u}) + E^{2k}_{\text{transport}}(\tilde{s}, \tilde{r}, \tilde{u}).
\end{equation}

\begin{remark}
Although both components of $E^{2k}_{\text{transport}}$ satisfy transport equations, we remark that they will be used slightly differently in the arguments that follow. In particular, the estimate for $\hat{\omega}$ will be used to complete div-curl estimates for $\tilde{u}$, whereas $\tilde{s}$ can be estimated directly with the $\mathcal{H}^{2k}$ norm from above and below as in Section \ref{S:Entropy estimates}. 
\end{remark}

\subsection{The basic energy estimate}
\label{S:Basic_energy_estimate}
In this section, we will prove a basic energy estimate for the quantities $\tilde{s}, \tilde{r},$ and $\tilde{u}$ in weighted $L^2$ spaces. Although the higher order estimates in Sobolev spaces require additional techniques, this section serves as a starting point in our analysis.
As discussed in Section \ref{S:Function_spaces}, the $\tilde{s}$ equation will later be treated separately (see section \ref{S:Entropy estimates}) as it satisfies a transport equation as opposed to the wave equations satisfied by $\tilde{r}$ and $\tilde{u}$ (see Section \ref{S:Higher_order_energy_estimates}).

After computing the linearized equations \eqref{E:Linearized_System_2}, multiply \eqref{E:Linearized_System_2_s_equation} by $ r^{\frac{1}{\gamma-1}} \tilde{s}$, the second equation \eqref{E:Linearized_System_2_r_equation} by $\frac{1}{\gamma-1} r^{\frac{2-\gamma}{\gamma-1}} \tilde{r}$ and contract the third equation \eqref{E:Linearized_System_2_u_equation} with $(\Gamma+r) r^{\frac{1}{\gamma-1}} G_{\alpha \beta} \tilde{u}^\beta$. Using \eqref{E:v_constraint_equation}, \eqref{E:Linearized_System_2_constraint_equation}, and identities like $\tilde{u}_\alpha \proj^{\alpha \mu} = \tilde{u}^\mu = G_{\alpha \beta} \tilde{u}^\beta \proj^{\alpha \mu}$, we get
\begin{subequations}
\begin{align}
    \label{E:Basic_energy_estimate_tilde_s_equation}
    \frac{1}{2} r^{\frac{1}{\gamma-1}}D_t \tilde{s}^2 &= - r^{\frac{1}{\gamma-1}}\tilde{s} \tilde{u}^\mu  \partial_\mu s\\
    \label{E:Basic_energy_estimate_tilde_r_equation}
    \frac{1}{2(\gamma-1)} r^{\frac{2-\gamma}{\gamma-1}} D_t \tilde{r}^2 + \textcolor{blue}{\frac{1}{\gamma-1} r^{\frac{2-\gamma}{\gamma-1}} \tilde{r}\tilde{u}^\mu \partial_\mu r} + \textcolor{blue}{r^{\frac{1}{\gamma-1}} \tilde{r} \partial_\mu \tilde{u}^\mu} &=-\tilde{r}^2 r^{\frac{2-\gamma}{\gamma-1}} \partial_\mu u^\mu\\
    \label{E:Basic_energy_estimate_tilde_u_equation}
    \frac{\Gamma +r}{2} r^{\frac{1}{\gamma-1}} D_t |\tilde{u}|^2 + \textcolor{blue}{ r^{\frac{1}{\gamma-1}} \tilde{u}^\mu \partial_\mu \tilde{r}}
    &= -(\Gamma+r) r^{\frac{1}{\gamma-1}} \tilde{u}_\alpha \tilde{u}^\mu \partial_\mu u^\alpha -r^{\frac{1}{\gamma-1}} |\tilde{u}|^2 u^\mu \partial_\mu r \\
    \notag
    & \hspace{5mm} + \frac{\Gamma'}{\Gamma+r} r^{\frac{1}{\gamma-1}} \tilde{s} \tilde{u}^\mu \partial_\mu r + \frac{1}{\Gamma+r} r^{\frac{1}{\gamma-1}} \tilde{r} \tilde{u}^\mu \partial_\mu r.
\end{align}
\end{subequations}
Thus, after adding the equations together, the \textcolor{blue}{blue} terms combine to form a perfect derivative:
\begin{equation}
    \textcolor{blue}{\frac{1}{\gamma-1} r^{\frac{2-\gamma}{\gamma-1}} \tilde{r} \tilde{u}^\mu \partial_\mu r} + \textcolor{blue}{r^{\frac{1}{\gamma-1}} \tilde{r} \partial_\mu \tilde{u}^\mu} + \textcolor{blue}{ r^{\frac{1}{\gamma-1}} \tilde{u}^\mu \partial_\mu \tilde{r}} = \textcolor{blue}{ \partial_\mu (r^{\frac{1}{\gamma-1}} \tilde{r} \tilde{u}^\mu )}
\end{equation}
which will be handled using integration be parts. Combining the equations together, we get
\begin{equation}
    \begin{split}
         \frac{1}{2} r^{\frac{1}{\gamma-1}}D_t \tilde{s}^2 + \frac{1}{2(\gamma-1)} r^{\frac{2-\gamma}{\gamma-1}} D_t \tilde{r}^2 + \frac{\Gamma+r}{2} r^{\frac{1}{\gamma-1}} D_t |\tilde{u}|^2
         &+ \textcolor{blue}{ \partial_\mu (r^{\frac{1}{\gamma-1}} \tilde{r} \tilde{u}^\mu )} \\
         &=  - r^{\frac{1}{\gamma-1}}\tilde{s} \tilde{u}^\mu  \partial_\mu s \\
         & \hspace{5mm} -\tilde{r}^2 r^{\frac{2-\gamma}{\gamma-1}} \partial_\mu u^\mu \\
         &\hspace{5mm}  -(\Gamma+r) r^{\frac{1}{\gamma-1}} \tilde{u}_\alpha \tilde{u}^\mu \partial_\mu u^\alpha -r^{\frac{1}{\gamma-1}} |\tilde{u}|^2 u^\mu \partial_\mu r \\
         & \hspace{5mm} + \frac{\Gamma'}{\Gamma+r} r^{\frac{1}{\gamma-1}} \tilde{s} \tilde{u}^\mu \partial_\mu r + \frac{1}{\Gamma+r} r^{\frac{1}{\gamma-1}} \tilde{r} \tilde{u}^\mu \partial_\mu r.
    \end{split}
\end{equation}
Next, we integrate with respect to $x$ over $\Omega_t$ and use the moving domain formula (noting that the fluid particles on the boundary move with velocity $\frac{u^i}{u^0}$):
\begin{equation}
\begin{split}
     \label{E:Moving_Domain_formula}
    \frac{d}{dt} \int_{\Omega_t} f \hspace{1mm} dx &= \int_{\Omega_t} \partial_t f \hspace{1mm} dx + \int_{\Omega_t} \partial_i \left(f \frac{u^i}{u^0}\right) \hspace{1mm} dx \\
    &= \int_{\Omega_t} \frac{1}{u^0} D_t f \hspace{1mm} dx + \int_{\Omega_t} f \partial_i \left( \frac{u^i}{u^0}\right) \hspace{1mm} dx .\\
\end{split}
\end{equation}
For the \textcolor{blue}{blue} terms, we have $r=0$ on $\partial \Omega_t$, so integrating by parts yields
\begin{equation}
\begin{split}
     \int_{\Omega_t} \textcolor{blue}{\partial_\mu (r^{\frac{1}{\gamma-1}} \tilde{r}\tilde{u}^\mu)} \hspace{1mm} dx & =\int_{\Omega_t} \partial_i(r^{\frac{1}{\gamma-1}} \tilde{r}\tilde{u}^i) \hspace{1mm} dx + \int_{\Omega_t} \partial_t(r^{\frac{1}{\gamma-1}} \tilde{r}\tilde{u}^0) \hspace{1mm} dx \\
     &=- \int_{\partial \Omega_t} r^{\frac{1}{\gamma-1}} \tilde{r} \tilde{u}^i \nu_i dS + \int_{\Omega_t} \partial_t(r^{\frac{1}{\gamma-1}} \tilde{r}\tilde{u}^0) \hspace{1mm} dx \\
     &= \int_{\Omega_t} \partial_t(r^{\frac{1}{\gamma-1}} \tilde{r}\tilde{u}^0) \hspace{1mm} dx
\end{split}
\end{equation}
For the energy terms, we plan to apply \eqref{E:Moving_Domain_formula} to the function
\begin{equation}
    f(t,x) = \frac{1}{2} r^{\frac{1}{\gamma-1}} \tilde{s}^2 + \frac{1}{2(\gamma-1)} r^{\frac{2-\gamma}{\gamma-1}} \tilde{r}^2 + \frac{\Gamma+r}{2} r^{\frac{1}{\gamma-1}} |\tilde{u}|^2
\end{equation}

First, recall using the $r$ equation that $D_t r = - (\gamma-1) r \partial_\mu u^\mu $. We then get
\begin{equation}
\begin{split}
    D_t f(t,x) &=  \frac{1}{2} r^{\frac{1}{\gamma-1}}D_t \tilde{s}^2 + \frac{1}{2(\gamma-1)} r^{\frac{2-\gamma}{\gamma-1}} D_t \tilde{r}^2 + \frac{\Gamma+r}{2} r^{\frac{1}{\gamma-1}} D_t |\tilde{u}|^2 \\
    & \hspace{5mm} - \frac{1}{2} r^{\frac{1}{\gamma-1}} \partial_\mu u^\mu \tilde{s}^2 - \frac{2-\gamma}{2(\gamma-1)} r^{\frac{2-\gamma}{\gamma-1}} \partial_\mu u^\mu \tilde{r}^2 - \frac{1}{2} r^{\frac{1}{\gamma-1}} (\Gamma + \gamma r)\partial_\mu u^\mu |\tilde{u}|^2 \\
    &= \textcolor{blue}{- \partial_\mu (r^{\frac{1}{\gamma-1}} \tilde{r} \tilde{u}^\mu )}  \\
    & \hspace{5mm}- r^{\frac{1}{\gamma-1}}\tilde{s} \tilde{u}^\mu  \partial_\mu s \\
    & \hspace{5mm} -\tilde{r}^2 r^{\frac{2-\gamma}{\gamma-1}} \partial_\mu u^\mu \\
    &\hspace{5mm}  -(\Gamma+r) r^{\frac{1}{\gamma-1}} \tilde{u}_\alpha \tilde{u}^\mu \partial_\mu u^\alpha -r^{\frac{1}{\gamma-1}} |\tilde{u}|^2 u^\mu \partial_\mu r \\
    & \hspace{5mm} + \frac{\Gamma'}{\Gamma+r} r^{\frac{1}{\gamma-1}} \tilde{s} \tilde{u}^\mu \partial_\mu r + \frac{1}{\Gamma+r} r^{\frac{1}{\gamma-1}} \tilde{r} \tilde{u}^\mu \partial_\mu r\\
     & \hspace{5mm} - \frac{1}{2} r^{\frac{1}{\gamma-1}} \partial_\mu u^\mu \tilde{s}^2 - \frac{2-\gamma}{2(\gamma-1)} r^{\frac{2-\gamma}{\gamma-1}}  \partial_\mu u^\mu \tilde{r}^2 - \frac{1}{2} r^{\frac{1}{\gamma-1}} (\Gamma + \gamma r)\partial_\mu u^\mu |\tilde{u}|^2 \\
\end{split}
\end{equation}
using the fact that $D_t \Gamma =0$ and computing several more expressions.

After gathering terms, we apply \eqref{E:Moving_Domain_formula} to $f(t,x)$ and integrate over $\Omega_t$. This yields
\begin{equation}
\label{E:Moving_domain_computation}
\begin{split}
     \frac{d}{dt} \int_{\Omega_t} f(t,x) \hspace{1mm} dx &= \int_{\Omega_t} f(t,x) \partial_i \left(\frac{u^i}{u^0} \right) \mss dx \\
     &\hspace{5mm} - \int_{\Omega_t} \frac{1}{u^0}\left( \frac{1}{2} r^{\frac{1}{\gamma-1}} \tilde{s}^2 + \frac{\gamma}{2(\gamma-1)} r^{\frac{2-\gamma}{\gamma-1}} \tilde{r}^2 + \frac{1}{2}r^{\frac{1}{\gamma-1}} (\Gamma + \gamma r)|\tilde{u}|^2 \right) \partial_\mu u^\mu \mss dx \\
     & \hspace{5mm} - \int_{\Omega_t} \frac{1}{u^0} \textcolor{blue}{ \partial_t (r^{\frac{1}{\gamma-1}} \tilde{r} \tilde{u}^0 )} \mss dx \\
     & \hspace{5mm}- \int_{\Omega_t} \frac{1}{u^0} r^{\frac{1}{\gamma-1}}\tilde{s} \tilde{u}^\mu  \partial_\mu s \mss dx\\
     & \hspace{5mm}- \int_{\Omega_t} \frac{1}{u^0} (\Gamma+r) r^{\frac{1}{\gamma-1}} \tilde{u}_\alpha \tilde{u}^\mu \partial_\mu u^\alpha + \frac{1}{u^0}  r^{\frac{1}{\gamma-1}} |\tilde{u}|^2 u^\mu \partial_\mu r \mss dx \\
    &\hspace{5mm} + \int_{\Omega_t} \frac{1}{u^0} \frac{\Gamma'}{\Gamma+r} r^{\frac{1}{\gamma-1}} \tilde{s} \tilde{u}^\mu \partial_\mu r +  \frac{1}{u^0} \frac{1}{\Gamma+r} r^{\frac{1}{\gamma-1}} \tilde{r} \tilde{u}^\mu \partial_\mu r dx \\
\end{split}
\end{equation}

Thus, we define the squared base energy ($k=0$) to be 
\begin{equation}
\label{E:Energy_expression}
     E^0(\tilde{s}, \tilde{r},\tilde{u})[t] := \norm{ (\tilde{r}, \tilde{u})}_{\tilde{\mathcal{H}}}^2 + \frac{1}{2} \norm{\tilde{s}}_{H^{0, \frac{1}{2(\gamma-1)}}}^2  = \frac{1}{2} \int_{\Omega_t} r^{\frac{2-\gamma}{\gamma-1}} \left(\frac{1}{\gamma-1} \tilde{r}^2 + (\Gamma+r) r |\tilde{u}|^2 + r \tilde{s}^2 \right) \mss dx.
\end{equation}

\begin{remark}
Observe the correspondence between \eqref{E:Energy_expression} and the higher order energy \eqref{E:E^2k_full_linearized_energy} where we have a wave part and a transport part. We only make the distinction here between $\tilde{s}$ and $(\tilde{r}, \tilde{u})$ to highlight the differences in the higher order estimates. In Section \ref{S:Entropy estimates}, we can simply differentiate the $\tilde{s}$ equation with spatial derivatives $\partial^{2k}$. However, in Section \ref{S:Higher_order_energy_estimates}, we must differentiate the $\tilde{r}$ and $\tilde{u}$ equations with the convective derivative $D_t$.
\end{remark}

This will allow us to prove the following basic energy estimate.

\begin{proposition}[Basic Energy Inequality]
\label{P:Basic_Energy_Inequality}
Let $(s,r,u)$ be a solution to \eqref{E:System} that exists on some time interval $[0,T]$. Assume that $s,r,u$ and their first order derivatives are bounded in the $L^\infty(\Omega_t)$ norm for each $t \in [0,T]$ and $r$ vanishes simply on the free boundary. Then, the following estimate holds for solutions $(\tilde{s}, \tilde{r}, \tilde{u})$ to \eqref{E:Linearized_System_2}:
\begin{equation}
    E^0[t] \lesssim E^0[0] \exp \left(\int_0^t \mathcal{C}(\norm{\partial s, \partial r, \partial u, s, r, u}_{L^\infty(\Omega_\tau)})\mss d\tau \right)
\end{equation}
where $E^0[t]$ is given by \eqref{E:Energy_expression} and $\mathcal{C}$ is a function depending on the $L^\infty(\Omega_\tau)$ norms of $s,r,u$ and their first order derivatives.

\end{proposition}

Before giving the proof, we provide the following remarks which are critical in understanding the methods used throughout this paper.

\begin{remark}[The symbol `$\lesssim$']
Throughout our work, the symbol $\lesssim$ will be used in the usual fashion, i.e.
\begin{equation*}
    A \lesssim B \Longleftrightarrow A \leq \mathcal{D} B
\end{equation*}
where $\mathcal{D}$ is a constant depending on the the fixed data of the problem. In our case, $\mathcal{D}$ will depend on the background solution $(s, r, u)$ along with $\gamma$ and $T$.
    
\end{remark}

\begin{remark}
\label{R:O(1)_weights}
    In view of \eqref{E:weight_fact_1}, we will often use inequalities of the following form when handling weights that appear in our integrals: 
\begin{equation}
\label{E:weight_argument}
    \int_{\Omega_t} r^{\frac{1}{\gamma-1}} |\tilde{u}|^2 \mss dx \leq \nnorm{\frac{1}{\Gamma+r}}_{L^\infty(\Omega_t)} \int_{\Omega_t} (\Gamma+r) r^{\frac{1}{\gamma-1}} |\tilde{u}|^2 \mss dx \lesssim E^0
\end{equation}
where we note that $\frac{1}{\Gamma+r}$ tends to a finite positive constant depending on $\gamma$ when we are sufficiently close to the boundary (see Remark \ref{R:weights_in_wave_energy_tilde_H_norm}). 
When we refer to an expression consisting of $s,r,u$ and $\gamma$ as $O(1)$ near the free boundary, we simply are referring to the fact that when $r$ is sufficiently small, this weight tends to a finite constant bounded away from zero. Thus, it can be absorbed into the $\lesssim$ symbol without having much effect on the primary estimate.
\end{remark}

\begin{remark}
\label{R:estimating_time_derivatives}
    Additionally, there are many situations in which we would like to estimate expressions involving \textit{spacetime} derivatives $\partial_\mu$ over the \textit{spatial} region $\Omega_t$. However, this can be done with the help of section \ref{S:Sovling_for_time_derivatives}. The primary idea is that we can use system \eqref{E:System} directly to solve for $\partial_t(s,r,u)$ and $\partial_t(\tilde{s}, \tilde{r}, \tilde{u})$ in terms of the spatial part $\partial(s,r,u)$ and $\partial(\tilde{s},\tilde{r},\tilde{u})$ respectively, plus additional terms that can be estimated. A key identity is 
\begin{equation}
\label{E:tilde_u_^0 identity}
    \tilde{u}^0 = \frac{u_j}{u^0} \tilde{u}^j
\end{equation}
which follows from the orthogonality of $u$ and $\tilde{u}$ in \eqref{E:Linearized_System_2_constraint_equation}.  
    With section \ref{S:Sovling_for_time_derivatives} taken for granted, one can treat $\partial_\mu$ as a generic first order spatial derivative $\partial$ in the following proof with the understanding that all time derivatives will be eventually written in terms of purely spatial derivatives.
\end{remark}

Here, we will provide the proof of Proposition \ref{P:Basic_Energy_Inequality}.

\begin{proof}
Using Remarks \ref{R:O(1)_weights} and \ref{R:estimating_time_derivatives}, the first two terms in \eqref{E:Moving_domain_computation} on the RHS containing $\partial_\mu u^\mu$ and $\partial_i \left(\frac{u^i}{u^0} \right)$ can be quite easily estimated using $E^0$ and an expression depending on the $L^\infty(\Omega_t)$ norms of $u$ and $\partial u$. We arrive at the inequality
\begin{equation}
    \begin{split}
        \frac{d}{dt} E^0 &\lesssim C_0(\norm{\partial u}_{L^\infty(\Omega_t)},\norm{ u}_{L^\infty(\Omega_t)} ) E^0 \\
     & \hspace{5mm} - \int_{\Omega_t} \frac{1}{u^0} \textcolor{blue}{ \partial_t (r^{\frac{1}{\gamma-1}} \tilde{r} \tilde{u}^0 )} \mss dx \\
     & \hspace{5mm}- \int_{\Omega_t} \frac{1}{u^0} r^{\frac{1}{\gamma-1}}\tilde{s} \tilde{u}^\mu  \partial_\mu s \mss dx\\
     & \hspace{5mm}- \int_{\Omega_t} \frac{1}{u^0}(\Gamma+r) r^{\frac{1}{\gamma-1}} \tilde{u}_\alpha \tilde{u}^\mu \partial_\mu u^\alpha + \frac{1}{u^0}r^{\frac{1}{\gamma-1}} |\tilde{u}|^2 u^\mu \partial_\mu r \mss dx \\
    &\hspace{5mm} + \int_{\Omega_t} \frac{1}{u^0}\frac{\Gamma'}{\Gamma+r} r^{\frac{1}{\gamma-1}} \tilde{s} \tilde{u}^\mu \partial_\mu r + \frac{1}{u^0}\frac{1}{\Gamma+r} r^{\frac{1}{\gamma-1}} \tilde{r} \tilde{u}^\mu \partial_\mu r dx \\
    &:= \norm{(\partial u, u)}_{L^\infty(\Omega_t)}E^0 + I_1+ I_2 + I_3 + I_4
    \end{split}
\end{equation}
where each lower order term $I_1, ..., I_4$ will be estimated in the following way.

The first term $I_1$ is unique in that we plan to combine it with the LHS of our inequality after integrating in time. We leave this term for now and will return to it later.

For $I_2$, we estimate it using the Cauchy Schwartz inequality. We have
\begin{equation}
\begin{split}
\label{E:I_2_term_estimate}
     |I_2| \leq \int_{\Omega_t} r^{\frac{1}{\gamma-1}} \left|(u^0)^{-1}\right| |\tilde{s} \tilde{u}^\mu \partial_\mu s| &\leq \norm{\partial s}_{L^\infty(\Omega_t)} \norm{(u^0)^{-1}}_{L^\infty(\Omega_\tau)} \left(\int_{\Omega_t} r^{\frac{1}{\gamma-1}} |\tilde{s}|^2 \right)^{\frac{1}{2}}  \left(\int_{\Omega_t} r^{\frac{1}{\gamma-1}} | \tilde{u}|^2 \right)^{\frac{1}{2}} \\
     &\lesssim C_2(\norm{\partial s}_{L^{\infty}(\Omega_t)}, \norm{u}_{L^\infty(\Omega_\tau)}) E^0.
\end{split}
\end{equation}
where we used remarks \ref{R:O(1)_weights}-\ref{R:estimating_time_derivatives}, and the fact that $G$ is a Riemannian metric, hence $|\tilde{u}|^2_{\delta} \lesssim |\tilde{u}|_G^2$ where $\delta$ is the Euclidean metric in spacetime.

For $I_3$, we simplify $D_t r= -(\gamma-1)r \partial_\mu u^\mu$ and get a similar inequality
\begin{equation}
\label{E:I_3_term_estimate}
\begin{split}
     |I_3| &\leq \int_{\Omega_t} (\Gamma+r) r^{\frac{1}{\gamma-1}} \left|(u^0)^{-1}\right|  |\tilde{u}|^2 |\partial_\mu u^\mu| \mss dx \hspace{1mm} + \int_{\Omega_t}  r^{\frac{1}{\gamma-1}} \left|(u^0)^{-1}\right| |\tilde{u}|^2 (\gamma-1)r |\partial_\mu u^\mu| \mss dx \\
     &\lesssim C_3(\norm{\partial u}_{L^\infty(\Omega_t)},\norm{u}_{L^\infty(\Omega_\tau)}) E^0
\end{split}
\end{equation}
where we used Remark \ref{R:O(1)_weights}. 

For $I_4$, we observe that $\Gamma'(s) = - \frac{\gamma-1}{\gamma} \Gamma(s)$ using \eqref{E:Gamma_definition}, and we can apply similar arguments along with the Cauchy-Schwarz inequality to get

\begin{equation}
\label{E:I_4_term_estimate}
\begin{split}
     |I_4| &\leq \int_{\Omega_t} \left(\frac{\gamma-1}{\gamma} \right) \left|(u^0)^{-1}\right| \frac{1}{\Gamma+r} r^{\frac{1}{\gamma-1}} |\tilde{s} \tilde{u}^\mu \partial_\mu r| \mss dx +  \int_{\Omega_t}  \frac{1}{\Gamma+r} r^{\frac{1}{\gamma-1}} \left|(u^0)^{-1}\right||\tilde{r} \tilde{u}^\mu \partial_\mu r| \mss dx \\
     &\lesssim C_4(\norm{\partial r}_{L^\infty(\Omega_t)}, \norm{u}_{L^\infty(\Omega_\tau)}) E^0.
\end{split}
\end{equation}

Thus, for the energy defined in \eqref{E:Energy_expression} with $t$ replaced by $\tau$, we combine \eqref{E:I_2_term_estimate}-\eqref{E:I_4_term_estimate} and $I_1$ to get the inequality
\begin{equation}
    \frac{d}{d\tau} E^0 \lesssim C(\norm{\partial s}, \norm{\partial r}, \norm{\partial u}, \norm{u}) E + \norm{u} \int_{\Omega_\tau} \textcolor{blue}{ \partial_\tau (r^{\frac{1}{\gamma-1}} \tilde{r} \tilde{u}^0 )} \mss dx 
\end{equation}
where the constant depends only on $\gamma$ and the norms are on $L^\infty(\Omega_\tau)$.
Integrating in time from $0$ to $t \in [0,T]$, this implies 
\begin{equation}
\begin{split}
     E^0[t] & \lesssim E^0[0] + \int_0^t C(\norm{\partial s}, \norm{\partial r}, \norm{\partial u}, \norm{u}) E^0[\tau] \mss d\tau +\textcolor{blue}{\int_0^t  \int_{\Omega_{\tau}} \partial_\tau (r^{\frac{1}{\gamma-1}} \tilde{r} \tilde{u}^0 ) \mss dx d\tau} \\
\end{split}
\end{equation}
where the constant depends on $\gamma$ and the $L^\infty$ norm of $u$ which is bounded on $[0,T]$.
Then, from a quick computation using the moving domain formula, we get
\begin{equation}
\begin{split}
     \textcolor{blue}{\int_0^t \int_{\Omega_{\tau}} \partial_\tau (r^{\frac{1}{\gamma-1}} \tilde{r} \tilde{u}^0 ) \mss dx d\tau} &= \int_0^t \left(\frac{d}{d \tau} \int_{\Omega_\tau}r^{\frac{1}{\gamma-1}} \tilde{r} \tilde{u}^0  \mss dx d\tau - \int_{\Omega_\tau} \partial_i (r^{\frac{1}{\gamma-1}} \tilde{r} \tilde{u}^0 \frac{u^i}{u^0}) \mss dx d\tau \right)   \\
     &= \int_0^t \frac{d}{d \tau} \int_{\Omega_\tau}r^{\frac{1}{\gamma-1}} \tilde{r} \tilde{u}^0  \mss dx d\tau \\
     &=  \int_{\Omega_t} (r(t,x))^{\frac{1}{\gamma-1}} \tilde{r}(t,x) \tilde{u}^0(t,x) \mss dx - \int_{\Omega_0} (r(0,x))^{\frac{1}{\gamma-1}} \tilde{r}(0,x) \tilde{u}^0(0,x) \mss dx
\end{split}
\end{equation}
where we integrate by parts for one of the terms and used the fact that $r$ vanishes on the free boundary. Thus, we can substitute this into our inequality to get 
\begin{equation}
    \begin{split}
        E^0[t] & \lesssim E^0[0] + \int_0^t C(\norm{\partial s}, \norm{\partial r}, \norm{\partial u}, \norm{u} ) E^0[\tau] \mss d\tau \\
     & \hspace{8mm}+\textcolor{blue}{ \int_{\Omega_t} (r(t,x))^{\frac{1}{\gamma-1}} \tilde{r}(t,x) \tilde{u}^0(t,x) \mss dx - \int_{\Omega_0} (r(0,x))^{\frac{1}{\gamma-1}} \tilde{r}(0,x) \tilde{u}^0(0,x) \mss dx }
    \end{split}
\end{equation}
where the constant depends on $\gamma$ and $\norm{u}_{L^\infty(\Omega_\tau)}$ which is bounded.

Then, for $x$ sufficiently close to the boundary, we have estimates of the form
\begin{equation}
\begin{split}
      \left|\textcolor{blue}{ \int_{\Omega_t} r^{\frac{1}{\gamma-1}} \tilde{r}\tilde{u}^0 \mss dx } \right| &\leq  \left(\int_{\Omega_t} r^{\frac{1}{\gamma-1}} |\tilde{r}|^2 \right)^{\frac{1}{2}}  \left(\int_{\Omega_t} r^{\frac{1}{\gamma-1}} | \tilde{u}|^2 \right)^{\frac{1}{2}} \\
      & \lesssim \norm{r}_{L^\infty(\Omega_\tau) }^{1/2} \left(\int_{\Omega_t} r^{\frac{2-\gamma}{\gamma-1}} |\tilde{r}|^2 \right)^{\frac{1}{2}}  \left(\int_{\Omega_t} r^{\frac{1}{\gamma-1}} | \tilde{u}|^2 \right)^{\frac{1}{2}} \\
      &\lesssim \hat{\varepsilon}^{1/2} E^0[t]
\end{split}
\end{equation}
where we applied Assumption \ref{A:Assumption_smallness r} and used the smallness of $r$ so that $\hat{\varepsilon}^{1/2} \ll \frac{1}{\sqrt{2}} < 1$.  A similar estimate also holds for the $\Omega_0$ term. Thus, we can move the $\hat{\varepsilon}^{1/2} E^0[t]$ term to the LHS which will yield
\begin{equation}
    E^0[t] \lesssim \frac{1}{1- \hat{\varepsilon}^{1/2}} \left(E^0[0](1+\hat{\varepsilon}^{1/2}) + \int_0^t (\norm{\partial s}_{L^\infty(\Omega_\tau)}+ \norm{\partial r}_{L^\infty(\Omega_\tau)} + \norm{\partial u}_{L^\infty(\Omega_\tau)} ) E^0[\tau] \mss d\tau \right) 
\end{equation}

Thus, by Gronwall's inequality, we have 
\begin{equation}
    E^0[t] \lesssim E^0[0] \exp \left(\int_0^t \mathcal{C}(\norm{\partial s},\norm{\partial r},\norm{\partial u}, \norm{s}, \norm{r}, \norm{u}) \mss d\tau \right)
\end{equation}
with a constant depending only on $\gamma$ and our distance to the free boundary. This completes our proof of the basic energy estimate.

\end{proof}

\section{Creating a Book-keeping scheme}
\label{S:Creating the Book-keeping scheme}
In works such as \cite{DisconziIfrimTataru}, a book-keeping scheme was created using the equations' scaling law for the leading order dynamics near the free boundary. However, no such scaling seems available for \eqref{E:Linearized_System_2}; thus, we will need to develop several key ideas around the notion of \textbf{order} which is first defined in Definition \ref{D:definition_critical_subcritical_supercritical}. After differentiating the equations, we plan to create a book-keeping scheme which accounts for the number of derivatives and the powers of our free boundary weight $r$. Along the way, any derivatives of the background quantities $s, r, u$ will serve as $L^\infty$ coefficients for each of the terms that we will encounter.  

Using system \eqref{E:Linearized_System_2}, we record the computations when taking $D_t^{2k}$ of the equations (where $D_t^0$ is the identity). Since the $\tilde{s}$ equation will be treated separately in Section \ref{S:Entropy estimates}, we put only the $\tilde{r}$ and $\tilde{u}$ equations here for convenience. Writing the equations in commutator notation and keeping only the higher order terms on the LHS, we get the following system:

\begin{subequations}
\label{E:Linearized_System_6}
\begin{align}
\label{E:Linearized_System_6_r_equation}
    &D_t \left(D_t^{2k}  \tilde{r}\right) + \left(D_t^{2k} \tilde{u}^\mu \right)\partial_\mu  r + (\gamma-1) r  \partial_\mu \left(D_t^{2k}  \tilde{u}^\mu \right)= B_{2k} \\
\label{E:Linearized_System_6_v_equation}
    &D_t \left(D_t^{2k} \tilde{u}^\alpha \right)+  \frac{1}{\Gamma+r} \proj^{\alpha \mu} \partial_\mu \left(D_t^{2k} \tilde{r}\right) = C^{\alpha}_{2k}
\end{align}
\end{subequations}
where
\begin{equation}
\label{E:G_and_H_definition}
    \begin{split}
        B_{2k} &= D_t^{2k} g -\sum_{i=0}^{2k-1} \left(D_t^i \tilde{u}^\mu \partial_\mu \left(D_t^{2k-i} r \right)- D_t^{2k-i} r  \partial_\mu \left(D_t^i  \tilde{u}^\mu \right) \right)   \\
        & \hspace{4cm}-  \sum_{i=0}^{2k-1} D_t^i \tilde{u}^\mu [D_t^{2k-i}, \partial_\mu ]r - (\gamma-1) \sum_{i=1}^{2k} D_t^{2k-i} r [D_t^i, \partial_\mu] \tilde{u}^\mu\\
        C^\alpha_{2k} &=  D_t^{2k} h^\alpha -\sum_{i=0}^{2k-1} D_t^{2k-i} \left(\frac{1}{\Gamma+r} \proj^{\alpha \mu} \right) \partial_\mu \left(D_t^i \tilde{r}\right) -\sum_{i=1}^{2k} D_t^{2k-i} \left(\frac{1}{\Gamma+r} \proj^{\alpha \mu} \right) [D_t^i,\partial_\mu ] \tilde{r}\\
    \end{split}
\end{equation}
and $g$ and $h^\alpha$ are defined in \eqref{E:Linearized_f_g_h}. Additionally, we removed terms from the summation where the commuator was zero. Each of the terms in \eqref{E:G_and_H_definition} will be included in the higher order wave equation estimates in Section \ref{S:Higher_order_energy_estimates}.

Using several lemmas from \cite{DisconziIfrimTataru} on embedding theorems combined with Assumption \ref{A:Assumption_smallness r}, we are able to prove the following lemmas which motivate our definition of \textbf{order} defined in Remark \ref{R:order_free_boundary_term}.

\begin{lemma}
\label{L:Free_boundary_add_r}
Suppose that $\sigma \geq 0$, $r$ vanishes simply on the free boundary, and $f \in H^{k,\sigma}$. Then $f \in H^{k, \sigma + \frac{1}{2}}$, and more specifically
\begin{equation*}
     \norm{f}_{H^{k, \sigma + \frac{1}{2}}(\Omega_t)}^2 \leq \hat{\varepsilon}  \norm{f}_{H^{k, \sigma}(\Omega_t)}^2
\end{equation*}
where the integrals are interpreted with respect to Remark \ref{R:Omega_t_near_free_boundary}.

\end{lemma}

\begin{lemma}[Free Boundary Embedding Lemma]
\label{L:free_boundary_embedding_lemma}
Let $H^{l,m}$ be the corresponding Sobolev space for the free-boundary problem with weight $r$ that vanishes simply near the free boundary, and $l,m \geq 0$. If $k \geq 0$ is provided, then
\begin{equation*}
    H^{2k , k+ \frac{1}{2}} \subset H^{l,m}
\end{equation*}
provided $ l \leq 2k$ and $m -l + k - \frac{1}{2} \geq 0$.

\end{lemma}
\begin{remark}
By following the proof, we also have that natural corollary that $H^{2k, k} \subset H^{l,m}$ provided $m- l + k \geq 0$.
\end{remark}


By combining the above lemmas we have the following useful corollary:
\begin{corollary}[Free-Boundary trading derivatives for weight]
\label{C:Trading_deriv_for_weight}
Suppose that $f \in H^{j+1, \sigma+1}$ and $j, \sigma \geq 0$. Then, we have 
\begin{equation}
    \norm{f}_{H^{j ,\sigma}} \lesssim \hat{\varepsilon} \norm{f}_{H^{j+1,\sigma}}
\end{equation}
where $\hat{\varepsilon}$ is the small positive constant defined in Assumption \ref{A:Assumption_smallness r}.
\end{corollary}
\begin{proof}
From \cite{DisconziIfrimTataru}, we see that $ \norm{f}_{H^{j ,\sigma}} \lesssim \norm{f}_{H^{j+1 ,\sigma+1}}$ with a constant depending on $\sigma$. Then, we can simply pull out a power of $r$ in the $L^\infty(\Omega_t)$ norm and use Assumption \ref{A:Assumption_smallness r}.
\end{proof}


\subsection{Defining order for free-boundary terms}

\begin{definition}[$\mathcal{H}^{2k}$-critical, subcritical, and supercritical terms]
\label{D:definition_critical_subcritical_supercritical}
Let $m,l \in \mathbb{N}$. 
In view of Lemma \ref{L:free_boundary_embedding_lemma}, any terms of the form 
\begin{equation*}
    r^m \partial^l \tilde{u},   r^m \partial^l \tilde{s}
\end{equation*}
or
\begin{equation*}
    r^m \partial^l \tilde{r}
\end{equation*}
with $0\leq l \leq 2k$ and $m \geq 0$ will be called \textbf{$\mathcal{H}^{2k}$-critical} provided that $m- l +k - \frac{1}{2} =0$, or respectively $m-l + k =0$. Additionally, terms such that $m - l + k - \frac{1}{2} >0$ or respectively $m- l +k >0$ will be called \textbf{$\mathcal{H}^{2k}$-subcritical}. Similarly, terms where $m-l +k - \frac{1}{2} < 0$, $m-l +k < 0$ will be denoted \textbf{$\mathcal{H}^{2k}$-supercritical}. In view of subsection \ref{S:Sovling_for_time_derivatives}, a similar classification can also be used for the terms $r^m \bm{\partial}^l \tilde{u},   r^m \bm{\partial^l} \tilde{s}$, and $r^m \bm{\partial}^l \tilde{r}$.

\begin{remark}
\label{R:order_free_boundary_term}
We can also refer to the $\mathcal{H}^{2k}$-\textbf{order} of a given term by computing the value $\mathcal{O} := m - l +k - \frac{1}{2}$ (or $\mathcal{O}:= m- l +k$ for terms involving $\tilde{r}$) and the sign of $\mathcal{O}$ will determine if the term is subcritical or supercritical. As an example, a term of the form $r^k \partial^{2k} \tilde{u}$ will be called $\mathcal{H}^{2k}$-supercritical of order $-\frac{1}{2}$. The value $\mathcal{O}$ can be interpreted as a indicator for how well that term can be estimated using the $\mathcal{H}^{2k}$ norm. Critical terms have the exact number of derivatives and powers of $r$ to be estimated, whereas supercrtical terms are lacking key additional powers of $r$. Also, note that the order $\mathcal{O}$ depends on a particular derivative level, i.e. what value of $k$ we are using in order to estimate terms using the $\mathcal{H}^{2k}$ norm. Also, note that that the weights in the $\mathcal{H}^{2k}$ norm depend not only on $k$, but also on $\gamma>1$ which is a fixed constant. For clarity, one can interpret this notion of \textbf{order} as in the simplified setting where $\gamma=2$ since our starting Lemma \ref{L:free_boundary_embedding_lemma} is built around the function spaces $H^{2k, k+\frac{1}{2}}$ and $H^{2k, k}$. In that case, the $\mathcal{H}^{2k}$-\textbf{order} of a term $T$ can be interpreted as a statement about the powers of $r$ present/needed in the term $T$ in order for the inequality
\begin{equation*}
    \norm{T}_{L^2}^2 \lesssim \norm{(\tilde{s}, \tilde{r}, \tilde{u})}_{\mathcal{H}^{2k}}^2 \approx \norm{\tilde{s}}_{H^{2k, k+\frac{1}{2}}}^2 +  \norm{\tilde{r}}_{H^{2k, k}}^2 + \norm{\tilde{u}}_{H^{2k, k+\frac{1}{2}}}^2
\end{equation*}
to hold, and the symbol $\approx$ is used to represent the norm equivalence from Remark \ref{R:H^2k_norm_equivalence}. For arbitrary $\gamma>1$, We will instead use the function spaces $H^{2k,k+ \frac{1}{2} + \frac{2-\gamma}{2(\gamma-1)}}$ and $H^{2k, k + \frac{2-\gamma}{2(\gamma-1)}}$. In this setting, the $\mathcal{H}^{2k}$-order for a term $T$ is a statement about the inequality
\begin{equation}
\label{E:Remark_H2k_order}
    \norm{T}_{L^2(r^{\frac{2-\gamma}{\gamma-1}})}^2 \lesssim \norm{(\tilde{s}, \tilde{r}, \tilde{u})}_{\mathcal{H}^{2k}}^2 \approx \norm{\tilde{s}}_{H^{2k, \frac{2-\gamma}{2(\gamma-1)}+k+\frac{1}{2}}}^2 +  \norm{\tilde{r}}_{H^{2k, \frac{2-\gamma}{2(\gamma-1)}+ k}}^2 + \norm{\tilde{u}}_{H^{2k, \frac{2-\gamma}{2(\gamma-1)}+k+\frac{1}{2}}}^2
\end{equation}
which we summarize:
\begin{itemize}
    \item Suppose that $T$ is $\mathcal{H}^{2k}$ subcritical. We claim that 
    inequality \eqref{E:Remark_H2k_order} holds as well as 
    \begin{equation}
    \label{E:Remark_H2k_order_2}
        \norm{T}_{L^2(r^{\frac{2-\gamma}{\gamma-1}})}^2 \lesssim \hat{\varepsilon} \norm{(\tilde{s}, \tilde{r}, \tilde{u})}_{\mathcal{H}^{2k}}^2 
    \end{equation}
    using the previous Lemmas (also, note that \eqref{E:Remark_H2k_order_2} implies \eqref{E:Remark_H2k_order}). To see this fact, let's assume that $T$ consists of $\tilde{r}$ terms so that $T=r^m \partial^l \tilde{r}$ (although $\tilde{u}$ and $\tilde{s}$ terms are similar). By Definition \ref{D:definition_critical_subcritical_supercritical}, we have that $m-l+k > 0$. If $l= 2k$, then we must have $m>k$ (extra power of $r$) and we can apply Lemma \ref{L:Free_boundary_add_r} to get the desired inequality \eqref{E:Remark_H2k_order_2} with $\hat{\varepsilon}$. If $l < 2k$, i.e. $l = 2k-b$ with $0< b \leq 2k$, then we must have $m>k-b$ (and $m=0$ for $b>k$). We would get 
    \begin{equation}
        \norm{r^m \partial^{2k-b} \tilde{r} }_{H^{0,\frac{2-\gamma}{2(\gamma-1)}}}^2 \lesssim \norm{\tilde{r}}_{H^{2k-b, m + \frac{2-\gamma}{2(\gamma-1)}}} \leq \norm{\tilde{r}}_{H^{2k, m+b + \frac{2-\gamma}{2(\gamma-1)}}}.
    \end{equation}
    Then, since $m+b >k$ we can apply Lemma \ref{L:Free_boundary_add_r} to finish the proof of \eqref{E:Remark_H2k_order_2}.    
    \item If $T$ is $\mathcal{H}^{2k}$ critical, then \eqref{E:Remark_H2k_order} holds by Lemma \ref{L:free_boundary_embedding_lemma}, but \eqref{E:Remark_H2k_order_2} in general does not since we do not have any additional powers of $r$ to spare.
    \item If $T$ is $\mathcal{H}^{2k}$ supercritical with order $\mathcal{O}<0$, then \eqref{E:Remark_H2k_order} does not hold (and subsequently \eqref{E:Remark_H2k_order_2} also does not), but \eqref{E:Remark_H2k_order} is true with $T$ replaced by $\hat{T} = r^{-\mathcal{O}} T$ as $T$ is lacking key powers of $r$. To see this fact, we simply observe that the order of the term $\hat{T}$ will be $(m-\mathcal{O}) + l -k = 0$ which is critical.
\end{itemize}

\end{remark}

\begin{lemma}[Sum of $\mathcal{H}^{2k}$ orders]
\label{L:Sum_orders}
Suppose that $T_1$ and $T_2$ are free boundary terms (i.e. $T_1$ and $T_2$ are of the form $r^m \partial^l (\tilde{u}, \tilde{r}, \tilde{s})$ for some $m\geq 0$ and $0 \leq l \leq 2k$) that have $\mathcal{H}^{2k}$ orders $\mathcal{O}_1$ and $\mathcal{O}_2$ respectively. If $\mathcal{O}_1 + \mathcal{O}_2 \geq 0$, then $\norm{T_1 T_2}_{L_1} \lesssim \norm{(\tilde{s}, \tilde{r}, \tilde{u})}_{\mathcal{H}^{2k}}^2$, i.e the $\mathcal{H}^{2k}$ order of a product is the sum of its orders.

\begin{proof}

Without loss of generality, the proof can be obtained by splitting into cases:
\begin{enumerate}
    \item $T_1 = r^{m_1} \partial^{l_1} \tilde{u}$ and $T_2 = r^{m_2} \partial^{l_2} \tilde{u}$, \hspace{5mm} $m_1,m_2 \geq 0$, $0 \leq l_1, l_2 \leq 2k$
    \item $T_1 = r^{m_1} \partial^{l_1} \tilde{r}$ and $T_2 = r^{m_2} \partial^{l_2} \tilde{r}$, \hspace{5mm}   $m_1,m_2 \geq 0$, $0 \leq l_1, l_2 \leq 2k$
    \item $T_1 = r^{m_1} \partial^{l_1} \tilde{r}$ and $T_2 = r^{m_2} \partial^{l_2} \tilde{u}$, \hspace{5mm}   $m_1,m_2 \geq 0$, $0 \leq l_1, l_2 \leq 2k$
\end{enumerate}
and correctly distributing the powers of $r$ as needed.

\end{proof}
\end{lemma}

\begin{remark}
\label{R:Sum_orders}
We note that Lemma \ref{L:Sum_orders} is especially useful in energy estimates where we would like to compute the order of a product of two terms, but we are estimating in $L^1$ instead of $L^2$ and utilizing the Cauchy-Schwarz inequality. Using Lemma \ref{L:Sum_orders}, we can extend the notion of order to a product of terms $T_1 T_2$ by replacing inequalities \eqref{E:Remark_H2k_order} and \eqref{E:Remark_H2k_order_2} with 
    \begin{equation}
        \norm{T_1 T_2}_{L^1(r^{\frac{2-\gamma}{\gamma-1}})} \lesssim \norm{(\tilde{s}, \tilde{r}, \tilde{u})}_{\mathcal{H}^{2k}}^2
    \end{equation}
    \begin{equation}
        \norm{T_1 T_2}_{L^1(r^{\frac{2-\gamma}{\gamma-1}})} \lesssim \hat{\varepsilon} \norm{(\tilde{s}, \tilde{r}, \tilde{u})}_{\mathcal{H}^{2k}}^2
    \end{equation}
    With this new definition for products, Lemma \ref{L:Sum_orders} states that if the $T_1$ has $\mathcal{H}^{2k}$-order $\mathcal{O}_1$ and $T_2$ has $\mathcal{H}^{2k}$-order $\mathcal{O}_2$, then $T_1 T_2$ has $\mathcal{H}^{2k}$-order $\mathcal{O}_1 + \mathcal{O}_2$ provided that $\mathcal{O}_1 + \mathcal{O}_2 \geq 0$. When referring to the order of products $T_1 T_2$, we will always compute the order of $T_1$ and $T_2$ separately before using Lemma \ref{L:Sum_orders} to make a statement about the order of $T_1 T_2$.
\end{remark}

\end{definition}

\begin{remark}[Definition of the $\simeq$ symbol]
\label{R:simeq_symbol_definition}
    Now that we have described the notion of order, we will often use the symbol `$\simeq$' to mean the following. Suppose that $\tilde{A}(\tilde{u}, \tilde{r},\tilde{s})$ and $\tilde{B}(\tilde{u}, \tilde{r},\tilde{s})$ are expressions involving powers of $r$ and derivatives $(\partial, \bm{\partial}, D_t)$ of our linearized variables. Using Remark \ref{R:order_free_boundary_term}, suppose that $\tilde{A}$ and $\tilde{B}$ both contain terms of order $\mathcal{O}$ at worst (i.e. there are no terms with order strictly more negative than $\mathcal{O}$). Then, we will write 
\begin{equation}
\label{E:Order_equivalent_definition}
    \tilde{A} \simeq \tilde{B} \Longleftrightarrow C_1(\partial^{2k+1}(u,r,s),u,s)\tilde{A} = C_2(\partial^{2k+1}(u,r,s),u,s)\tilde{B} + P
\end{equation}
where $C_1$ and $C_2$ are expressions containing $(u,s)$ and up to $2k+1$ derivatives of $(u,r,s)$, and $P$ contains terms with order $\mathcal{J} > \mathcal{O}$ (i.e. strictly better terms). Also, note that $C_1$ and $C_2$ must not contain undifferentiated $r$, since powers of $r$ would contribute to the calculated order for $\tilde{A}$ and $\tilde{B}$. For illustrative purposes (and by slight abuse of notation), we will sometimes write $A \simeq B +P$ instead of just $\tilde{A} \simeq \tilde{B}$ if we want to highlight the terms in $P$ and calculate their order which will be strictly greater that $\tilde{B}$.

As a final remark, we stated that $\partial(r, s, u)$ or $\bm{\partial}(r,s,u)$ does not, in general, contribute to the order of given term. However, there is one exception which can be seen by analyzing \eqref{E:System_r_equation} above. For situations where we have $D_t r$, we see that a full power of $r$ is obtained since $D_t r = -(\gamma-1)r \partial_\mu u^\mu$ and powers of $r$ are needed for calculating order. Thus, it can be useful to add the fact that
\begin{equation}
\label{E:D_t r equals r, convective derivative of r}
\begin{split}
      D_t r &\simeq r \\
\end{split}
\end{equation}
even though in \eqref{E:Order_equivalent_definition} we used the symbol $\simeq$ when referring to terms involving $(\tilde{u},\tilde{r}, \tilde{s})$.  
\end{remark}

It remains to show several lemmas and commutator identities. We will ultimately use our scheme to simplify the $B_{2k}$ and $C_{2k}^\alpha$ terms that appear in \eqref{E:G_and_H_definition}. Then, we can estimate these terms using the $\mathcal{H}^{2k}$ norm in Section \ref{S:Higher_order_energy_estimates}. The proof of these lemmas is a routine application of the commutator identities:
\begin{equation}
    \label{E:Commutator_idenitities_general}
    \begin{split}
        [\partial, D_t^N] \phi &= [\partial, D_t](D_t^{N-1} \phi) + D_t([\partial, D_t^{N-1}] \phi) \\
        [D_t, \partial^N] \phi &= [D_t, \partial] \partial^{N-1} \phi + \partial\left([D_t, \partial^{N-1}]\phi \right) \\
    \end{split}
\end{equation}

\begin{lemma}[First Commutator Lemma]
\label{L:Commutator}
Let $N \in \mathbb{N}$. For the convective derivative $D_t = u^\mu \partial_\mu$, generic spatial derivative $\partial$, and $\phi$ sufficiently smooth, the following identities hold:
\begin{equation}
    \begin{split}
        [\partial, D_t] \phi &= (\partial u^\nu) \partial_\nu \phi \\
        [\partial, D_t^2] \phi &= 2(\partial u^\nu) \partial_\nu (D_t \phi) + [D_t(\partial u^\nu) - (\partial u^\mu)(\partial_\mu u^\nu)] \partial_\nu \phi \\
        ... \\
        [\partial, D_t^N] \phi &\simeq \sum_{i=0}^{N-1} C_i^\nu \partial_\nu (D_t^i \phi) \\
    \end{split}
\end{equation}
where $C_i^\nu:= C_i(\partial u, D_t u)$ for $i \in \{0, ..., N-1\}$ are coefficients depending on the spatial and convective derivatives of $u$.

\end{lemma}

\begin{remark}
    We note that Lemma \ref{L:Commutator} is used extensively in Theorem \ref{Th:Higher_order_energy_estimates} in order to simplify the commutator terms from \eqref{E:G_and_H_definition}. For a discussion of the commutators used in elliptic estimates, see Section \ref{S:Elliptic estimates}.
\end{remark}

We also have an important commutator lemma which will be used in our higher order estimates.
\begin{lemma}[Second Commutator Lemma]
\label{L:Commutator_2}
Let $N \in \mathbb{N}$. For the convective derivative $D_t = u^\mu \partial_\mu$, generic $N$-th order spatial derivative $\partial^N$, and $\phi$ sufficiently smooth, the following identities hold:
\begin{equation}
    \begin{split}
        [D_t, \partial] \phi &= - (\partial u^\nu) \partial_\nu \phi \\
        [D_t, \partial^2] \phi &= -2 (\partial u^\nu) \partial_\nu (\partial \phi)- (\partial^2 u^\nu) \partial_\nu \phi \\
        ... \\
        [D_t, \partial^N] \phi &\simeq  \sum_{i=0}^{N-1} B_i^\nu \partial_\nu (\partial^i \phi) \\
    \end{split}
\end{equation}
where $B_i^\nu:= B_i(\partial u)$ for $i \in \{0, ..., N-1\}$ are coefficients depending on the spatial derivatives of $u$.
\end{lemma}

Let's continue the book-keeping process for higher order convective derivatives $D_t^i \tilde{s}, D_t^i \tilde{r}$, and $D_t^i \tilde{u}$. 
The following Lemma can be taken with Section \ref{S:Sovling_for_time_derivatives} and remark \ref{R:Solving_for_time_derivatives} in mind that we preview here. Namely, we use a generic spacetime derivative with the understanding that time derivatives will eventually be solved for in terms of spatial derivatives plus additional terms that are more sub-critical (see Lemma \ref{L:Solving_for_time_derivatives} and Remark \ref{R:Solving_for_time_derivatives} below). For now, we will write Lemma \ref{L:Book-keeping} using the generic spacetime derivative symbol $\bm{\partial}$ from Remark \ref{R:notation_for_derivatives}. 

\begin{lemma}[Book-keeping $D_t^i$ derivatives]
\label{L:Book-keeping}
The following simplifications hold when counting derivatives and powers of $r$. For $i=1$, we are left with the terms
\begin{equation}
\label{E:Lemma_base_case}
    \begin{split}
        D_t \tilde{s} & \simeq \tilde{u} \\
        D_t \tilde{r} &\simeq r \bm{\partial} \tilde{u} + \tilde{u} \\
        D_t \tilde{u} &\simeq \bm{\partial} \tilde{r} + \tilde{u} + \tilde{s} \\
    \end{split}
\end{equation}
where we only ignored $\bm{\partial} s, \bm{\partial} r, \bm{\partial} u, s, u,$ and $\proj$, and we listed only $\mathcal{H}^{i}$ critical or supercritical terms. Moreover, for $i \geq 2$, we get the following
\begin{equation}
    \begin{split}
        & D_t^i \tilde{s} \simeq D_t^{i-1} \tilde{u} \\    
        &D_t^i \tilde{r} \simeq \begin{cases}
        \sum_{l=0}^{i/2} r^{l} \bm{\partial}^{l+i/2} \tilde{r} &, \text{ if $i$ even} \\
        \sum_{l=0}^{(i+1)/2} r^{l} \bm{\partial}^{l+(i-1)/2} \tilde{u} &, \text{ if $i$ odd} \\
        \end{cases}
        \\
        &D_t^i \tilde{u} \simeq \begin{cases}
         \sum_{l=0}^{i/2} r^{l} \bm{\partial}^{l+i/2} \tilde{u} + \sum_{j=0}^{i/2 -1} r^j \bm{\partial}^{j+i/2} \tilde{r} &, \text{ if $i$ even} \\
        \sum_{l=0}^{(i-1)/2} r^{l} \bm{\partial}^{l+(i+1)/2} \tilde{r} + r^l \bm{\partial}^{l+(i-1)/2} (\tilde{u} + \tilde{s}) &, \text{ if $i$ odd}  \\
        \end{cases} \\
        \end{split}
\end{equation}

\begin{proof}
This lemma is proved inductively using the fact that $D_t r \simeq r$, Lemma \ref{L:Commutator}, and our assumptions on key terms that can be ignored at each step. 

\end{proof}
\end{lemma}

**************************************

\subsection{Solving for time derivatives}
\label{S:Sovling_for_time_derivatives}
There are many instances in which we would like to estimate the spacetime derivatives $\partial_\mu$ of a particular quantity using only the spatial derivatives $\partial$. To do so, we can can solve for the time derivatives in terms of spatial derivatives by using system \eqref{E:System} directly and the orthogonality identity \eqref{E:tilde_u_^0 identity}. 

\begin{lemma}[Solving for time derivatives]
\label{L:Solving_for_time_derivatives}
The following computations and simplifications hold under the book-keeping scheme from Section \ref{S:Creating the Book-keeping scheme}:

\begin{equation}
\label{E:Solving_for_time_derivatives}
    \begin{split}
    (u^0)^2 &= 1 + u^j u_j \\
     \tilde{u}^0 &= \frac{u_j}{u^0} \tilde{u}^j \simeq \tilde{u} \\
       \partial_i \tilde{u}^0 &= \frac{u_j}{u^0} \partial_i \tilde{u}^j + \tilde{u}^j \partial_i \frac{u_j}{u_0} \simeq \partial \tilde{u} + \tilde{u}\\
        \partial_t s &= - \frac{u^i}{u^0} \partial_i s\\
           \partial_t \tilde{s} &=  - \frac{u^i}{u^0} \partial_i \tilde{s} + \frac{u^i \tilde{u}^0 - u^0 \tilde{u}^i }{(u^0)^2} \partial_i s \simeq \partial \tilde{s} + \tilde{u}\\
             \partial_t u^0 &=  C_1^i \partial_i u^0 + C_2^i \partial_i r + C_3 r \partial_i u^i \\
              \partial_t \tilde{u}^0 &= C_1^i \partial_i \tilde{u}^0+ C_2^i \partial_i \tilde{r} + C_3 r \partial_i \tilde{u}^i + C_3\tilde{r} \partial_i u^i + \widetilde{C_1^i} \partial_i u^0  + \widetilde{C_2^i} \partial_i r  + \widetilde{C_3} r \partial_i u^i \\
              & \hspace{10mm} \simeq \partial \tilde{u} + \partial \tilde{r} + r \partial \tilde{u} + (\tilde{u} + \tilde{s} + \tilde{r}) +r(\tilde{u} + \tilde{s} + \tilde{r}) + r^2 (\tilde{u} + \tilde{s} + \tilde{r}) \\
               \partial_t r &= - \frac{u^i}{u^0} \partial_i r - \frac{\gamma-1}{u^0}r \partial_i u^i -  \frac{\gamma-1}{ u^0} r [\partial_t u^0] \\
    \partial_t \tilde{r} &= - \frac{u^i}{u^0} \partial_i \tilde{r} - \frac{\gamma-1}{u^0}r \partial_i \tilde{u}^i -  \frac{\gamma-1}{ u^0} r [\partial_t \tilde{u}^0] \\
    & \hspace{2mm} + \frac{u^i \tilde{u}^0 - u^0 \tilde{u}^i }{(u^0)^2} \partial_i r + \left(\frac{(\gamma-1)}{(u^0)^2} \tilde{u}^0 r - \frac{\gamma-1}{u^0} \tilde{r}\right) \partial_i u^i + \left(\frac{(\gamma-1)}{(u^0)^2} \tilde{u}^0 r - \frac{\gamma-1}{u^0} \tilde{r}\right) [\partial_t u^0]\\
    & \hspace{10mm} \simeq \partial \tilde{r} + r (\partial \tilde{u} + \partial \tilde{r}) + (\tilde{u} + \tilde{r})  + r(\tilde{u} + \tilde{s} + \tilde{r}) +r^2(\tilde{u} + \tilde{s} + \tilde{r}) + r^3 (\tilde{u} + \tilde{s} + \tilde{r}) \\
      \partial_t u^j &= - \frac{u^i}{u^0} \partial_i u^j - \frac{1}{(\Gamma+r)u^0} \proj^{j 0} \left[\partial_t r\right] -  \frac{1}{(\Gamma+r)u^0} \proj^{j i} \partial_i r \\
      \partial_t \tilde{u}^j &=- \frac{u^i}{u^0} \partial_i \tilde{u}^j  - \frac{1}{(\Gamma+r)u^0} \proj^{j 0} \left[\partial_t \tilde{r}\right] -  \frac{1}{(\Gamma+r)u^0} \proj^{j i} \partial_i \tilde{r} \\
      & \hspace{2mm} + \frac{u^i \tilde{u}^0 - u^0 \tilde{u}^i }{(u^0)^2} \partial_i u^j + \left(\frac{u^j(\Gamma' \tilde{s} + \tilde{r})+(\Gamma +r)\tilde{u}^j}{(\Gamma+r)^2} \right)[\partial_t r] \\
      & \hspace{5mm} + \left(\frac{(\delta^{ji} +u^j u^i )[(\Gamma+r)\tilde{u}^0 + u^0(\Gamma' \tilde{s} + \tilde{r})] - (\Gamma+r)u^0(\tilde{u}^j u^i + u^j \tilde{u}^i)}{(\Gamma+r)^2 (u^0)^2} \right)[\partial_i r] \\
      & \hspace{10mm} \simeq \partial \tilde{u} + \partial \tilde{r} +r (\partial \tilde{u} + \partial \tilde{r}) +(\tilde{u}  + \tilde{s}+ \tilde{r})  + r(\tilde{u} + \tilde{s} + \tilde{r}) +r^2(\tilde{u} + \tilde{s} + \tilde{r}) + r^3 (\tilde{u} + \tilde{s} + \tilde{r})
    \end{split}
\end{equation}
\newpage
with
\begin{equation}
    \begin{split}
     a_1 &= u^0 - \frac{(\gamma-1)((u^0)^2-1)}{(\Gamma+r)u^0}r  \\
     \tilde{a}_1 &= \tilde{u}^0 - \frac{(\gamma-1)((u^0)^2-1)}{(\Gamma+r)u^0} \tilde{r} \\
     & \hspace{7mm} -  r \left( \frac{2(\Gamma+r)(\gamma-1)u^0 \tilde{u}^0 -(\gamma-1)((u^0)^2-1)[(\Gamma' \tilde{s} + \tilde{r})u^0 + (\Gamma +r)\tilde{u}^0]  }{(\Gamma+r)^2 (u^0)^2} \right) \\
     & \simeq \tilde{u} + \tilde{r} + r (\tilde{u} + \tilde{s} + \tilde{r}) \\
        C_1^i :&= -\frac{u^i}{a_1} \\
        \widetilde{C_1^i} &= \frac{u^i \tilde{a}_1 - a_1 \tilde{u}^i}{(a_1)^2} \simeq \tilde{u} + \tilde{r} + r (\tilde{u} + \tilde{s} + \tilde{r}) \\
                C_2^i :&= \frac{1}{(\Gamma+r)u^0} C_1^i \\
                \widetilde{C_2^i} &= \frac{1}{(\Gamma+r)u^0} \widetilde{C_1^i} - \frac{[(\Gamma' \tilde{s} + \tilde{r})u^0 + (\Gamma +r)\tilde{u}^0]}{(\Gamma+r)^2 (u^0)^2} C_1^i  \simeq \tilde{u} + \tilde{s} + \tilde{r} + r (\tilde{u} + \tilde{s} + \tilde{r}) \\
                C_3 :&= \frac{(\gamma-1)((u^0)^2-1)}{u^i} C_2^i  \\
                \widetilde{C_3} &= \frac{(\gamma-1)((u^0)^2-1)}{u^i} \widetilde{C_2^i} + \frac{u^i 2(\gamma-1) \tilde{u}^0 -(\gamma-1)((u^0)^2-1) \tilde{u}^i  }{(u^i)^2} C_2^i  \simeq \tilde{u} + \tilde{s} + \tilde{r} + r (\tilde{u} + \tilde{s} + \tilde{r})
    \end{split}
\end{equation}
where, namely, we only ignored $\partial s, \partial r, \partial u, s, u$ and $O(1)$ coefficients in the simplfication $\simeq$.
    
\end{lemma}

\begin{remark}
\label{R:Solving_for_time_derivatives}

In general, we have that
\begin{equation}
\label{E:simplified_time_deriv_expressions}
\begin{split}
      \partial_t \tilde{s}  &\simeq \partial\tilde{s}  + (\text{terms which are more subcritical than $ \partial \tilde{s}$})\\
        \partial_t \tilde{u}  &\simeq \partial \tilde{u}  + (\text{terms which are more subcritical than $ \partial \tilde{u}$})\\
        \partial_t \tilde{r}  &\simeq \partial\tilde{r}  + (\text{terms which are more subcritical than $ \partial \tilde{r}$})\\
\end{split}
\end{equation}
with coefficient  expressions that have at most one spatial derivative of $s,r,$ or $u$. 
The point is that any time derivative of our linearized variables $(\tilde{s}, \tilde{r}, \tilde{u})$ can be written in terms of spatial derivatives (with good coefficients) plus additional subcritical terms that can be easily estimated. When estimating terms such as in Sections \ref{S:Vorticity_estimates} and \ref{S:Higher_order_energy_estimates}, we will often apply Lemma \ref{L:Solving_for_time_derivatives} proactively so as to not introduce unnecessary subcritical terms that obscure the argument. Thus, going forward, one can treat the spacetime expressions below with the generic \textit{spatial} derivative symbol $\partial$
\begin{equation*}
    \begin{split}
        \partial_\mu \tilde{s} &\simeq \partial \tilde{s}\\
         \partial_\mu \tilde{r} &\simeq \partial \tilde{r} \\
          \partial_\mu \tilde{u}^\nu &\simeq \partial \tilde{u}
    \end{split}
\end{equation*}
from the point of view of their $\mathcal{H}^{2k}$ order, as opposed to Remark \ref{R:notation_for_derivatives} where $\bm{\partial}$ would be used.

Recall that when referring to the \textbf{order} of a free boundary term, we are doing so with respect to Definition \ref{D:definition_critical_subcritical_supercritical} and Remark \ref{R:order_free_boundary_term} where terms with order $\mathcal{O}=0$ are denoted as critical, $\mathcal{O} < 0$ are supercritical, and $\mathcal{O} > 0$ are subcritical. 
For example, it is crucial in our expression \eqref{E:simplified_time_deriv_expressions} for $\partial_t \tilde{r}$ that the parentheses do \textit{not} contain terms of the form $\partial \tilde{u}$. 
This is because, at a given level say $k=1$, the $\mathcal{H}^2$ order of a term $\partial \tilde{r}$ is $0$ (i.e. critical), but the order of a term like $\partial \tilde{u}$ is $-\frac{1}{2}$ (i.e. supercritical) with order strictly \textit{worse}. If this were the case, we could not reliably argue that we could replace $\partial_t \tilde{r}$ with $\partial \tilde{r}$ since we would be introducing terms that require additional powers of $r$ in order to be estimated with the $\mathcal{H}^{2k}$ norm.
\end{remark}

\begin{proof}
    We leave computation details for Lemma \ref{L:Solving_for_time_derivatives} in the Appendix.
\end{proof}

To finish this section, we prove a useful Lemma which expands on the notion of order $\mathcal{O}$ to each of our key operations on free boundary terms. This Lemma is especially helpful in proving higher order analogs of our elliptic and div-curl estimates using induction. The last proposition in Lemma \ref{L:Order_of_operators} is helpful in jumping from the base case $\mathcal{H}^2$ to the general case $\mathcal{H}^{2k}$.

\begin{lemma}[The orders of each operation]
\label{L:Order_of_operators}
Let $k \in \mathbb{N}$ be fixed. Suppose that $T$ is a free boundary term with $\mathcal{H}^{2k}$ order $\mathcal{O}$ (i.e. $T$ is of the form $r^m \partial^l (\tilde{s}, \tilde{r}, \tilde{u})$ for some $m, l \in \mathbb{N}$, $m\geq0,0\leq l\leq 2k$, ). Then the following holds:
\begin{enumerate} 
    \item the term $rT$ has $\mathcal{H}^{2k}$ order $\mathcal{O} +1$
    \item the term $\partial T$ has $\mathcal{H}^{2k}$ order $\mathcal{O} -1$ 
    \item the term $D_t T$ has a minimum $\mathcal{H}^{2k}$ order $\mathcal{O} - \frac{1}{2}$
    \item If $K \in \mathbb{N}$ with $K \geq k$, then $T$ has $H^{2K}$ order $\mathcal{O}+ K-k$.
\end{enumerate}
\begin{proof}
    Since the case with $\tilde{s}$ is identical to $\tilde{u}$, it suffices to consider the following cases for our free boundary term $T$:
    \begin{enumerate}[(a)]
        \item $T_1= r^{m_1} \partial^{l_1} \tilde{r}$ which has $\mathcal{H}^{2k}$ order $m_1-l_1+k = \mathcal{O}$
        \item $T_2= r^{m_2} \partial^{l_2} \tilde{u}$ which has $\mathcal{H}^{2k}$ order $m_2-l_2+k - \frac{1}{2} = \mathcal{O}$
    \end{enumerate}
    
     For Part 1, observe that for cases (a) and (b), we simply have an extra power of $r$
    \begin{equation}
        \begin{split}
            rT_1 = r \left(r^{m_1} \partial^{l_1} \tilde{r} \right) = r^{m_1+1} \partial^{l_1} \tilde{r} &\implies rT_1 \text{ has $\mathcal{H}^{2k}$ order $m_1+1-l_1+k = \mathcal{O} +1$} \\
            rT_2 = r \left(r^{m_2} \partial^{l_2} \tilde{u} \right) = r^{m_2+1} \partial^{l_2} \tilde{u} &\implies rT_2 \text{ has $\mathcal{H}^{2k}$ order $m+1-l+k - \frac{1}{2} = \mathcal{O} +1$} \\
        \end{split}
    \end{equation}
    For Part 2, we see that there are two sub-cases when $m_1, m_2 \geq 1$ and $m_1 = 0 = m_2$. If $m_1 =0 = m_2$, then
    \begin{equation}
        \begin{split}
            \partial T_1 &= \partial (\partial^{l_1} \tilde{r}) = \partial^{l_1 + 1} \tilde{r} \implies \partial T_1 \text{ has $\mathcal{H}^{2k}$ order $m_1-(l_1 +1)+k = \mathcal{O} -1$} \\
             \partial T_2 &= \partial (\partial^{l_2} \tilde{u}) = \partial^{l_2 + 1} \tilde{u} \implies \partial T_2 \text{ has $\mathcal{H}^{2k}$ order $m_2-(l_2 +1)+k - \frac{1}{2} = \mathcal{O} -1$} \\
        \end{split}
    \end{equation}
    If $m_1, m_2 \geq 1$, then we get
    \begin{equation}
         \begin{split}
            \partial T_1 &= \partial (r^{m_1} \partial^{l_1} \tilde{r}) = m_1 r^{m_1-1} \partial^{l_1} \tilde{r} + r^{m_1} \partial^{l_1 +1} \tilde{r} \implies \partial T_1 \text{ has $\mathcal{H}^{2k}$ order $ \mathcal{O} -1$} \\
             \partial T_2 &= \partial (r^{m_2} \partial^{l_2} \tilde{u})= m_2 r^{m_2-1} \partial^{l_2} \tilde{u} + r^{m_2} \partial^{l_2 +1} \tilde{u} \implies \partial T_2 \text{ has $\mathcal{H}^{2k}$ order $ \mathcal{O} -1$} \\
        \end{split}
    \end{equation}
    since both terms have $\mathcal{H}^{2k}$ order $\mathcal{O} -1$, noting that we either lose a power of $r$ or gain a derivative (and both operations decrease the order from $\mathcal{O}$ to $\mathcal{O} -1$).
    
    For Part 3, we will first consider when $l_1 =0 = l_2$. For case with $T_1$, we have 
    \begin{equation}
         \begin{split}
            D_t T_1 = D_t(r^{m_1} \tilde{r}) &= m_1 r^{m_1-1} (D_t r) \tilde{r} + r^{m_1} D_t \tilde{r} \\
            &\simeq  m_1 r^{m_1} \tilde{r} + r^{m_1} (r \partial \tilde{u} + \tilde{u}) \\
            &= r^{m_1 +1} \partial \tilde{u} + r^{m_1} \tilde{u} + m_1 r^{m_1} \tilde{r}
        \end{split}
    \end{equation}
    where we used \eqref{E:D_t r equals r, convective derivative of r} and Lemma \ref{L:Book-keeping}. Then, computing the $\mathcal{H}^{2k}$ order (and noting that we now have some terms with $\tilde{u}$ instead of $\tilde{r}$), we see that the first two terms have order $\mathcal{O} - \frac{1}{2}$, while the last one has order $\mathcal{O}$. Thus, $D_t T_1$ has a minimum $\mathcal{H}^{2k}$ order $\mathcal{O} - \frac{1}{2}$ (i.e. it has order $\mathcal{O} - \frac{1}{2}$ at worst). For $T_2$, we get a similar computation using \eqref{E:D_t r equals r, convective derivative of r} and Lemma \ref{L:Book-keeping}
     \begin{equation}
         \begin{split}
            D_t T_2 = D_t(r^{m_2} \tilde{u}) &= m_2 r^{m_2-1} (D_t r) \tilde{u} + r^{m_2} D_t \tilde{u} \\
            &\simeq  m_2 r^{m_2} \tilde{u} + r^{m_2} (\partial \tilde{r} + \tilde{u} + \tilde{s}) \\
            &= r^{m_2} \partial \tilde{r} + r^{m_2} \tilde{u} + m_2 r^{m_2} \tilde{u} + r^{m_2} \tilde{s}
        \end{split}
    \end{equation}
    where the first term has order $m_2- 1 + k = (m_2 - 0 +k - \frac{1}{2}) - \frac{1}{2} = \mathcal{O} - \frac{1}{2}$ and the remaining terms have order $\mathcal{O}$. 

    For Part 3, it remains to consider when $l_1, l_2 \geq 1$. We will additionally make use of Lemma \ref{L:Commutator_2}. For the $T_1$ case, we have
\begin{equation}
         \begin{split}
            D_t T_1 = D_t(r^{m_1} \partial^{l_1} \tilde{r}) &\simeq r^{m_1} D_t (\partial^{l_1} \tilde{r}) + m_1 r^{m_1} \partial^{l_1} \tilde{r} \\
            &\simeq r^{m_1} \partial^{l_1} (D_t \tilde{r}) + r^{m_1} [D_t, \partial^{l}] \tilde{r} + m_1 r^{m_1} \partial^{l_1} \tilde{r} \\
         &\simeq r^{m_1} \partial^{l_1} (r \partial \tilde{u} + \tilde{u})+ r^{m_1} \left( \sum_{i=0}^{l_1-1} B_i^\nu \partial_\nu (\partial^i \tilde{r})  \right) + m_1 r^{m_1} \partial^{l_1} \tilde{r} \\
         &\simeq r^{m_1} \partial^{l_1} \tilde{u} + r^{m_1}  \partial r \partial^{l_1} \tilde{u} + r^{m_1 +1} \partial^{l_1 + 1} \tilde{u} + \left(\sum_{i=0}^{l_1-1} r^{m_1} \partial^{i+1} \tilde{r} \right) + m_1 r^{m_1} \partial^{l_1} \tilde{r} 
        \end{split}
    \end{equation}
    where we note that $\partial^{l_1}(r \partial \tilde{u})$ only produces critical terms when all $\partial^{l_1}$ derivatives hit $\partial \tilde{u}$, or one derivative hits $r$ and $\partial^{l_1 -1}$ derivatives hit $\partial \tilde{u}$.  For the summation terms, we make use of Remak \ref{R:Solving_for_time_derivatives} to simplify $\partial_\nu \partial^i \tilde{r}$ as $\partial^{i+1} \tilde{r}$. Then, counting the $\mathcal{H}^{2k}$ order for each of the terms, we see that the first three have order $\mathcal{O} - \frac{1}{2}$, whereas the remaining terms have order $\mathcal{O}$ at worst (with additional terms in the summation having order $\mathcal{O} + l_1 - i$ with $0 \leq i< l_1$). 

    For the $T_2$ case with $l_1, l_2 \geq 1$, we perform a similar computation and keep track of whether the terms have $\tilde{r}$ or $\tilde{u}$ when computing order
    \begin{equation}
         \begin{split}
            D_t T_2 = D_t(r^{m_2} \partial^{l_2} \tilde{u}) &\simeq r^{m_2} D_t (\partial^{l_2} \tilde{u}) + m_2 r^{m_2} \partial^{l_2} \tilde{u} \\
            &\simeq r^{m_2} \partial^{l_2} (D_t \tilde{u}) + r^{m_2} [D_t, \partial^{l}] \tilde{u} + m_2 r^{m_2} \partial^{l_2} \tilde{u} \\
         &\simeq r^{m_2} \partial^{l_2} (\partial \tilde{r} + \tilde{u} + \tilde{s})+ r^{m_2} \left( \sum_{i=0}^{l_2-1} B_i^\nu \partial_\nu (\partial^i \tilde{u})  \right) + m_2 r^{m_2} \partial^{l_2} \tilde{u} \\
         &\simeq r^{m_2} \partial^{l_2 +1} \tilde{r} + r^{m_2} \partial^{l_2} \tilde{u} + r^{m_2} \partial^{l_2} \tilde{s}  + \left(\sum_{i=0}^{l_1-1} r^{m_2} \partial^{i+1} \tilde{u} \right) + m_2 r^{m_2} \partial^{l_2} \tilde{u} 
        \end{split}
    \end{equation}
    where the first term has $\mathcal{H}^{2k}$ order $m_2 - (l_2 +1) +k = (m_2 - l_2 + k - \frac{1}{2}) - \frac{1}{2} = \mathcal{O} - \frac{1}{2}$, and the remaining terms have order $\mathcal{O}$ at worst.
    Thus $D_t T$ only produces terms that have a minimum order of $\mathcal{O} - \frac{1}{2}$, and all other terms have strictly $\mathcal{O}$ order or better at the $\mathcal{H}^{2k}$ level. This completes the proof of Part 3.

    Lastly, we will provide the proof of Part 4. In the previous parts, each operation changed the number of powers of $r$, the number of derivatives, or whether we are using $\tilde{r}$ or $\tilde{u}$ and each of these has an effect on the computed order. In Part 4, we are only seeing what happens if we raise the level from $k$ to some $K \geq k$, and the proof can be seen by noticing that for $T_1$,
    \begin{equation*}
        m_1 - l_1 + K  = (m_1 -l_1 + k) + K -k  = \mathcal{O} + K - k
    \end{equation*}
    and for $T_2$, 
    \begin{equation*}
        m_2 - l_2 + K - \frac{1}{2}  = (m_2 -l_2 + k - \frac{1}{2}) + K -k  = \mathcal{O} + K - k.
    \end{equation*}
    This completes the proof of Part 4, and the proof of Lemma \ref{L:Order_of_operators} is complete.
\end{proof}
\end{lemma}

\begin{remark}
\label{R:Order_of_operators}
    A natural corollary of this Lemma is the following. Suppose that $T$ has $\mathcal{H}^{2k}$-order $\mathcal{O}$. Then, $D_t^2 T$ has $\mathcal{H}^{2k}$-order $\mathcal{O} -1$, but this implies that $D_t^2 T$ has $\mathcal{H}^{2(k+1)}$-order $\mathcal{O}$ (if we increase the level $k$ by 1). Hence, if a given term is $\mathcal{H}^2$ subcritical (with $k=1$), applying $D_t^{2(k-1)}$ (for $k\geq 2$) will keep the term $\mathcal{H}^{2k}$-subcritical at the corresponding $k$ level. 
\end{remark}

\section{Estimates for the transport energy}

\subsection{Entropy estimates}
\label{S:Entropy estimates}

Recall the definition of $E^{2k}_{\text{transport}}$ in \eqref{E:E^2k_transport} which consists of both the entropy and the vorticity. Our goal in this section is to prove estimates for the entropy, whereas the voriticity will be handled separately in Section \ref{S:Vorticity_estimates}. Instead of applying $D_t^{2k}$ to the $\tilde{s}$ equation, we can simply take $\partial^{2k}$ and estimate $\tilde{s}$ using the $\mathcal{H}^{2k}$ norm directly. This can be done because $\tilde{s}$ satisfies a transport equation instead of a wave equation that would require the intermediate $E^{2k}_{\text{wave}}$ energy. We have the following proposition which is routine by envoking out previous Lemmas on the order of terms.
\begin{proposition}[Entropy Estimates]
\label{P:Entropy Estimates}
\begin{equation}
    \frac{d}{dt} \norm{\tilde{s}}_{H^{2k, k+ \frac{1}{2(\gamma-1)}}}^2 \lesssim B_2 \norm{(\tilde{s}, \tilde{r}, \tilde{u})}_{\mathcal{H}^{2k}}^2
\end{equation}
where $B_2$ depends on $L^\infty$ norms of up to $2k+1$ derivatives of $(s,r,u)$.
\begin{proof}
    Using equation \eqref{E:Linearized_System_2_s_equation} Lemma \ref{L:Commutator_2} for the commutator $[D_t, \partial^{2k}]$, we get the following:
\begin{equation}
\begin{split}
    D_t (\partial^{2k} \tilde{s}) & = [D_t, \partial^{2k}] \tilde{s} + \partial^{2k} f \\
    &= \sum_{i=0}^{2k-1} B_i^\nu \partial_\nu (\partial^i \tilde{s}) + \partial^{2k}(\tilde{u}^\mu \partial_\mu s) \\
    &\simeq \sum_{i=0}^{2k} B_i \partial^i \tilde{s} + \partial^{2k}(\tilde{u}^\mu \partial_\mu s) \\
\end{split}
\end{equation}
where we recall the definition of $f$ from \eqref{E:Linearized_f_g_h}, and we used Remark \ref{R:Solving_for_time_derivatives} to combine $\partial_\nu \partial^i \tilde{s}$ and write everything in terms of spatial derivatives.
Then, multipliying the equation by $r^{2k + \frac{1}{\gamma-1}} \partial^{2k} \tilde{s}$, we get
\begin{equation}
\begin{split}
     \frac{1}{2} r^{2k+ \frac{1}{\gamma-1}} D_t (|\partial^{2k} \tilde{s}|^2) &\simeq  r^{2k + \frac{1}{\gamma-1}} \partial^{2k} \tilde{s} \left(\sum_{i=0}^{2k} B_i \partial^i \tilde{s}\right)  + r^{2k + \frac{1}{\gamma-1}} \partial^{2k} \tilde{s} \partial^{2k} (\tilde{u}^\mu \partial_\mu s) \\
     \frac{1}{2} D_t (|r^{k+ \frac{1}{2(\gamma-1)}} \partial^{2k} \tilde{s}|^2) &\simeq -\frac{1}{2}D_t(r^{2k+ \frac{1}{\gamma-1}})|\partial^{2k} \tilde{s}|^2  + r^{2k + \frac{1}{\gamma-1}} \partial^{2k} \tilde{s} \left(\sum_{i=0}^{2k} B_i \partial^i \tilde{s} \right)+ r^{2k + \frac{1}{\gamma-1}} \partial^{2k} \tilde{s} \partial^{2k} (\tilde{u}^\mu \partial_\mu s) \\
     &\simeq \left(k(\gamma-1)+ \frac{1}{2}\right)  r^{2k+ \frac{1}{\gamma-1}}|\partial^{2k} \tilde{s}|^2  + r^{2k + \frac{1}{\gamma-1}} \partial^{2k} \tilde{s} \left(\sum_{i=0}^{2k} B_i \partial^i \tilde{s} \right) \\
     & \hspace{5mm} + r^{2k + \frac{1}{\gamma-1}} \partial^{2k} \tilde{s} \partial^{2k} (\tilde{u}^\mu \partial_\mu s) \\
\end{split}
\end{equation}
Integrating, this leads to:
\begin{equation}
\label{E:Entropy_calc_1}
    \begin{split}
         \frac{d}{dt} \frac{1}{2}  \int_{\Omega_t} |r^{k+ \frac{1}{2(\gamma-1)}} \partial^{2k} \tilde{s}|^2 dx &\simeq  \int_{\Omega_t} \frac{1}{u^0} \left(k(\gamma-1) + \frac{1}{2}\right)  r^{2k + \frac{1}{\gamma-1}}|\partial^{2k} \tilde{s}|^2 \hspace{1mm} dx \\
     &\hspace{5mm} + \int_{\Omega_t} r^{2k + \frac{1}{\gamma-1}} \partial^{2k} \tilde{s} \left(\sum_{i=0}^{2k} B_i \partial^i \tilde{s} \right) \hspace{1mm} dx \\
      &\hspace{5mm} + \int_{\Omega_t} \frac{1}{u^0}  r^{2k + \frac{1}{\gamma-1}} \partial^{2k} \tilde{s} \partial^{2k} (\tilde{u}^\mu \partial_\mu s) \hspace{1mm} dx \\
      & \hspace{5mm} + \frac{1}{2} \int_{\Omega_t} |r^{k+ \frac{1}{2(\gamma-1)}} \partial^{2k} \tilde{s}|^2 \partial_i \left(\frac{u^i}{u^0} \right) dx \\
      &= J_1 + J_2 + J_3 + J_4.
    \end{split}
\end{equation}
For the $J_1$ term, we note that $u$ and its derivatives have been absorbed in $\simeq$ symbol (for example when we simplifying $D_t(r^{2k + \frac{1}{\gamma-1}})$). Thus, we can remove any derivatives of $u$ and quickly estimate 
\begin{equation}
\label{E:J_1_term_entropy}
    |J_1| \lesssim \norm{\partial u}_{L^\infty(\Omega_t)}  \left(\int_{\Omega_t} |r^{k+ \frac{1}{2(\gamma-1)}} \partial^{2k} \tilde{s}|^2 dx \right) \lesssim \norm{\partial u}_{L^\infty(\Omega_t)} \norm{(\tilde{s}, \tilde{r}, \tilde{u})}_{\mathcal{H}^{2k}}^2
\end{equation}
recalling the definition of $\mathcal{H}^{2k}$ from \eqref{E:Higher_order_H^2k_norm}, and Remark \ref{R:H^2k_norm_equivalence}.
The estimate for $J_4$ follows similarly since $\partial_i \left( \frac{u^i}{u^0}\right)$ is an expression depending on up to one spatial derivative of $u$.
For the $J_2$ and $J_3$ terms, we simply use Lemmas \ref{L:free_boundary_embedding_lemma} and \ref{L:Solving_for_time_derivatives} as needed while keeping track of subcritical terms.

\end{proof}
\end{proposition}

\subsection{Vorticity estimates}
\label{S:Vorticity_estimates}
We also plan to derive similar estimates for the vorticity which will be used in connection with the elliptic estimates. We recall from \cite{Choquet-BruhatBook-2009} that the relativistic vorticity is given by the following
\begin{equation}
    \omega_{\alpha \beta} = \partial_\alpha v_\beta - \partial_\beta v_\alpha.
\end{equation}
where $v$ is the enthalpy current defined by
\begin{equation}
    v_\alpha = h u_\alpha, \hspace{6mm} h =\frac{\Gamma+r}{\Gamma}, \hspace{6mm} \Gamma = \Gamma(s) =\frac{(\gamma-1)^{\frac{\gamma-1}{\gamma}}}{\gamma} e^{-\frac{\gamma-1}{\gamma} s}
\end{equation}

Similar to \cite{RezzollaZanottiBookRelHydro-2013}, we get an expression for $u^\alpha \omega_{\alpha \beta}$ using the following steps. From the fact that $v^\alpha v_\alpha = -h^2$, we get
\begin{equation}
     - v^\alpha \partial_\beta v_\alpha = h \partial_\beta h.
\end{equation}
From \eqref{E:System_v_equation} rewritten in terms of $v$, we get
\begin{equation}
    \frac{1}{h^2} v^\alpha \partial_\alpha v_\beta = - \frac{1}{\Gamma+r} \partial_\beta r
\end{equation}
Then, we can compute the following expression
\begin{equation}
    \begin{split}
        \frac{1}{h} \left(u^\alpha \partial_\alpha v_\beta - u^\alpha \partial_\beta v_\alpha \right) 
         &= \frac{r}{(\Gamma +r)} \frac{\gamma-1}{\gamma} \partial_\beta s \\
    \end{split}
\end{equation}
Multiplying by $h$, we get
\begin{equation}
\label{E:Vorticity_eq_1}
    u^\alpha \omega_{\alpha \beta} = \frac{r}{\gamma \Gamma} (\gamma-1) \partial_\beta s 
\end{equation}
which agrees with \cite{RezzollaZanottiBookRelHydro-2013} as 
\begin{equation}
\label{E:varepsilon_alternate_form}
    \varepsilon = \frac{r}{\gamma \Gamma}
\end{equation}
 after checking \eqref{E:internal_energy_simplified}. To find the evolution equation satisfied by $\omega_{\alpha \beta}$, we can first compute that
\begin{equation*}
    \partial_\alpha \varepsilon = \frac{1}{\gamma \Gamma} \left( \partial_\alpha r + \frac{\gamma-1}{\gamma} r \partial_\alpha s \right).
\end{equation*} 
Using similar computations as \cite{RezzollaZanottiBookRelHydro-2013} and applying \eqref{E:varepsilon_alternate_form}, we have
\begin{equation}
\label{E:Vorticity_eq_2}
    \begin{split}
        u^\mu \partial_\mu \omega_{\alpha \beta} + \partial_\alpha u^\mu \omega_{\mu \beta} + \partial_\beta u^\mu \omega_{\alpha \mu} 
        &= \frac{\gamma-1}{\gamma \Gamma} (\partial_\alpha r \partial_\beta s - \partial_\beta r \partial_\alpha s).
    \end{split}
\end{equation}
which is the evolution equation for the vorticity $\omega_{\alpha \beta}$ in terms of the unknowns $(s, r, u)$.

Next, we introduce the full linearized voricity $\tilde{\omega}_{\alpha \beta}$ by
\begin{equation}
\label{E:full_linearized_vorticity_definition}
    \begin{split}
        \tilde{\omega}_{\alpha \beta} &= \partial_\alpha(\widetilde{h u_\beta}) -\partial_\beta(\widetilde{h u_\alpha}) \\
        &= \partial_\alpha (h \tilde{u}_\beta) - \partial_\beta(h \tilde{u}_\alpha) + \partial_\alpha (\tilde{h} u_\beta) - \partial_\beta(\tilde{h} u_\alpha) \\
        &:= \hat{\omega}_{\alpha \beta} + \overline{\omega}_{\alpha \beta}
    \end{split}
\end{equation}
where 
\begin{equation}
\label{E:reduced_lin_vort_evolution_eq}
    \hat{\omega}_{\alpha \beta}:= \partial_\alpha (h \tilde{u}_\beta) - \partial_\beta(h \tilde{u}_\alpha)
\end{equation}
will be called the \textbf{reduced} linearized vorticity and $\overline{\omega}_{\alpha \beta}$ is the $\tilde{h}$-scaled vorticity. 

The plan is to derive estimates for the full linearized vorticity $\tilde{\omega}$, and then relate this back to the reduced linearized vorticity $\hat{\omega}$ by $\norm{\hat{\omega}} \lesssim \norm{\tilde{\omega}} + \norm{\overline{\omega}}$. Our reasoning is a matter of convenience, since estimates for $\tilde{\omega}$ are more readily available in view of \eqref{E:full_lin_vort_evolution_eq}. Additionally, $\hat{\omega}$ better resembles $\curl \tilde{u}$ in connection with our elliptic estimates (see Section \ref{S:Elliptic estimates} and Lemma \ref{L:relating_vorticity_spatial_curl}).

Throughout this section, we will use the following facts similar to the bookkeeping system from Lemma \ref{L:Book-keeping}. We put the order simplifications here for reference
\begin{equation}
\label{E:vorticity_facts_useful}
\begin{split}
    \overline{\omega}_{\alpha \beta} &  \simeq \partial \tilde{h} + \tilde{h}\\
    \tilde{h} &\simeq \tilde{r} + r \tilde{s}\\
     \partial_\alpha \tilde{h} &\simeq \partial \tilde{r} + r \partial \tilde{s}+ \tilde{r} + \tilde{s}  + r \tilde{s}
\end{split}
\end{equation}
where we also make use of Lemma \ref{L:Solving_for_time_derivatives} in the expression for $\partial_\alpha \tilde{h}$, i.e. $\partial_t \tilde{h}$ will have a slightly different expression, just with more subcritical terms coming from $\partial_t \tilde{s} \sim \partial \tilde{s}$ and $\partial_t \tilde{r} \sim \partial \tilde{r}$.



From linearizing \eqref{E:Vorticity_eq_2}, we have the evolution equation for $\tilde{\omega}$ given by
\begin{equation}
\label{E:full_lin_vort_evolution_eq}
    \begin{split}
        u^\mu \partial_\mu \tilde{\omega}_{\alpha \beta} + \partial_\alpha u^\mu \tilde{\omega}_{\mu \beta} + \partial_\beta u^\mu \tilde{\omega}_{\alpha \mu} &= -\tilde{u}^\mu \partial_\mu \omega_{\alpha \beta} - \partial_\alpha \tilde{u}^\mu \omega_{\mu \beta} - \partial_\beta \tilde{u}^\mu \omega_{\alpha \mu}  \\
        & \hspace{10mm} + \frac{(\gamma-1)^2}{\gamma^2 \Gamma} \tilde{s} (\partial_\alpha r \partial_\beta s - \partial_\beta r \partial_\alpha s) \\
        & \hspace{10mm} + \frac{\gamma-1}{\gamma \Gamma} (\partial_\alpha \tilde{r} \partial_\beta s + \partial_\alpha r \partial_\beta \tilde{s} - \partial_\beta \tilde{r} \partial_\alpha s - \partial_\beta r \partial_\alpha \tilde{s})
    \end{split}
\end{equation}

\begin{remark}
For the following lemma, recall Remark \ref{R:H^2k_norm_equivalence} where we have by embedding theorems that $\mathcal{H}^{2k} \sim H^{2k ,k + \frac{1}{2(\gamma-1)}} \times H^{2k ,k + \frac{1}{2(\gamma-1)} - \frac{1}{2}} \times H^{2k ,k + \frac{1}{2(\gamma-1)}}$ where the weight on $\tilde{r}$ is less by $\frac{1}{2}$.
\end{remark}

\begin{lemma}[Estimates for $\tilde{\omega}$]
\label{L:Estimate_full_lin_vorticity}
For a solution $\tilde{\omega}$ to \eqref{E:full_lin_vort_evolution_eq} on an interval $[0,T]$ defined in this section, we have the following estimate
\begin{equation}
\label{E:Full_lin_vort_estimate}
\frac{d}{dt} \norm{\tilde{\omega}}_{H^{2k-1, k+ \frac{1}{2(\gamma-1)}}}^2 \lesssim \mathcal{C} \norm{(\tilde{s}, \tilde{r}, \tilde{u})}_{\mathcal{H}^{2k}}^2
\end{equation}
where $\mathcal{C}$ depends on the $L^\infty(\Omega_t)$ norms of up to $2k+1$ derivatives of $s,r$, and $u$.

\end{lemma}
\begin{proof}
We begin by fixing a spatial multiindex $l$ such that $|l| \leq 2k-1$. We recall the notation $|\partial^l \tilde{\omega}|_G^2 = G^\alpha_\gamma G^\beta_\delta \partial^l \tilde{\omega}^{\gamma \delta} \partial^l \tilde{\omega}_{\alpha \beta}$ (where $l$ is not summed over) and $G$ is the Riemannian metric defined in \eqref{E:G_metric_definition}. Then, we apply $\partial^l$ to \eqref{E:full_lin_vort_evolution_eq} and contract against the expression $r^{2k+\frac{1}{(\gamma-1)}} G^\alpha_\gamma G^\beta_\delta \partial^l \tilde{\omega}^{\gamma \delta}$ . 

When contracting, we will use the quick observation that
\begin{equation}
\label{E:vorticity_esimtate_start_1}
\begin{split}
    r^{2k+\frac{1}{(\gamma-1)}} G^\alpha_\gamma G^\beta_\delta \partial^l \tilde{\omega}^{\gamma \delta} u^\mu \partial_\mu \partial^l \tilde{\omega}_{\alpha \beta} &= \frac{1}{2} r^{2k+\frac{1}{(\gamma-1)}} D_t (|\partial^l \tilde{\omega}|_G^2) -  r^{2k+\frac{1}{(\gamma-1)}} \frac{1}{2} D_t(G^\alpha_\gamma G^\beta_\delta) \partial^l \tilde{\omega}^{\gamma \delta} \partial^l \tilde{\omega}_{\alpha \beta} \\
    &= \frac{1}{2} D_t( r^{2k+\frac{1}{(\gamma-1)}} |\partial^l \tilde{\omega}|_G^2) - \frac{1}{2}(2k + \frac{1}{\gamma-1}) r^{2k+\frac{1}{(\gamma-1)}} \partial_\mu u^\mu |\partial^l \tilde{\omega}|_G^2 \\
    & \hspace{7mm} -  r^{2k+\frac{1}{(\gamma-1)}} \frac{1}{2} D_t(G^\alpha_\gamma G^\beta_\delta) \partial^l \tilde{\omega}^{\gamma \delta} \partial^l \tilde{\omega}_{\alpha \beta}.
\end{split}
\end{equation}
and only the first term in \eqref{E:vorticity_esimtate_start_1} will be kept on the LHS. This leads to the following long expression:

\begin{equation}
\label{E:Vorticity_long_expression}
    \begin{split}
           \frac{1}{2} D_t( r^{2k+\frac{1}{(\gamma-1)}} |\partial^l \tilde{\omega}|_G^2) &= \frac{1}{2}(2k + \frac{1}{\gamma-1}) r^{2k+\frac{1}{(\gamma-1)}} \partial_\mu u^\mu |\partial^l \tilde{\omega}|_G^2 \\
           & \hspace{7mm} + r^{2k+\frac{1}{(\gamma-1)}} \frac{1}{2} D_t(G^\alpha_\gamma G^\beta_\delta) \partial^l \tilde{\omega}^{\gamma \delta} \partial^l \tilde{\omega}_{\alpha \beta} \\
           & \hspace{7mm} - r^{2k+\frac{1}{(\gamma-1)}} G^\alpha_\gamma G^\beta_\delta \partial^l \tilde{\omega}^{\gamma \delta}  \partial^l \left(\partial_\alpha u^\mu \tilde{\omega}_{\mu \beta} + \partial_\beta u^\mu \tilde{\omega}_{\alpha \mu}  \right) \\
             & \hspace{7mm} - r^{2k+\frac{1}{(\gamma-1)}} G^\alpha_\gamma G^\beta_\delta \partial^l \tilde{\omega}^{\gamma \delta} \partial^l\left( \tilde{u}^\mu \partial_\mu \omega_{\alpha \beta}\right) \\
              & \hspace{7mm} - r^{2k+\frac{1}{(\gamma-1)}} G^\alpha_\gamma G^\beta_\delta \partial^l \tilde{\omega}^{\gamma \delta} \partial^l\left(\partial_\alpha  \tilde{u}^\mu \omega_{\mu \beta} \right)\\
               & \hspace{7mm} - r^{2k+\frac{1}{(\gamma-1)}} G^\alpha_\gamma G^\beta_\delta \partial^l \tilde{\omega}^{\gamma \delta} \partial^l\left(\partial_\beta  \tilde{u}^\mu \omega_{\alpha \mu} \right)\\
               & \hspace{7mm} + r^{2k+\frac{1}{(\gamma-1)}} G^\alpha_\gamma G^\beta_\delta \partial^l \tilde{\omega}^{\gamma \delta} \partial^l\left(\frac{(\gamma-1)^2}{\gamma^2 \Gamma} \tilde{s} (\partial_\alpha r \partial_\beta s - \partial_\beta r \partial_\alpha s) \right)\\
                & \hspace{7mm} + r^{2k+\frac{1}{(\gamma-1)}} G^\alpha_\gamma G^\beta_\delta \partial^l \tilde{\omega}^{\gamma \delta} \partial^l\left(\frac{\gamma-1}{\gamma \Gamma} (\partial_\alpha \tilde{r} \partial_\beta s + \partial_\alpha r \partial_\beta \tilde{s} - \partial_\beta \tilde{r} \partial_\alpha s - \partial_\beta r \partial_\alpha \tilde{s})  \right)\\
                &:= \sum_{i=1}^8 R_i
    \end{split}
\end{equation}
where each term on the RHS of \eqref{E:Vorticity_long_expression} will be handled separately. Then, using the function $f(x,t) = r^{2k+\frac{1}{(\gamma-1)}} |\partial^l \tilde{\omega}|_G^2$, we have:
\begin{equation}
\label{E:vorticity_estimates_integral}
\begin{split}
     \frac{d}{dt} \int_{\Omega_t} f \hspace{1mm} dx &= \int_{\Omega_t} f \partial_i \left(\frac{u^i}{u^0} \right) \hspace{1mm} dx + \int_{\Omega_t} \frac{1}{u^0} \sum_{i=1}^8 R_i \hspace{1mm} dx
\end{split}
\end{equation}
We plan to estimate each of the terms on the RHS of \eqref{E:vorticity_estimates_integral} by the $\mathcal{H}^{2k}$ norm which is also used in Theorem \ref{Th:Higher_order_energy_estimates}. This is done by analyzing the $\mathcal{H}^{2k}$ order (see Remark \ref{R:order_free_boundary_term}) and making use of Lemma \ref{L:Solving_for_time_derivatives} as needed for any time derivatives that appear.

We start by observing that the first term on the RHS of \eqref{E:vorticity_estimates_integral} is handled in an identical way as the term $J_4$ from \eqref{E:Entropy_calc_1} in the entropy estimates. We can use \eqref{E:full_linearized_vorticity_definition} and
\eqref{E:vorticity_facts_useful} to get the additional comment that
\begin{equation}
\label{E:vorticity_facts_useful_2}
\begin{split}
    \tilde{\omega}:= \hat{\omega} + \overline{\omega} \simeq \partial \tilde{u} + \partial \tilde{r} + r \partial \tilde{s}+ \tilde{s} + \tilde{r} + r \tilde{s}
\end{split}
\end{equation}
noting that, from an $\mathcal{H}^{2k}$ order perspective, $\tilde{\omega}$ contains terms with the same order as $\partial \tilde{u}$ or better. These terms primarily come from the reduced linearized vorticity $\hat{\omega}$, whereas $\overline{\omega} \sim \partial \tilde{h} + \tilde{h}$ and $\tilde{h}$ contains terms that are like $\tilde{r}$ at worst.

To finish estimating the first term on the RHS of \eqref{E:vorticity_estimates_integral}, we have
\begin{equation}
\label{E:R_0_term_vorticity_estimate}
\begin{split}
    \left|\int_{\Omega_t} r^{2k+\frac{1}{(\gamma-1)}} |\partial^l \tilde{\omega}|_G^2 \partial_i \left(\frac{u^i}{u^0} \right) \hspace{1mm} dx \right| &\lesssim \norm{\partial u}_{L^\infty(\Omega_t)} \int_{\Omega_t} r^{2k+\frac{1}{(\gamma-1)}} \left|G^\alpha_\gamma G^\beta_\delta \partial^l \tilde{\omega}^{\gamma \delta} \partial^l \tilde{\omega}_{\alpha \beta} \right| \hspace{1mm} dx \\
    &\lesssim \norm{\partial u}_{L^\infty(\Omega_t)} \int_{\Omega_t} r^{\frac{2-\gamma}{\gamma-1}} \left|r^{k+\frac{1}{2}}  \partial^l (\partial \tilde{u} + \partial \tilde{r}+ r \partial \tilde{s} + \tilde{s} + \tilde{r} + r\tilde{s})\right|^2 \hspace{1mm} dx
\end{split}
\end{equation}
where we applied \eqref{E:vorticity_facts_useful_2}. The point is that if we greatly expand the expression $G^\alpha_\gamma G^\beta_\delta \partial^l \tilde{\omega}^{\gamma \delta} \partial^l \tilde{\omega}_{\alpha \beta}$, only the terms with the \textit{worst} $\mathcal{H}^{2k}$ order need to be checked, since terms that are more subcritical have plenty of powers of $r$ to be estimated using the $\mathcal{H}^{2k}$ norm. In this integral, the terms with the worst order are $r^{k+\frac{1}{2}} \partial^{l+1} \tilde{u}$. Since $|l|\leq 2k-1 $, we see that the order of this term is exactly $0$ (critical) when $|l|=2k-1$ and subcritical when $|l| < 2k-1$. Thus, it can be estimated using the $\mathcal{H}^{2k}$ norm. For any cross terms in the expansion of $\left|  \partial^l (\partial \tilde{u} + \partial \tilde{r}+ r \partial \tilde{s} + \tilde{s} + \tilde{r} + r\tilde{s})\right|^2$, we simply apply Lemma \ref{L:Sum_orders} where the order of a product of is the sum of the orders of each term. For example, in the product $|(r^{k + \frac{1}{2(\gamma-1)}} \partial^{l+1} \tilde{u}) (r^{k + \frac{1}{2(\gamma-1)}} \partial^{l+1} \tilde{r})|$, we see that the order (i.e. when $\gamma=2$) is $0 + 1/2 = 1/2 \geq 0$ at worst when $|l|=2k-1$, so Lemma \ref{L:Sum_orders} implies that
\begin{equation}
    \int_{\Omega_t} |(r^{k + \frac{1}{2(\gamma-1)}} \partial^{l+1} \tilde{u}) (r^{k + \frac{1}{2(\gamma-1)}} \partial^{l+1} \tilde{r})| \hspace{1mm} dx \lesssim \norm{(\tilde{s}, \tilde{r}, \tilde{u})}_{\mathcal{H}^{2k}}^2 
\end{equation}
All other cross terms in $\left|  \partial^l (\partial \tilde{u} + \partial \tilde{r}+ r \partial \tilde{s} + \tilde{s} + \tilde{r} + r\tilde{s})\right|^2$ are handled using Lemma \ref{L:Sum_orders}. Thus, we get
\begin{equation}
      \left|\int_{\Omega_t} r^{\frac{2-\gamma}{\gamma-1}} |r^{k+\frac{1}{2}}\partial^l \tilde{\omega}|_G^2 \partial_i \left(\frac{u^i}{u^0} \right) \hspace{1mm} dx \right| \lesssim  \mathcal{C}_0 \norm{(\tilde{s}, \tilde{r}, \tilde{u})}_{\mathcal{H}^{2k}}^2 
\end{equation}
where $\mathcal{C}_0$ depends on the $L^\infty(\Omega_t)$ norms of up to $2k$ derivatives of $s,r$, and $u$.

We will proceed through each remaining term on the RHS of \eqref{E:vorticity_estimates_integral}. For the integrals containing $R_1$ and $R_2$, the estimate is quite similar to the first term since derivatives of $u$ will be removed, and $D_t (G^\alpha_\gamma G^\beta_\delta) = D_t((g^\alpha_\gamma + 2u^\alpha u_\gamma)(g^\beta_\delta + 2u^\beta u_\delta))$ is another expression containing derivatives of $u$. We have
\begin{equation}
    |R_1| \lesssim \norm{\partial u}_{L^\infty(\Omega_t)} \int_{\Omega_t} r^{\frac{2-\gamma}{\gamma-1}}  \left(k+ \frac{1}{2(\gamma-1)}\right) |r^{k+\frac{1}{2}}\partial^l \tilde{\omega}|_G^2 \hspace{1mm} dx \lesssim \mathcal{C}_1 \norm{(\tilde{s}, \tilde{r}, \tilde{u})}_{\mathcal{H}^{2k}}^2 
\end{equation}
and
\begin{equation}
\begin{split}
     |R_2| & \lesssim \int_{\Omega_t} r^{2k+\frac{1}{(\gamma-1)}} \frac{1}{2} |D_t(G^\alpha_\gamma G^\beta_\delta) \partial^l \tilde{\omega}^{\gamma \delta} \partial^l \tilde{\omega}_{\alpha \beta}| \hspace{1mm} dx \\
     & \lesssim  \norm{\partial u}_{L^\infty(\Omega_t)} \int_{\Omega_t} r^{\frac{2-\gamma}{\gamma-1}}  \left|r^{k+\frac{1}{2}}  \partial^l (\partial \tilde{u} + \partial \tilde{r}+ r \partial \tilde{s} + \tilde{s} + \tilde{r} + r\tilde{s})\right|^2 \hspace{1mm} dx \\
     &\lesssim \mathcal{C}_2 \norm{(\tilde{s}, \tilde{r}, \tilde{u})}_{\mathcal{H}^{2k}}^2 
\end{split}
\end{equation}
where we handled it in the same way as \eqref{E:R_0_term_vorticity_estimate}. 

For the $R_3$ integral, distributing $ \partial^l \left(\partial_\alpha u^\mu \tilde{\omega}_{\mu \beta} + \partial_\beta u^\mu \tilde{\omega}_{\alpha \mu}  \right) $ leads to terms that contain $\partial^l \tilde{\omega}$ at worst along with up to $2k$ derivatives of $u$. We can write this as
\begin{equation}
    \partial^l \left(\partial_\alpha u^\mu \tilde{\omega}_{\mu \beta} + \partial_\beta u^\mu \tilde{\omega}_{\alpha \mu}  \right) \simeq \sum_{n=0}^{|l|} \partial^n \tilde{\omega} \partial^{|l|-n+1} u 
\end{equation}
so that, upon estimating, we will have
\begin{equation}
    |R_3| \lesssim \norm{\partial^{2k} u}_{L^\infty(\Omega_t)} \int_{\Omega_t} r^{\frac{2-\gamma}{\gamma-1}}   \sum_{n=0}^{|l|} |r^{k+\frac{1}{2}} \partial^l \tilde{\omega}| |r^{k+ \frac{1}{2}}\partial^n \tilde{\omega}| \hspace{1mm} dx \\
    \lesssim \mathcal{C}_3 \norm{(\tilde{s}, \tilde{r}, \tilde{u})}_{\mathcal{H}^{2k}}^2 
\end{equation}
since the worst term in the summation only contains $2k-1$ derivatives of $\tilde{\omega}$.

For the $R_4$ integral, we have $\partial^l (\tilde{u}^\mu \partial_\mu \omega_{\alpha \beta})$ which will lead to terms with $\partial^{2k-1} \tilde{u}$ at worst and coefficients containing up to $2k+1$ derivatives of $u,r$, and $s$ coming from $\partial^l  (\partial_\mu \omega_{\alpha \beta})$. Since we have previously estimated up to $\partial^{2k} \tilde{u}$, the terms in $R_4$ will be more subcritical (and thus easier to estimate) than $R_1$ through $R_3$ which all contain $2k-1$ derivatives of $\tilde{\omega}$. We can write it as
\begin{equation}
\begin{split}
    |R_4| &\lesssim \norm{\partial^{2k+1} (u,r,s)}_{L^\infty(\Omega_t)} \int_{\Omega_t} r^{\frac{2-\gamma}{\gamma-1}}   | (r^{k+\frac{1}{2}} \partial^l \tilde{\omega}) (r^{k+\frac{1}{2}} \partial^l \tilde{u})| \hspace{1mm} dx \\
     & \lesssim  \mathcal{C}_4 \norm{(\tilde{s}, \tilde{r}, \tilde{u})}_{\mathcal{H}^{2k}}^2 
\end{split}
\end{equation}

For the $R_5$ and $R_6$ integrals, the argument is quite similar since we get up to $2k$ derivatives of $\tilde{u}$ and up to $2k$ derivatives of $\omega$. 

For the $R_7$, integral, we will have terms containing up to $2k-1$ derivatives of $\tilde{s}$ and $2k$ derivatives of $r$ and $s$. However, we know from computing $\mathcal{H}^{2k}$ orders that $\partial^{2k-1} \tilde{s}$ has the same order as $\partial^{2k-1} \tilde{u}$, so it can be estimated in the same way as $R_4$.

Finally, for the $R_8$ integral,  we will have terms containing up to $2k$ derivatives of $\tilde{s}$ and $\tilde{r}$, and up to $2k$ derivatives of $s$ and $r$. However, we know that the order of these terms will be comparable to $\partial^{2k} \tilde{u}$ at worst (and in fact, terms with $\partial^{2k} \tilde{r}$ are more subcritical by an order value of $1/2$). 

Thus, by estimating each of the terms in \eqref{E:vorticity_estimates_integral} separately and removing expressions involving $L^\infty$ norms of up to $2k+1$ derivatives of $s, r,$ and $u$, we combine the above estimates with \eqref{E:vorticity_estimates_integral} to get
\begin{equation}
    \frac{d}{dt} \int_{\Omega_t} r^{\frac{2-\gamma}{\gamma-1}} |r^{k+\frac{1}{2}} \partial^l \tilde{\omega}|_G^2 \hspace{1mm} dx \lesssim \mathcal{C} \norm{(\tilde{s}, \tilde{r}, \tilde{u})}_{\mathcal{H}^{2k}}^2 
\end{equation}
where $\mathcal{C}$ depends on up to $2k+1$ derivatives of $s, r,$ and $u$. Summing over the multiindex $l$ achieves the desired vorticity estimate using the $\mathcal{H}^{2k}$ norm.

\end{proof}
Lemma \ref{L:Estimate_full_lin_vorticity} leads nicely into the next theorem, and we note how this theorem compares with the transport energy defined in \eqref{E:E^2k_transport}.

\begin{theorem}[Linearized vorticity, Transport energy estimates]
\label{Th:Transport_Energy_estimates}
The following estimate holds for the linearized vorticity defined in \eqref{E:full_linearized_vorticity_definition}. 
\begin{equation}
 \norm{\hat{\omega}}_{H^{2k-1, k+ \frac{1}{2(\gamma-1)}}(\Omega_t)}^2 \lesssim  \norm{\tilde{\omega}_0}_{H^{2k-1, k+ \frac{1}{2(\gamma-1)}} (\Omega_0)}^2 + \varepsilon \norm{(\tilde{s}, \tilde{r}, \tilde{u})}_{\mathcal{H}^{2k}(\Omega_t) }^2 + \int_0^t \mathcal{C} \norm{(\tilde{s}, \tilde{r}, \tilde{u})}_{\mathcal{H}^{2k}(\Omega_\tau)}^2 \hspace{1mm} d \tau 
\end{equation}
where $\tilde{\omega}_0$ is the initial linearized vorticity and $\mathcal{C}$ is from Lemma \ref{L:Estimate_full_lin_vorticity}.
\end{theorem}
\begin{proof}
We plan to use Lemma \ref{L:Estimate_full_lin_vorticity} and the following identity $\hat{\omega}= \tilde{\omega} - \overline{\omega}$ to split the estimate:
\begin{equation}
    \norm{\hat{\omega}}_{H^{2k-1, k+ \frac{1}{2(\gamma-1)}}(\Omega_t)}^2 \lesssim  \norm{\tilde{\omega}}_{H^{2k-1, k+ \frac{1}{2(\gamma-1)}}(\Omega_t)}^2 + \norm{\overline{\omega}}_{H^{2k-1, k+ \frac{1}{2(\gamma-1)}}(\Omega_t)}^2
\end{equation}
Now, from the definition of $\overline{\omega}$ and the simplifications outlined in \eqref{E:vorticity_facts_useful}, we get the following when applying a multiindex $l$ with $|l|\leq 2k-1$:
\begin{equation}
\begin{split}
     \partial^l \overline{\omega} &\simeq \partial^l (\partial (\tilde{r} + r \tilde{s}) + \tilde{r} + r \tilde{s}) \\
     &\simeq \partial^{l+1} \tilde{r} + r\partial^{l+1} \tilde{s} + \partial^l \tilde{s}
\end{split}
\end{equation}
where the additional terms are even more subcritical than the ones listed. Then, using the $L^2(r^{2k+\frac{1}{(\gamma-1)}})$ norm, we have 
\begin{equation}
\begin{split}
     \int_{\Omega_t} r^{\frac{2-\gamma}{\gamma-1}}|r^{k+\frac{1}{2}}\partial^l \overline{\omega}|_G^2 \hspace{1mm} dx & \lesssim \int_{\Omega_t} r^{\frac{2-\gamma}{\gamma-1}} |r^{k+\frac{1}{2}} \partial^{l+1} \tilde{r}|^2 \hspace{1mm} dx +  \int_{\Omega_t} r^{\frac{2-\gamma}{\gamma-1}} |r^{k+\frac{1}{2}+1} \partial^{l+1} \tilde{s}|^2 \hspace{1mm} dx \\
     & \hspace{5mm} + \int_{\Omega_t} r^{\frac{2-\gamma}{\gamma-1}} |r^{k+\frac{1}{2}} \partial^{l} \tilde{s}|^2 \hspace{1mm} dx  \\
     &\lesssim \hat{\varepsilon} \int_{\Omega_t} r^{\frac{2-\gamma}{\gamma-1}} |r^{k} \partial^{l+1} \tilde{r}|^2 \hspace{1mm} dx +  \hat{\varepsilon} \int_{\Omega_t} r^{\frac{2-\gamma}{\gamma-1}} |r^{k+\frac{1}{2}} \partial^{l+1} \tilde{s}|^2 \hspace{1mm} dx \\
     & \hspace{5mm} + \hat{\varepsilon} \int_{\Omega_t} r^{\frac{2-\gamma}{\gamma-1}} |r^{k-1+\frac{1}{2}} \partial^{l} \tilde{s}|^2 \hspace{1mm} dx  \\
\end{split}
\end{equation}
where we used the smallness of $r$ from Remark \ref{R:Omega_t_near_free_boundary}.
Summing over $|l| \leq 2k-1$ and counting the order for each term (where the worst terms take the form $r^{k} \partial^{2k} \tilde{r}, r^{k+\frac{1}{2}} \partial^{2k} \tilde{u}$, and $r^{k+\frac{1}{2} - 1} \partial^{2k-1} \tilde{s}$), we get the estimate
\begin{equation}
\label{E:hat_omega_estimate_by_H2k}
\begin{split}
     \norm{\overline{\omega}}_{H^{2k-1, k+ \frac{1}{2(\gamma-1)}}}^2
     &\lesssim \hat{\varepsilon} \norm{(\tilde{s}, \tilde{r}, \tilde{u})}_{\mathcal{H}^{2k}}^2
\end{split}
\end{equation}

By integrating Lemma \ref{L:Estimate_full_lin_vorticity}, we have
\begin{equation}
    \norm{\tilde{\omega}}_{H^{2k-1, k+ \frac{1}{2(\gamma-1)}} (\Omega_t)}^2 \lesssim  \norm{\tilde{\omega}_0}^2_{H^{2k-1, k+ \frac{1}{2(\gamma-1)}} (\Omega_0)} + \int_0^t \mathcal{C} \norm{(\tilde{s}, \tilde{r}, \tilde{u})}_{\mathcal{H}^{2k}(\Omega_\tau)}^2 \hspace{1mm} d \tau
\end{equation}
where $\tilde{\omega}_0$ can be solved for using the initial data for the linearized equation $(\tilde{s}_0, \tilde{r}_0, \tilde{u}_0)$. Combining the above equations together yields the desired result.
\end{proof}

\begin{remark}
    Theorem \ref{Th:Transport_Energy_estimates} will be used in our main result, Theorem \ref{Th:Main_Theorem_Estimates_in_H^2k}.
\end{remark}

\section{Elliptic and Div-Curl Estimates}
\label{S:Elliptic estimates}

Our goal is to bound the total energy \eqref{E:E^2k_full_linearized_energy} from above and below by the $\mathcal{H}^{2k}$ norm \eqref{E:Higher_order_H^2k_norm}. In this section, we focus on bounding \eqref{E:E^2k_full_linearized_energy} from below, and one of the key ingredients for the wave part \eqref{E:E^2k_wave} is elliptic estimates involving $\tilde{r}$ and the spatial divergence of $\tilde{u}$. The proofs in this section involve isolating a ``good'' spatial elliptic operator similar to \cite{DisconziIfrimTataru}, with the added difficulty in that we must explicitly split off time derivatives while isolating the critical terms and using Section \ref{S:Sovling_for_time_derivatives}. Additionally, we will often make use of Lemma \ref{L:Book-keeping} where identities with $D_t$ serve as another method for counting order.

Our analysis will focus on isolating the ``good'' spatial elliptic operators, as one can readily compare that the structure of these operators closely resembles the elliptic operators already treated in \cite{DisconziIfrimTataru}.

\subsection{Elliptic estimates for \texorpdfstring{$\tilde{r}$}{r}}

Upon taking $D_t$ of equation \eqref{E:Linearized_System_2_r_equation} and applying Theorem \ref{L:Book-keeping} as needed, we can rewrite $D_t^2 \tilde{r}$ in the following way:

\begin{equation}
\label{E:r_Elliptic_start}
    D_t^2 \tilde{r} \simeq L_1 \tilde{r} + \text{ additional terms}
\end{equation} 
with
\begin{equation}
\label{E:L_1_definition}
\begin{split}
     L_1 \tilde{r} &:= \frac{\gamma-1}{\Gamma + r} \proj^{\mu \nu} (r  \partial_\mu \partial_\nu \tilde{r} + \frac{1}{\gamma-1} \partial_\mu r \partial_\nu \tilde{r}) \\
     &= \frac{\gamma-1}{\Gamma + r} \proj^{i j} (r  \partial_i \partial_j \tilde{r} + \frac{1}{\gamma-1} \partial_i r \partial_j \tilde{r}) \\
     & \hspace{7mm} + \frac{\gamma-1}{\Gamma+r} \proj^{00}(r  \partial_t^2 \tilde{r} + \frac{1}{\gamma-1}\partial_t r \partial_t \tilde{r}) \\
      & \hspace{7mm} + \frac{2(\gamma-1)}{\Gamma+r} \proj^{i0}(r  \partial_i \partial_t \tilde{r})\\
     & \hspace{7mm} + \frac{1}{\Gamma+r} \proj^{0j}(\partial_t r \partial_j \tilde{r}) \\
     & \hspace{7mm} + \frac{1}{\Gamma+r} \proj^{i0}(\partial_i r \partial_t \tilde{r}) \\
\end{split}
\end{equation}
serving as the ``spacetime'' elliptic operator for $\tilde{r}$, and we have split the time derivatives from the spatial derivatives.
To see how the additional terms in \eqref{E:r_Elliptic_start} will be treated, as well as the higher order elliptic estimates, we refer the reader to section \ref{S:Subsubsection_initial commutators} where these additional terms are shown to be subcritical using our book-keeping scheme from Section \ref{S:Creating the Book-keeping scheme}. 

It is important to note that since we are estimating the wave energy from below by the $\mathcal{H}^{2k}$ norm, we cannot simply cite section \ref{S:Sovling_for_time_derivatives} for the time derivatives since the structure of the terms is quite important. Instead, we will make use of the following additional facts from the definition of $D_t$:
\begin{equation}
\label{E:Elliptic_start_D_t_facts}
\begin{split}
    \partial_t \tilde{r} &= \frac{1}{u^0} D_t \tilde{r} - \frac{u^i}{u^0} \partial_i \tilde{r} \\
    r \partial_t^2 \tilde{r} &= \frac{1}{(u^0)^2} r D_t^2 \tilde{r} - \frac{u^i u^j}{(u^0)^2} r \partial_i \partial_j \tilde{r} - 2 \frac{u^i}{u^0} r \partial_i \partial_t \tilde{r} - \frac{1}{(u^0)^2} D_t u^0 r \partial_t \tilde{r} - \frac{1}{(u^0)^2} D_t u^i r \partial_i \tilde{r} \\
    &= \frac{1}{(u^0)^2} r D_t^2 \tilde{r} + \textcolor{red}{\frac{u^i u^j}{(u^0)^2} r \partial_i \partial_j \tilde{r}} - 2 \frac{u^i}{(u^0)^2} r \partial_i (D_t \tilde{r}) + 2 \frac{u^i}{u^0} \partial_i \frac{u^j}{u^0} r \partial_j \tilde{r}   - \frac{1}{(u^0)^2} D_t u^0 r \partial_t \tilde{r} - \frac{1}{(u^0)^2} D_t u^i r \partial_i \tilde{r} \\
\end{split}
\end{equation}
where, by Remark \ref{R:order_free_boundary_term}, the only critical terms (for $k=1$) are in \textcolor{red}{red}. For the remaining terms, we can write
\begin{equation}
\begin{split}
    \frac{1}{\Gamma+r} \proj^{00} \partial_t r \partial_t \tilde{r} &=  \frac{1}{\Gamma+r} \proj^{00} \partial_t r \left(\frac{1}{u^0} D_t \tilde{r} - \frac{u^i}{u^0} \partial_i \tilde{r} \right) \\
    &= \frac{1}{(\Gamma+r)u^0} \proj^{00} \partial_t r D_t \tilde{r} \textcolor{red}{- \frac{1}{(\Gamma+r)u^0} \proj^{00} \partial_t r u^i \partial_i \tilde{r}}
\end{split}
\end{equation}
and 
\begin{equation}
\begin{split}
    \frac{2(\gamma-1)}{\Gamma+r} \proj^{i0}(r  \partial_i \partial_t \tilde{r}) &=     \frac{2(\gamma-1)}{\Gamma+r} u^i u^0 r \partial_i \left( \frac{1}{u^0} D_t \tilde{r} - \frac{u^j}{u^0} \partial_j \tilde{r}\right) \\
    &= \frac{2(\gamma-1)}{\Gamma+r} r u^i  \partial_i (D_t \tilde{r}) \textcolor{red}{- \frac{2(\gamma-1)}{\Gamma+r} r u^i u^j \partial_i \partial_j \tilde{r}} \\
    & \hspace{7mm} + \frac{2(\gamma-1)}{\Gamma+r} u^i u^0 r \left[D_t\tilde{r} \partial_i \frac{1}{u^0} - \partial_i \frac{u^j}{u^0} \partial_j \tilde{r} \right].
\end{split}
\end{equation}
and
\begin{equation}
\begin{split}
     \frac{1}{\Gamma+r} \proj^{i0}(\partial_i r \partial_t \tilde{r}) &=  \frac{1}{\Gamma+r} \proj^{i0}\partial_i r \left( \frac{1}{u^0} D_t \tilde{r} - \frac{u^j}{u^0} \partial_j \tilde{r}\right) \\
     &= \frac{1}{(\Gamma+r)u^0} \proj^{i0} \partial_i r D_t \tilde{r} \textcolor{red}{- \frac{1}{(\Gamma+r)u^0} \proj^{i0} \partial_i r u^j \partial_j \tilde{r}}.
\end{split}
\end{equation}
So, if we gather up all the \textcolor{red}{red} critical terms (which look like $r \partial^2 \tilde{r}$ or $\partial \tilde{r}$), we will get precisely, 
\begin{equation}
    \begin{split}
        L_1 \tilde{r} &= \textcolor{red}{\frac{\gamma-1}{\Gamma+r} \left(\proj^{ij} + \frac{((u^0)^2-1)u^i u^j}{(u^0)^2} - 2u^i u^j \right)  r \partial_i \partial_j \tilde{r}} \\
        & \hspace{5mm} + \textcolor{red}{ \frac{1}{\Gamma+r} g^{ij} \partial_i r \partial_j \tilde{r} + \frac{1}{\Gamma+r} \left(u^0 u^j - \frac{((u^0)^2-1)}{u^0} u^j \right) \partial_t r \partial_j \tilde{r}   
        } \\
        &\hspace{5mm} + \frac{\gamma-1}{\Gamma+r}  \left(\frac{1}{(u^0)^2} r D_t^2 \tilde{r} - 2 \frac{u^i}{(u^0)^2} r \partial_i (D_t \tilde{r}) + 2 \frac{u^i}{u^0} \partial_i \frac{u^j}{u^0} r \partial_j \tilde{r}   - \frac{1}{(u^0)^2} D_t u^0 r \partial_t \tilde{r} - \frac{1}{(u^0)^2} D_t u^i r \partial_i \tilde{r} \right) \\
        & \hspace{5mm} + \frac{1}{(\Gamma+r)u^0} \proj^{00} \partial_t r D_t \tilde{r} \\
        & \hspace{5mm} + \frac{2(\gamma-1)}{\Gamma+r} r u^i  \partial_i (D_t \tilde{r})  + \frac{2(\gamma-1)}{\Gamma+r} u^i u^0 r \left[D_t\tilde{r} \partial_i \frac{1}{u^0} - \partial_i \frac{u^j}{u^0} \partial_j \tilde{r} \right] \\
        & \hspace{5mm} + \frac{1}{(\Gamma+r)u^0} \proj^{i0} \partial_i r D_t \tilde{r}\\
    \end{split}
\end{equation}
and we can check that each of the black terms is subcritical with the help of Lemma \ref{L:Book-keeping} and Remark \ref{R:order_free_boundary_term}. For a discussion of how these subcritical terms are handled, see section \ref{S:Subsubsection_initial commutators} below. Finally, we use \eqref{E:Elliptic_start_D_t_facts} to simplify $\partial_t r$ in the red critical terms (which splits into a subcritical and a critical term). Combining, we end up with 
\begin{equation}
    \begin{split}
        L_1 \tilde{r} &= \textcolor{red}{\frac{\gamma-1}{\Gamma+r} H^{ij} (r \partial_i \partial_j \tilde{r} + \frac{1}{\gamma-1} \partial_i r \partial_j \tilde{r})} \\
        & \hspace{5mm} + \left(\text{black subcritical terms} \right)
    \end{split}
\end{equation}
where
\begin{equation}
\label{E:H^ij_definition}
    H^{ij} := \delta^{ij} - \frac{u^i u^j}{(u^0)^2},
\end{equation}

If we set the ``good'' spatial elliptic operator to be
\begin{equation}
\label{E:tilde_L_1_definition}
    \textcolor{red}{\tilde{L}_1 \tilde{r} := \frac{\gamma-1}{\Gamma+r} H^{ij} (r \partial_i \partial_j \tilde{r} + \frac{1}{\gamma-1} \partial_i r \partial_j \tilde{r})}
\end{equation}
then, we get the expression:

\begin{equation}
\label{E:L_1_simplified_expression}
\boxed{
        L_1 \tilde{r} = \textcolor{red}{\tilde{L}_1 \tilde{r}} + \left(\text{black subcritical terms} \right)
    }
\end{equation}
where the term in red is critical, and all the black sub-critical terms are given by:
\begin{equation*}
\begin{split}
    &= \frac{1}{\Gamma+r} \frac{1}{(u^0)^2} D_t r u^j \partial_j \tilde{r} \\
     &\hspace{5mm} + \frac{\gamma-1}{\Gamma+r}  \left(\frac{1}{(u^0)^2} r D_t^2 \tilde{r} - 2 \frac{u^i}{(u^0)^2} r \partial_i (D_t \tilde{r}) + 2 \frac{u^i}{u^0} \partial_i \frac{u^j}{u^0} r \partial_j \tilde{r}   - \frac{1}{(u^0)^2} D_t u^0 r \partial_t \tilde{r} - \frac{1}{(u^0)^2} D_t u^i r \partial_i \tilde{r} \right) \\
        & \hspace{5mm} + \frac{1}{(\Gamma+r)u^0} \proj^{00} \partial_t r D_t \tilde{r} \\
        & \hspace{5mm} + \frac{2(\gamma-1)}{\Gamma+r} r u^i  \partial_i (D_t \tilde{r})  + \frac{2(\gamma-1)}{\Gamma+r} u^i u^0 r \left[D_t\tilde{r} \partial_i \frac{1}{u^0} - \partial_i \frac{u^j}{u^0} \partial_j \tilde{r} \right] \\
        & \hspace{5mm} + \frac{1}{(\Gamma+r)u^0} \proj^{i0} \partial_i r D_t \tilde{r}\\
\end{split}
\end{equation*}
under the book-keeping scheme of Section \ref{S:Creating the Book-keeping scheme} (where we can use Remark \ref{R:order_free_boundary_term} and Lemma \ref{L:Book-keeping} as needed to see that these terms have order $\mathcal{O} = \frac{1}{2}$ at worst). 
See section \ref{S:Subsubsection_initial commutators} below for how these terms are handled. 

\begin{lemma}[Ellptic estimates for $\tilde{r}$]
\label{L:Elliptic_estimates_tilde_r}
The following estimates hold for $\tilde{r}$ and the ``good'' spatial elliptic part $\tilde{L}_1 \tilde{r}$:
\begin{equation}
\label{E:Elliptic_estimate_for_tilde_r}
    \norm{\tilde{r}}_{H^{2, \frac{1}{2(\gamma-1)} + \frac{1}{2}}} \lesssim  \norm{\tilde{L}_1 \tilde{r}}_{H^{0, \frac{1}{2(\gamma-1)} - \frac{1}{2}}} + \norm{\tilde{r}}_{L^2(r^{\frac{2-\gamma}{\gamma-1}})}
\end{equation}
\end{lemma}
\begin{proof}
    After examining the structure of $\tilde{L}_1 \tilde{r}$, we see that the proof will closely follow \cite{DisconziIfrimTataru}  by proving two related inequalities:
    \begin{equation}
    \label{E:Inequality_goals_elliptic_r}
        \begin{split}
             \norm{\tilde{r}}_{H^{2, \frac{1}{2(\gamma-1)} + \frac{1}{2}}} &\lesssim  \norm{\tilde{L}_1 \tilde{r}}_{H^{0, \frac{1}{2(\gamma-1)} - \frac{1}{2}}} + \norm{\tilde{r}}_{H^{1, \frac{1}{2(\gamma-1)} - \frac{1}{2}}} \\
             \norm{\tilde{r}}_{H^{1, \frac{1}{2(\gamma-1)} - \frac{1}{2}}} &\lesssim  \norm{\tilde{L}_1 \tilde{r}}_{H^{0, \frac{1}{2(\gamma-1)} - \frac{1}{2}}} + \norm{\tilde{r}}_{L^2(r^{\frac{2-\gamma}{\gamma-1}})} \\
        \end{split}
    \end{equation}

Instead of repeating the proof of Lemma 5.3 in \cite{DisconziIfrimTataru}, we illustrate a few of the primary differences.
When derivatives hit the weight $\frac{1}{\Gamma+r}$, we see by \eqref{E:Gamma_definition} that
\begin{equation}
\label{E:Derivative_of_weight}
    \partial_k \left(\frac{1}{\Gamma+r} \right) = \frac{1}{(\Gamma+r)^2} \left(\frac{\gamma-1}{\gamma}\Gamma \partial_k s - \partial_k r \right)
\end{equation}
which only contains derivatives of $r$ and $s$. Similarly, when derivatives hit the metric $H^{ij}$, we get an expression that only contains derivatives of $u$ (and this metric also appears in \cite{DisconziIfrimTataru}). 

To show the next inequality in \eqref{E:Inequality_goals_elliptic_r}, we use a similar method as the proof of Corollary 5.5 in \cite{DisconziIfrimTataru}.
\end{proof}

\subsection{Elliptic estimates for \texorpdfstring{$\tilde{u}$}{u}}

Next, we will handle the $\tilde{u}$ equation. From \eqref{E:Linearized_System_6_v_equation}, we note that $D_t^2 \tilde{u}^\alpha$ will satisfy an equation involving one derivative of the spacetime divergence $\partial_\nu \tilde{u}^\nu$ (see \eqref{E:L_2_definition} for the definition of operator $L_2$). In view of Section \ref{S:Sovling_for_time_derivatives}, we would like to solve for time derivatives in terms of spatial derivatives and relate our analysis back to the spatial divergence of $\tilde{u}$ which we will denote by
\begin{equation}
\label{E:Def_spatial_div_u_tilde}
    \overrightarrow{\text{div }} \tilde{u} := \partial_j \tilde{u}^j
\end{equation}
In order to estimate our energies from below with the $\mathcal{H}^{2k}$ norm, we will need to do elliptic estimates involving $\overrightarrow{\text{div }} \tilde{u}$.
After extracting the spatial part, we will arrive at the ``good'' operator $\tilde{L}_2$ (see \eqref{E:good_L_2_definition}) which is properly suited for div-curl estimates. To complete the estimates, we will need to examine the spatial curl whose components are given by:
\begin{equation}
\label{E:Def_spatial_curl_u_tilde}
     (\overrightarrow{\text{curl }} \tilde{u})_{ij} := \partial_i \tilde{u}_j - \partial_j \tilde{u}_i.
\end{equation}
We plan to pair $\tilde{L}_2$ with the corresponding operator $\tilde{L}_3$ (see \eqref{E:L_3_definition} below), and we note that $\tilde{L}_3$ is roughly one spatial derivative of $\overrightarrow{\text{curl }} \tilde{u}$ with a weight $r$.

Our goal is Lemma \ref{L:Elliptic_div_curl_estimates_tilde_u}, but we will begin by highlighting key steps in the following computation to show how we correctly isolate the spatial components. Similar to the $\tilde{r}$ equation, we start by isolating the spacetime elliptic piece of $D_t^2 \tilde{u}^\alpha$ which will produce the following upon taking $D_t$ of \eqref{E:Linearized_System_6_v_equation}:
\begin{equation}
\label{E:Elliptic_u_equation}
        D_t^2 \tilde{u}^\alpha  \simeq \left(L_2 \tilde{u}\right)^\alpha + \text{ additional terms}
\end{equation}
where
\begin{equation}
\label{E:L_2_definition}
    \begin{split}
       \left(L_2 \tilde{u}\right)^\alpha &:= \frac{\gamma-1}{\Gamma+r} \proj^{\alpha \mu} \left(\partial_\mu (r \partial_\nu \tilde{u}^\nu) + \frac{1}{\gamma-1} \partial_\mu \tilde{u}^\nu \partial_\nu r\right) \\
        &\hspace{1mm}= \frac{\gamma-1}{\Gamma+r} \proj^{\alpha i} \left(\partial_i (r \partial_\nu \tilde{u}^\nu) + \frac{1}{\gamma-1} \partial_i \tilde{u}^\nu \partial_\nu r \right) +\frac{\gamma-1}{\Gamma+r} \proj^{\alpha 0} \left(\partial_t (r \partial_\nu \tilde{u}^\nu) + \frac{1}{\gamma-1} \partial_t \tilde{u}^\nu \partial_\nu r\right)  \\
    \end{split}
\end{equation}
and the additional terms will be shown to be subcritical (see section \ref{S:Subsubsection_initial commutators} for a discussion of these terms). Then, solving for
\begin{equation}
\label{E:solving for partial_t with partial_i}
    \partial_t = \frac{1}{u^0} D_t - \frac{u^i}{u^0} \partial_i
\end{equation}
 and gathering up similar terms, we have the useful identity
\begin{equation}
\begin{split}
    \proj^{\alpha \mu} \partial_\mu  &= \proj^{\alpha i} \partial_i + \proj^{\alpha 0} \partial_t \\
    &:= B^{\alpha i} \partial_i + \proj^{\alpha 0} \frac{1}{u^0} D_t
\end{split}
\end{equation}
where we defined:
\begin{equation}
    B^{\alpha i}:= g^{\alpha i} - g^{\alpha 0} \frac{u^i}{u^0}.
\end{equation}
Further, note the explicit identity
\begin{equation}
\label{E:G_B_interact}
    G_{\alpha \beta} B^{\alpha i} B^{\beta j} = H^{ij}.
\end{equation}
where $G_{\alpha \beta}$ is defined in \eqref{E:G_metric_definition}, and $H^{ij}$ is defined in \eqref{E:H^ij_definition}.

Using these identities, we arrive at the equivalent expression
\begin{equation}
\label{E:L_2_u_alpha_expression}
   \left(L_2 \tilde{u}\right)^\alpha = \textcolor{red}{\frac{\gamma-1}{\Gamma+r} B^{\alpha i} \left(\partial_i (r \partial_\nu \tilde{u}^\nu) + \frac{1}{\gamma-1} \partial_i \tilde{u}^\nu \partial_\nu r \right)} +\frac{\gamma-1}{\Gamma+r} \proj^{\alpha 0} \frac{1}{u^0} \left(D_t (r \partial_\nu \tilde{u}^\nu) + \frac{1}{\gamma-1} D_t \tilde{u}^\nu \partial_\nu r\right).
\end{equation}

\begin{remark}
\label{R:L_2_black_critical_and_subcritical}
    If we suppose that $k=1$ for computing order (see Remark \ref{R:order_free_boundary_term}), note that the terms in \textcolor{red}{red} have order $-\frac{1}{2}$ which are supercritical, but the terms in black have order $0$ which is a strictly better order. In the context of estimating $L_2 \tilde{u}^\alpha$, we will use the norm $H^{0,\frac{1}{2}}$ in which case we will have an extra $1/2$ power of $r$ in order to safely move these black terms to the other side. Thus, the terms in black can be treated in a very similar way as \eqref{E:L_1_simplified_expression} in the $\tilde{r}$ elliptic estimates.
\end{remark}

At this stage, we still have an expression for $L_2 \tilde{u}^\alpha$ involving the spacetime divergence $\partial_\nu \tilde{u}^\nu$. To further isolate the spatial divergence, we make use of the following identities by repeatedly splitting and removing time derivatives as needed. We apply \eqref{E:tilde_u_^0 identity} and \eqref{E:solving for partial_t with partial_i} to get
\begin{equation}
\begin{split}
    \partial_\nu \tilde{u}^\nu  &= \textcolor{red}{H^{jk} \partial_j \tilde{u}_k} + \frac{u^j}{(u^0)^2} D_t \tilde{u}_j + \tilde{u}^j \partial_t \frac{u_j}{u^0}.
\end{split}
\end{equation}
Similarly, we can apply the same identities such as \eqref{E:tilde_u_^0 identity} to replace any instances of $\tilde{u}^0$ with a spatial component in the expression $\partial_i \tilde{u}^\nu \partial_\nu r$. Splitting the sum over $\nu$ and solving for $\partial_t r$ using \eqref{E:solving for partial_t with partial_i}, we arrive at

\begin{equation}
    \begin{split}
        \partial_i \tilde{u}^\nu \partial_\nu r &= \textcolor{red}{\partial_i \tilde{u}^k \partial_k r} + \left(\textcolor{red}{- \frac{u_l u^m}{(u^0)^2} \partial_i \tilde{u}^l \partial_m r} + \frac{u_l}{(u^0)^2} \partial_i \tilde{u}^l D_t r - \frac{\tilde{u}^l u^m}{u^0} \partial_i \frac{u_l}{u^0} \partial_m r + \frac{\tilde{u}^l}{u^0} \partial_i \frac{u_l}{u^0} D_t r \right) \\
        &= \textcolor{red}{H^{km}\partial_i \tilde{u}_k \partial_m r } + \left(\frac{u_l}{(u^0)^2} \partial_i \tilde{u}^l D_t r - \frac{\tilde{u}^l u^m}{u^0} \partial_i \frac{u_l}{u^0} \partial_m r + \frac{\tilde{u}^l}{u^0} \partial_i \frac{u_l}{u^0} D_t r \right)
    \end{split}
\end{equation}
Plugging these identities into \eqref{E:L_2_u_alpha_expression}, we have
\begin{equation}
\begin{split}
   \left(L_2 \tilde{u}\right)^\alpha &= \textcolor{red}{\frac{\gamma-1}{\Gamma+r} B^{\alpha i} H^{jk} \left(\partial_i(r \partial_j \tilde{u}_k) + \frac{1}{\gamma-1} \partial_j r \partial_i \tilde{u}_k \right)} \\
    & \hspace{5mm} + \frac{\gamma-1}{\Gamma+r} B^{\alpha i} \partial_i \left(\frac{u^j}{(u^0)^2} r D_t \tilde{u}_j + r \tilde{u}^j \partial_t \frac{u_j}{u^0} \right) \\
    & \hspace{5mm} + \frac{1}{\Gamma+r} B^{\alpha i} \left(\frac{u_l}{(u^0)^2} \partial_i \tilde{u}^l D_t r - \frac{\tilde{u}^l u^m}{u^0} \partial_i \frac{u_l}{u^0} \partial_m r + \frac{\tilde{u}^l}{u^0} \partial_i \frac{u_l}{u^0} D_t r \right).
\end{split}
\end{equation}
If we define the ``good'' spatial elliptic part for the divergence of $\tilde{u}$ to be 
\begin{equation}
    \label{E:good_L_2_definition}
    \textcolor{red}{
     \left(\tilde{L}_2 \tilde{u} \right)^\alpha := \frac{\gamma-1}{\Gamma+r} B^{\alpha i} H^{jk} \left(\partial_i(r \partial_j \tilde{u}_k) + \frac{1}{\gamma-1} \partial_j r \partial_i \tilde{u}_k \right)},
\end{equation}
and observe Remark \ref{R:L_2_black_critical_and_subcritical},
we get the nice expression
\begin{equation}
\label{E:L_2_simplified_expression}
\boxed{
    \left(L_2 \tilde{u}\right)^\alpha = \textcolor{red}{ \left(\tilde{L}_2 \tilde{u} \right)^\alpha} + (\text{black critical and subcritical terms})
}
\end{equation}

Next, we plan to combine the $ \left(\tilde{L}_2 \tilde{u} \right)^\alpha$ with a corresponding curl term in order to prove the desired div-curl estimates. Consider the following
\begin{equation}
\label{E:L_3_definition}
     \left(\tilde{L}_3 \tilde{u} \right)^\alpha :=  \frac{\gamma-1}{\Gamma+r} B^{\alpha i} r^{- \frac{1}{\gamma-1}} H^{ml} \partial_l \left(r^{1 + \frac{1}{\gamma-1}} (\partial_m \tilde{u}_i - \partial_i \tilde{u}_m) \right)
\end{equation}

\begin{lemma}[Div-Curl estimates for $\tilde{u}$]
\label{L:Elliptic_div_curl_estimates_tilde_u}
The following estimates hold for $\tilde{u}$ and the ``good'' spatial elliptic parts $ \left(\tilde{L}_2 \tilde{u} \right)^\alpha$ and $ \left(\tilde{L}_3 \tilde{u} \right)^\alpha$:
\begin{equation}
\label{E:div_curl_tilde_u_estimate}
    \norm{\tilde{u}}_{H^{2, \frac{1}{2(\gamma-1)} + 1}} \lesssim  \norm{(\tilde{L}_2 +\tilde{L}_3 )\tilde{u}}_{H^{0, \frac{1}{2(\gamma-1)} }} + \norm{\tilde{u}}_{L^2(r^{\frac{1}{\gamma-1}}) }
\end{equation}
    
\end{lemma}

\begin{proof}
The proof is completed in two steps similar to Lemma \ref{L:Elliptic_estimates_tilde_r}, and it closely follows Lemma 5.3 in \cite{DisconziIfrimTataru} where one should compare with our definitions of $\tilde{L}_2$ and $\tilde{L}_3$. Addtionally, we make explicit use of identity \eqref{E:tilde_u_^0 identity} since $\tilde{u}^\alpha u_\alpha =0$, and we refer the reader to Remark \ref{R:estimating_time_derivatives}

We start by writing: 
    \begin{equation}
\label{E:div_curl_start}
    \begin{split}
        \norm{(\tilde{L}_2 +\tilde{L}_3 )\tilde{u}}_{H^{0, \frac{1}{2(\gamma-1)} }}^2 &= \int_{\Omega_t} r^{\frac{1}{\gamma-1}} G_{\alpha \beta} [(\tilde{L}_2 + \tilde{L}_3) \tilde{u} ]^\alpha [(\tilde{L}_2 + \tilde{L}_3) \tilde{u} ]^\beta \hspace{1mm} dx \\
    \end{split}
    \end{equation}
Then, using identity \eqref{E:G_B_interact}, we get the following collection of terms in the integral:
\begin{equation}
    \begin{split}
        & \left(\frac{\gamma-1}{\Gamma+r}\right)^2 r^{\frac{1}{\gamma-1}} H^{ia} H^{jk} H^{ln} \left[ r \partial_i \partial_j \tilde{u}_k + \partial_i r \partial_j \tilde{u}_k + \frac{1}{\gamma-1} \partial_j r \partial_i \tilde{u}_k  +  r^{- \frac{1}{\gamma-1}} \partial_k \left(r^{1 + \frac{1}{\gamma-1}} (\partial_j \tilde{u}_i - \partial_i \tilde{u}_j) \right) \right] \\
        & \hspace{6mm} \times  \left[ r \partial_a \partial_l \tilde{u}_n + \partial_a r \partial_l \tilde{u}_n + \frac{1}{\gamma-1} \partial_l r \partial_a \tilde{u}_n +  r^{- \frac{1}{\gamma-1}} \partial_n \left(r^{1 + \frac{1}{\gamma-1}} (\partial_l \tilde{u}_a - \partial_a \tilde{u}_l) \right)\right] \\
    \end{split}
\end{equation}
which can be expanded to get the following:
\begin{equation}
\label{E:all_div_curl_terms}
    \begin{split}
        & \left(\frac{\gamma-1}{\Gamma+r}\right)^2 r^{\frac{1}{\gamma-1}} H^{ia} H^{jk} H^{ln} \\
        & \hspace{3mm} \times \left[ \textcolor{red}{r \partial_i \partial_j \tilde{u}_k} + \partial_i r \partial_j \tilde{u}_k + \frac{1}{\gamma-1} \partial_j r \partial_i \tilde{u}_k + \textcolor{red}{r\partial_k \partial_j \tilde{u}_i - r \partial_k \partial_i \tilde{u}_j} + \frac{\gamma}{\gamma-1}(\partial_k r \partial_j \tilde{u}_i - \partial_k r \partial_i \tilde{u}_j) \right] \\
        & \hspace{3mm} \times \left[\textcolor{red}{r \partial_a \partial_l \tilde{u}_n} + \partial_a r \partial_l \tilde{u}_n + \frac{1}{\gamma-1} \partial_l r \partial_a \tilde{u}_n + \textcolor{red}{r\partial_n \partial_l \tilde{u}_a - r \partial_n \partial_a \tilde{u}_l} + \frac{\gamma}{\gamma-1}(\partial_n r \partial_l \tilde{u}_a - \partial_n r \partial_a \tilde{u}_l)\right] \\
    \end{split}
\end{equation}
where each of the terms in \textcolor{red}{red} take the form $r \partial^2 \tilde{u}$. Then, we group the \textcolor{red}{red} terms and integrate by parts while using the symmetry present in the expression
\begin{equation}
\label{E:D_big_metric_definition}
    D^{ia,jk,ln} := \left(\frac{\gamma-1}{\Gamma+r}\right)^2 H^{ia} H^{jk} H^{ln}.
\end{equation}
In this form, we can apply the same ideas as \cite{DisconziIfrimTataru} where we first prove 
\begin{equation}
\label{E:div_curl_ineq_1}
     \norm{(\tilde{L}_2 +\tilde{L}_3 )\tilde{u}}_{H^{0, \frac{1}{2(\gamma-1)} }}^2 \gtrsim \norm{\tilde{u}}_{H^{2, \frac{1}{2(\gamma-1)}+1}}^2 - \norm{\tilde{u}}_{H^{1, \frac{1}{2(\gamma-1)}}}^2,
\end{equation}
and then integrate by parts with $\partial_3 \tilde{u}_a$ to prove the second inequality
\begin{equation}
\label{E:div_curl_ineq_2}
    \norm{(\tilde{L}_2 +\tilde{L}_3 )\tilde{u}}_{H^{0, \frac{1}{2(\gamma-1)} }}^2 \gtrsim \norm{\tilde{u}}_{H^{1, \frac{1}{2(\gamma-1)}}}^2 - \norm{\tilde{u}}_{L^2(r^{\frac{1}{\gamma-1}})}^2
\end{equation}
in which case \eqref{E:div_curl_ineq_1} and \eqref{E:div_curl_ineq_2} will combine to prove the desired result.

\end{proof}

For the last lemma in this section, we will ultimately need to connect the spacetime two-form $\hat{\omega}$ back to $\overrightarrow{\text{curl }} \tilde{u}$ which is a key quantity in our div-curl estimates for $\tilde{u}$.

\begin{lemma}[Relating vorticity to spatial curl]
\label{L:relating_vorticity_spatial_curl}
The following estimate holds for the reduced linearized vorticity:
\begin{equation}
\label{E:Relating vorticity to spatial curl}
    \norm{\hat{\omega}}^2_{H^{2k-1, k+ \frac{1}{2(\gamma-1)}}} \gtrsim \norm{\overrightarrow{\text{curl }} \tilde{u}}^2_{H^{2k-1, k+ \frac{1}{2(\gamma-1)}}} - \hat{\varepsilon}^2 \norm{\tilde{u}}^2_{H^{2k, k+ \frac{1}{2(\gamma-1)}}}
\end{equation}
where $\hat{\varepsilon}$ is small positive constant defined in Assumption \ref{A:Assumption_smallness r}, and $\overrightarrow{\text{curl }}  \tilde{u}$ is defined in \ref{E:Def_spatial_curl_u_tilde}.
\end{lemma}
\begin{proof}
Recalling the definition of $\hat{\omega}$ in \eqref{E:reduced_lin_vort_evolution_eq} which is a spacetime two-form, we observe the following. First, we define the \textit{spatial} reduced linearized vorticity $\overrightarrow{\hat{\omega}}$ to be a spatial two-form with components
\begin{equation*}
    \hat{\omega}_{ij} =\partial_i(h \tilde{u}_j) - \partial_j (h \tilde{u}_i) = h( \partial_i \tilde{u}_j - \partial_j \tilde{u}_i) + \tilde{u}_j \partial_i h - \tilde{u}_i \partial_j h,
\end{equation*}
Then, if we denote $\left|\cdot \right|_{\delta^{(3)}}^2$ to be the spatial Euclidean norm (squared), we have 
\begin{equation}
\left|\overrightarrow{\hat{\omega}}\right|_{\delta^{(3)}}^2:= \delta^{ij} \delta^{km} \hat{\omega}_{ik} \hat{\omega}_{jm}.
\end{equation}
For a generic multiindex $l$ with $|l| \leq 2k-1$, we claim that the following inequality holds:

\begin{equation}  
\label{E:Inequality_vorticity_extract_spatial}
\left|\partial^l \overrightarrow{\hat{\omega}}\right|_{\delta^{(3)}}^2 \lesssim \left|\partial^l \hat{\omega}\right|_G^2
\end{equation}
Indeed, since $G$ is a Riemannian metric on $\mathbb{R}^4$, it equivalent to $\delta$ on $\mathbb{R}^4$ with constants depending on $u$. This gives
\begin{equation}
    \begin{split}
        \left|\partial^l \hat{\omega}\right|_G^2 \gtrsim \left|\partial^l \hat{\omega}\right|_{\delta^{(4)}}^2  &= \delta^{\alpha \gamma} \delta^{\beta \delta}\partial^l \hat{\omega}_{\alpha \beta} \partial^l \hat{\omega}_{\gamma \delta} \\
        &= \delta^{i \gamma} \delta^{\beta \delta}\partial^l \hat{\omega}_{i \beta} \partial^l \hat{\omega}_{\gamma \delta} + \delta^{0 \gamma} \delta^{\beta \delta}\partial^l \hat{\omega}_{0 \beta} \partial^l \hat{\omega}_{\gamma \delta} \\
        &= \delta^{i j} \delta^{\beta \delta}\partial^l \hat{\omega}_{i \beta} \partial^l \hat{\omega}_{j \delta} + \textcolor{red}{\delta^{i 0} \delta^{\beta \delta}\partial^l \hat{\omega}_{i \beta} \partial^l \hat{\omega}_{0 \delta}} +  \delta^{\beta \delta}\partial^l \hat{\omega}_{0 \beta} \partial^l \hat{\omega}_{0 \delta} \\
        &= \delta^{i j} \delta^{k \delta}\partial^l \hat{\omega}_{i k} \partial^l \hat{\omega}_{j \delta} + \delta^{i j} \delta^{0 \delta}\partial^l \hat{\omega}_{i 0} \partial^l \hat{\omega}_{j \delta} +  \delta^{\beta \delta}\partial^l \hat{\omega}_{0 \beta} \partial^l \hat{\omega}_{0 \delta} \\
        &=  \delta^{i j} \delta^{k m}\partial^l \hat{\omega}_{i k} \partial^l \hat{\omega}_{j m} + \textcolor{red}{ \delta^{i j} \delta^{k 0}\partial^l \hat{\omega}_{i k} \partial^l \hat{\omega}_{j 0}} + \delta^{i j} \partial^l \hat{\omega}_{i 0} \partial^l \hat{\omega}_{j 0} +   \textcolor{purple}{\delta^{\beta \delta}\partial^l \hat{\omega}_{0 \beta} \partial^l \hat{\omega}_{0 \delta}} \\
        &= \left| \partial^l \overrightarrow{\hat{\omega}}\right|_{\delta^{(3)}}^2 + \delta^{i j} \partial^l \hat{\omega}_{i 0} \partial^l \hat{\omega}_{j 0} +   \delta^{ij}\partial^l \hat{\omega}_{0 i} \partial^l \hat{\omega}_{0 j}
    \end{split}
\end{equation}
where we split indices, the \textcolor{red}{red} terms are zero, and the \textcolor{purple}{purple} term is zero when $\beta=0$ or $\delta=0$ using the definition of $\hat{\omega}$. Then, using $\hat{\omega}_{ij} = - \hat{\omega}_{ji}$, we have 
\begin{equation}
\begin{split}
     \left| \partial^l \overrightarrow{\hat{\omega}}\right|_{\delta^{(3)}}^2 + \delta^{i j} \partial^l \hat{\omega}_{i 0} \partial^l \hat{\omega}_{j 0} +   \delta^{ij}\partial^l \hat{\omega}_{0 i} \partial^l \hat{\omega}_{0 j} &= \left| \partial^l \overrightarrow{\hat{\omega}}\right|_{\delta^{(3)}}^2 + \sum_{i=1}^3 (\partial^l \hat{\omega}_{i0})^2 \\
     &\geq  \left| \partial^l \overrightarrow{\hat{\omega}}\right|_{\delta^{(3)}}^2
\end{split}
\end{equation}
which completes the proof of \eqref{E:Inequality_vorticity_extract_spatial}. Multiplying \eqref{E:Inequality_vorticity_extract_spatial} by the weight $r^{2k + \frac{1}{\gamma-1}}$, integrating, and summing over the multiindex $l$ with $|l| \leq 2k-1$ produces the inequality
\begin{equation}
     \norm{\hat{\omega}}_{H^{2k-1, k+ \frac{1}{2(\gamma-1)}}}^2 \gtrsim  \norm{ \overrightarrow{\hat{\omega}}}_{H^{2k-1, k+ \frac{1}{2(\gamma-1)}}}^2.
\end{equation}
To finish the proof of Lemma \ref{L:relating_vorticity_spatial_curl}, we simply observe from \eqref{E:Def_spatial_curl_u_tilde} that 
\begin{equation}
    \begin{split}
        \delta^{i j} \delta^{k m}\partial^l \hat{\omega}_{i k} \partial^l \hat{\omega}_{j m} &= \delta^{i j} \delta^{k m}\partial^l \left[h (\overrightarrow{\text{ curl }} \tilde{u})_{ik} + \tilde{u}_k \partial_i h - \tilde{u}_i \partial_k h  \right] \partial^l \left[h (\overrightarrow{\text{ curl }} \tilde{u})_{jm} + \tilde{u}_m \partial_j h - \tilde{u}_j \partial_m h  \right] \\
        &\simeq \delta^{i j} \delta^{k m} (h  \partial^l (\overrightarrow{\text{ curl }} \tilde{u})_{ik} + \text{ additional terms})  \left(h \partial^l (\overrightarrow{\text{ curl }} \tilde{u})_{jm} + \text{ additional terms} \right)
    \end{split}
\end{equation}
We note that any terms where less than $|l|$ derivatives hit $\overrightarrow{\text{ curl }} \tilde{u}$ will be less critical. For each additional term, we will always have $2k-2$ or less derivatives of $\tilde{u}$, and thus we can can apply Corollary \ref{C:Trading_deriv_for_weight} to gain an extra weight and apply the standard smallness arguments with Cauchy-Schwarz as needed for cross terms (see Lemma \ref{L:Elliptic_div_curl_estimates_tilde_u} where similar cross terms are handled). 
Moreover, since $h \sim O(1)$ near the free boundary, and $\partial_i h = \frac{1}{\Gamma} \partial_i r + \frac{\gamma-1}{\gamma \Gamma} r \partial_i s$ does not contribute to the order, we can pull out derivatives or $r$ and $s$ in the $L^\infty(\Omega_t)$ norm. Thus, we have
\begin{equation}
\begin{split}
      \norm{\hat{\omega}}_{H^{2k-1, k+ \frac{1}{2(\gamma-1)}}}^2 & \gtrsim  \norm{\overrightarrow{\text{ curl }} \tilde{u} + (\text{smallness terms})}_{H^{2k-1, k+ \frac{1}{2(\gamma-1)}}}^2.
\end{split}
\end{equation}
which produces the desired inequality \eqref{E:Relating vorticity to spatial curl} after removing the terms with $\hat{\varepsilon}$.
\end{proof}

\subsection{Higher order commutators with the convective derivative}
\label{S:Subsubsection_initial commutators}

In this subsection, we plan to adapt the elliptic estimates to higher values of $k$ and show that all lower order and commutator terms can be treated using the correct weighted Sobolev norm.
First, we will prove Lemma \ref{L:Commutator_D_t_with_L_1_2_3} which handles the additional terms that appear when taking multiple convective derivatives of our system \eqref{E:Linearized_System_2} for $\tilde{r}$ and $\tilde{u}$. 
In Section \ref{S:Subsection_weighted_spatial_derivatives}, we will focus on commuting with weighted spatial derivatives where the $\tilde{r}$ equation will be discussed in detail. For $\tilde{u}$, the method is quite similar, and we state the important result in Lemma \ref{L:Higher_order_div_curl_estimates_tilde_u}.

Earlier in section \ref{S:Elliptic estimates}, recall that there were many ``additional terms'' or black subcritical terms that appeared when defining our ``good'' operators $\tilde{L}_1, \tilde{L}_2$, and $\tilde{L}_3$. 
In the following summarizing lemma, we will indicate how each of these terms is subcritical, as well as the interaction between $D_t$ and these operators so that we can eventually prove higher order elliptic estimates.

\begin{lemma}[Commutators between $D_t$ and $\tilde{L}_1, \tilde{L}_2, \tilde{L}_3$]
\label{L:Commutator_D_t_with_L_1_2_3}
Using the Book-keeping scheme of Lemma \ref{L:Book-keeping}, the following identities hold

\begin{align}
    \label{E:D_t_commutator_with_L_1}
    D_t^{2k} \tilde{r} &\simeq \tilde{L}_1 (D_t^{2k-2} \tilde{r}) + \left(\text{terms with $\mathcal{H}^{2k}$-order $\frac{1}{2}$ at worst}\right) \\
         \label{E:div_curl_simplication_1}
    D_t^{2k} \tilde{u}^\alpha &\simeq \tilde{L}_2 (D_t^{2k-2} \tilde{u}^\alpha) + \left(\text{terms with $\mathcal{H}^{2k}$-order $0$ at worst}\right) \\
    \label{E:D_t_commutator_with_L_3}
    D_t^{2k-2} (\tilde{L_3} \tilde{u})^\alpha &\simeq \tilde{L}_3 (D_t^{2k-2} \tilde{u}^\alpha) + \left(\text{terms with $\mathcal{H}^{2k}$-order $0$ at worst}\right)
\end{align}
where the order of the terms has been computed with Remark \ref{R:order_free_boundary_term}.
Moreover, $[\tilde{L_3}, D_t^{2k-2}] \tilde{u}$ contains terms which are critical/subcritical at worst.

\begin{proof}
     First, we plan to collect ``additional terms'' that appear in \eqref{E:r_Elliptic_start}. We begin by writing the structure of \eqref{E:r_Elliptic_start} in a manner similar to Lemma \ref{L:Book-keeping} where derivatives and powers of $r$ are examined. We get the following:

\begin{equation}
\begin{split}
\label{E:Higher_order_elliptic_start}
    D_t^2 \tilde{r} & \simeq L_1 \tilde{r} + r \partial(\tilde{s} + \tilde{r} + \tilde{u}) + r [D_t, \partial] \tilde{u} + \tilde{u} \\
     &\simeq L_1 \tilde{r} + \textcolor{blue}{r \partial \tilde{u} + r \partial \tilde{s} + \tilde{u}} + \textcolor{purple}{r \partial \tilde{r}}
\end{split}
\end{equation}

where $L_1$ is the operator defined in \eqref{E:L_1_definition}. If we account for the black subcritical terms that appear in \eqref{E:L_1_simplified_expression}, we will have
\begin{equation}
\label{E:proof_D_t^2_tilde_r_simplified}
     D_t^2 \tilde{r} \simeq \tilde{L}_1 \tilde{r} + ( \text{$\mathcal{H}^2$-order $\frac{1}{2}$ terms at worst}) + \textcolor{blue}{r \partial \tilde{u} + r \partial \tilde{s} + \tilde{u}} + \textcolor{purple}{r \partial \tilde{r}}
\end{equation}
If we compute the order of each of these terms at the $j=1$ level, we see that the \textcolor{blue}{blue} terms have order $\mathcal{O} = \frac{1}{2} >0$ and the \textcolor{purple}{purple} terms have order $\mathcal{O} =1 >0$. By Lemma \ref{L:Order_of_operators} and Remark \ref{R:Order_of_operators}, applying $D_t^{2k-2}$ to the \textcolor{blue}{blue} and \textcolor{purple}{purple} terms will produce corresponding subcritical terms with $\mathcal{H}^{2k}$-orders $\mathcal{O} = \frac{1}{2}$ at worst for level $k$. Thus, we will take $D_t^{2k-2}$ of \eqref{E:proof_D_t^2_tilde_r_simplified} to get
\begin{equation}
\begin{split}
    D_t^{2k} \tilde{r} &\simeq D_t^{2k-2} (\tilde{L}_1 \tilde{r}) + (\mathcal{O} = \frac{1}{2} \text{ terms at worst}) \\
    &= \tilde{L}_1 (D_t^{2k-2} \tilde{r}) + [D_t^{2k-2}, \tilde{L}_1] \tilde{r} + (\mathcal{O} = \frac{1}{2} \text{ terms at worst})
\end{split}
\end{equation}
To finish the proof of \eqref{E:D_t_commutator_with_L_1}, it remains to check the commutator $[D_t^{2k-2}, \tilde{L}_1] \tilde{r}$ which can be done with the help of Lemma \ref{L:Commutator_2} and an induction argument. We start by computing the first commutator with the help of Lemma \ref{L:Commutator_2}.
\begin{equation}
\begin{split}
     [D_t, \tilde{L}_1] \phi &= [D_t, \frac{\gamma-1}{\Gamma+r} H^{ij} r \partial_i \partial_j] \phi + [D_t, \frac{1}{\Gamma+r} H^{ij} \partial_i r \partial_j ] \phi \\
     &\simeq [D_t, r \partial^2] \phi + [D_t, \partial] \phi \\
     &\simeq r \partial^2 \phi + \partial \phi + r \partial \phi.
\end{split}
\end{equation}
recalling that $\partial$ indicates a generic spatial derivative, and we are ordering terms from the worst to the best order.
We have
\begin{equation}
\label{E:D_t^2_L_1_commutator_calculation}
\begin{split}
     [D_t^2, \tilde{L}_1] \phi &= D_t [D_t, \tilde{L}_1] \phi + [D_t, \tilde{L}_1]D_t \phi \\
        & \simeq   r \partial^2 (D_t \phi) + \partial (D_t \phi) + r \partial (D_t \phi) + r \partial^2 \phi + \partial \phi + r \partial \phi \\
\end{split}
\end{equation}
and we also note by Remark \ref{R:Order_of_operators}
and when $\phi = \tilde{r}$, we will get
\begin{equation}
\begin{split}
     [D_t^2, \tilde{L}_1] \tilde{r} &\simeq r \partial^2 (r \partial \tilde{u} + \tilde{u}) + \partial (r \partial \tilde{u} + \tilde{u}) + r \partial (r \partial \tilde{u} + \tilde{u})+ r \partial^2 \tilde{r}  + \partial \tilde{r} + r \partial \tilde{r} \\
        &\simeq (r^2 \partial^3 \tilde{u} + r \partial^2 \tilde{u} + \partial \tilde{u}) +r \partial^2 \tilde{r} + \partial \tilde{r} + r \partial \tilde{r} 
\end{split}
\end{equation}
Then since, the ``$k$'' value when computing order at this level is $2$, (i.e. from Remark \ref{R:order_free_boundary_term}) we see that each one of these terms has $\mathcal{H}^4$-order $\frac{1}{2}, 1$, and $2$ respectively (and they are all subcritical with $\mathcal{O}=\frac{1}{2}$ at worst). To complete the induction, suppose that for some $j\geq 2$, we have  $[D_t^{2j-2}, \tilde{L}_1]$ is subcritical (with order $\mathcal{O} = \frac{1}{2}$) at level $j$. Then, we compute using the commutator identities that
\begin{equation}
\label{E:Second_term_in_commutator_induction}
    [D_t^{2(j+1)-2}, \tilde{L}_1] \tilde{r} =  [D_t^{2j}, \tilde{L}_1] \tilde{r} = [D_t^2 D_t^{2j-2}, \tilde{L}_1] \tilde{r} = D_t^2 [D_t^{2j-2}, \tilde{L}_1] r + [D_t^2, \tilde{L}_1] D_t^{2j-2} \tilde{r}
\end{equation}
Since $[D_t^{2j-2}, \tilde{L}_1]$ is subcritical at level $j$, we can conclude using Remark \ref{R:Order_of_operators}  that the first term in \eqref{E:Second_term_in_commutator_induction} must have order $\mathcal{O} = \frac{1}{2}$ at the $j+1$ level. For the second term in \eqref{E:Second_term_in_commutator_induction}, it is easiest to compute the order at the $j+1$ level by counting derivatives and powers of $r$. We will have using Lemma \ref{L:Book-keeping} that
\begin{equation}
\begin{split}
     [D_t^2, \tilde{L}_1] D_t^{2j-2} \tilde{r} &\simeq r \partial^2 (D_t D_t^{2j-2} \tilde{r}) + \partial (D_t D_t^{2j-2} \tilde{r}) + r \partial (D_t D_t^{2j-2} \tilde{r}) + r \partial^2 D_t^{2j-2} \tilde{r} + \partial D_t^{2j-2} \tilde{r} + r \partial D_t^{2j-2} \tilde{r} \\
     &\simeq r \partial^2 (r^j \partial^{2j-1} \tilde{u}) + \partial (r^j \partial^{2j-1} \tilde{u}) + r \partial (r^j \partial^{2j-1} \tilde{u}) \\
     & \hspace{7mm} + r \partial^2 (r^{j-1} \partial^{2j-2} \tilde{r}) + \partial \left(r^{j-1} \partial^{2j-2} \tilde{r}\right) + r \partial \left(r^{j-1} \partial^{2j-2} \tilde{r} \right) \\
\end{split}
\end{equation}
noting that there are collection of terms in the simplification for $D_t^{2j-1} \tilde{r}$ which all have the same order. Thus, computing the total order at the $j+1$ level, we get that each of the terms has order $\mathcal{O} = \frac{1}{2}$ at worst. We have completed the induction and shown that $[D_t^{2k-2}, \tilde{L}_1]$ has order $\mathcal{O} = \frac{1}{2}$ at worst for level $k$, and this completes the proof of \eqref{E:D_t_commutator_with_L_1}. The proofs of \eqref{E:div_curl_simplication_1} and \eqref{E:D_t_commutator_with_L_3} are similar.

\end{proof}

\end{lemma}

\begin{remark}
When we go to estimate $E^{2k}_{\text{wave}}$ from below, we will observe that these subcritical/ critical terms will not cause any issue. First, consider estimating terms in the $\tilde{r}$ equation which are subcritical. In Lemma \ref{L:Higher_order_elliptic_estimates_tilde_r} below, there terms are ultimately handled with smallness arguments combined with Remark \ref{R:order_free_boundary_term}. A similar situation also happens for the $\tilde{u}$ equation. Although the expression for $D_t^{2k} \tilde{u}$ contains terms which are critical as well (and not just subcritical), we recall the  useful fact that the $\tilde{u}$ equation gets a multiplier of order $1/2$ when doing energy estimates. Subsequently, the $E^{2k}_{\text{wave}}$ energy has an extra $r^{1/2}$ weight and this will be perfect for applying the same smallness arguments for these extra terms.
\end{remark}

\subsection{Commutators with weighted spatial derivatives}
\label{S:Subsection_weighted_spatial_derivatives}

Our goal is to extend Lemmas \ref{L:Elliptic_estimates_tilde_r} and \ref{L:Elliptic_div_curl_estimates_tilde_u} to higher order Sobolev spaces. To do so, we will need to compute commutators between our good elliptic operators $\tilde{L}_1, \tilde{L}_2$, $\tilde{L}_3$, and the weighted spatial derivatives $r^{k-j} \partial^{2(k-j)}$. 

Let $k\in \mathbb{N}$ be fixed, and $j \in \mathbb{N}$ with $1 \leq j \leq k$. To begin, we will use the following notation:
\begin{equation}
    \label{E:Notation_for_higher_order_commutators}
    \begin{split}
         m&:= k-j \\
         L^{a,b} &:= r^a \partial^{b} \\
         \tilde{r}_{2j} &:= D_t^{2j} \tilde{r}
    \end{split}
\end{equation}
and we recall from \eqref{E:tilde_L_1_definition} that
\begin{equation*}
    \tilde{L}_1 = \frac{\gamma-1}{\Gamma+r} H^{ij} (r \partial_i \partial_j + \frac{1}{\gamma-1} \partial_i r \partial_j )
\end{equation*}

We plan to handle $\tilde{L}_1$ first, before moving on to $\tilde{L}_2$ and $\tilde{L}_3$. By Lemma \ref{L:Commutator_D_t_with_L_1_2_3}, we have that
\begin{equation}
\label{E:r_2j_equation_with_P}
    \tilde{r}_{2j} \simeq \tilde{L}_1 \tilde{r}_{2j-2} + P
\end{equation}
where $P$ are terms that have an order of $1/2$ or better at level $j$ (i.e. at the $\mathcal{H}^{2j}$ level). We plan to ultimately handle these terms using smallness arguments after we apply $L^{m, 2m}$ to \eqref{E:r_2j_equation_with_P} and analyze with the $H^{0, \frac{2-\gamma}{2(\gamma-1)}}$ norm.
Applying $L^{m, 2m}$ to both sides (where $L^{m,2m}$ is defined in \eqref{E:Notation_for_higher_order_commutators}), we have
\begin{equation}
    L^{m, 2m} \tilde{r}_{2j} \simeq \tilde{L}_1 L^{m, 2m} \tilde{r}_{2j-2} + [L^{m, 2m}, \tilde{L}_1] \tilde{r}_{2j-2} + L^{m,2m} P
\end{equation}

\begin{remark}
\label{R:Order after applying L to free boundary terms}
Observe what will happen to the order of the terms in $P$ when we apply $L^{m,2m}$. At the start, $P$ has terms with $\mathcal{H}^{2j}$ order $\frac{1}{2}$ (at worst). However, by repeatedly applying Lemma \ref{L:Order_of_operators}, we see that $\partial^{2m} P$ has $\mathcal{H}^{2j}$-order $\frac{1}{2} - 2m $. Then, $r^m \partial^{2m} P$ has $\mathcal{H}^{2j}$-order $\frac{1}{2} - 2m +m = \frac{1}{2} -m = \frac{1}{2} - (k-j)$. Then, if we want to calculate the order at level $k$ instead of level $j$ (and note that $k\geq j$), we see by Part 4 in Lemma \ref{L:Order_of_operators} that $r^m \partial^{2m} P$ must have $\mathcal{H}^{2k}$-order $\frac{1}{2} - (k-j) + (k-j) = \frac{1}{2}$. 
\end{remark}

By Remark \ref{R:Order after applying L to free boundary terms}, this means that at the $\mathcal{H}^{2k}$ level, $L^{m,2m} P$ will still be subcritical and we plan to handle these terms with smallness arguments. Thus, let us make use of the following notation:
\begin{equation}
\label{E:P_alpha_notation}
    P_{\mathcal{O} \geq 1/2} := \left(\text{terms with $\mathcal{H}^{2k}$ order $\mathcal{O} \geq \frac{1}{2}$}\right)
\end{equation}
which we will use to collect any perturbative/smallness terms along the way that can handled easily with respect to the $\mathcal{H}^{2k}$ norm. Using this notation, we have
\begin{equation}
\label{E:L_1_commutator_start}
     L^{m, 2m} \tilde{r}_{2j} \simeq \tilde{L}_1 L^{m, 2m} \tilde{r}_{2j-2} + [L^{m, 2m}, \tilde{L}_1] \tilde{r}_{2j-2} + P_{\mathcal{O} \geq 1/2}
\end{equation}
and the bulk of our analysis concerns proper handling of the commutator $[L^{m, 2m}, \tilde{L}_1] \tilde{r}_{2j-2}$.

Since the commutator is trivial for $m=0$, let's assume that $m\geq 1$, and we plan to apply an induction argument on $m$. As we will see in the coming pages, the critical terms in the commutator need to be handled in a particular way where we absorb a critical term into the elliptic operator $\tilde{L}_1$, and then the rest of the terms are subcritical. Note that if we had arrived at \eqref{E:L_1_commutator_start} with $L^{m-1,2(m-2)}$ instead of $L^{m,2m}$, this would produce 
\begin{equation}
\label{E:L_1_commutator_start_m-1_case}
     L^{m-1, 2(m-1)} \tilde{r}_{2j} \simeq \tilde{L}_1 L^{m-1, 2(m-1)} \tilde{r}_{2j-2} + [L^{m-1, 2(m-1)}, \tilde{L}_1] \tilde{r}_{2j-2} + P_{\mathcal{O} \geq 3/2}
\end{equation}
where the other terms, at worst, have $\mathcal{H}^{2k}$-order 1 (instead of critical in the $m$ case). 
For the inductive hypothesis, we need to assume that $L^{m-1, 2(m-1)}$ ``behaves well'' when we commute it with $\tilde{L}_1$ in that we have already absorbed the 
order 1 terms.
We will make the following inductive assumption (recalling from \eqref{E:Notation_for_higher_order_commutators} that $k=m+j$):
\begin{assumption}[Inductive Hypothesis on $m-1$]
\label{A:Inductive_hypothesis_for_L_1}
\end{assumption}
\begin{enumerate}
    \item The terms in $[L^{m-1, 2(m-1)}, \tilde{L}_1]\tilde{r}_{2j-2}$ all have $\mathcal{H}^{2(m-1+j)}$ order $\frac{1}{2}$, except for key order 0 terms which have already been handled by absorbing into $\tilde{L}_1 L^{m-1, 2(m-1)} \tilde{r}_{2j-2}$ with the analog of Proposition \ref{P:Proposition_Inequality_4_for_L_1} below. 
    \item The above assumption also holds for $\tilde{L}_1$ replaced by $\tilde{L}_1 + b \frac{\gamma-1}{\Gamma+r} H^{ij} \partial_i r \partial_j $ with $b \geq 0$.
\end{enumerate}  

\begin{remark}
    \label{R:Consequences_of_Inductive_Hyp_on_L_1}
    There are several consequences of Assumption \ref{A:Inductive_hypothesis_for_L_1} that we summarize:
    \begin{enumerate}
        \item The terms in $[L^{m-1, 2(m-1)}, \tilde{L}_1]\tilde{r}_{2j-2}$ have $\mathcal{H}^{2(m+j)} = \mathcal{H}^{2k}$ order $\frac{3}{2}$, i.e. they belong to $P_{\mathcal{O} \geq 3/2}$
        \item The terms in $[L^{m-1, 2(m-1)}, \tilde{L}_1] \partial \tilde{r}_{2j-2}$ belong to $P_{\mathcal{O} \geq 1/2}$.
        \item Suppose that $m \geq 2$. We claim that Assumption \ref{A:Inductive_hypothesis_for_L_1} on the $m-1$ case also provides information on lower order commutators. Indeed, we compute the following using two different ways:
        \begin{equation}
        \begin{split}
            [r\partial \left(L^{m-2, 2m-3} \right), \tilde{L_1}]\varphi &\simeq [L^{m-1, 2m-2} + L^{m-2, 2m-3} , \tilde{L}_1]\varphi = [L^{m-1, 2m-2}, \tilde{L}_1]\varphi + [ L^{m-2, 2m-3} , \tilde{L}_1]\varphi \\
             [r\partial \left(L^{m-2, 2m-3} \right), \tilde{L_1}]\varphi  & = r \partial [L^{m-2, 2m-3}, \tilde{L}_1] \varphi + [r \partial , \tilde{L}_1] L^{m-2, 2m-3}\varphi
        \end{split}
        \end{equation}
        Combining, we have
        \begin{equation}
            [L^{m-1, 2m-2}, \tilde{L}_1]\varphi \simeq r \partial [L^{m-2, 2m-3}, \tilde{L}_1] \varphi + [r \partial , \tilde{L}_1] L^{m-2, 2m-3}\varphi + [ L^{m-2, 2m-3} , \tilde{L}_1]\varphi.
        \end{equation}
        Thus, since we know that $[L^{m-1, 2(m-1)}, \tilde{L}_1]\tilde{r}_{2j-2}$ belongs to $P_{\mathcal{O} \geq 3/2}$ and $[L^{m-1, 2(m-1)}, \tilde{L}_1]\partial \tilde{r}_{2j-2}$ belongs to $P_{\mathcal{O} \geq 1/2}$, we can also conclude that $[L^{m-2, 2m-3}, \tilde{L}_1]\tilde{r}_{2j-2}$ belongs to $P_{\mathcal{O} \geq 3/2}$ and $[L^{m-2, 2m-3}, \tilde{L}_1]\partial \tilde{r}_{2j-2}$ belongs to $P_{\mathcal{O} \geq 1/2}$. By repeated use of this argument, a similar statement also holds for commutators involving $L^{l, m+l-1}$ where $1\leq l \leq m-1$.
    \end{enumerate}
\end{remark}

Now that we have finished discussing the consequences of our inductive hypothesis, we proceed with the commutator by computing the following
\begin{equation}
    \label{E:tilde_L_1_commutator_with_L}
    [L^{m,2m}, \tilde{L}_1] \tilde{r}_{2j-2} = L^{m, 2m} \tilde{L}_1 \tilde{r}_{2j-2} - \tilde{L}_1 L^{m, 2m} \tilde{r}_{2j-2}
\end{equation}
and
\begin{equation}
    \begin{split}
        L^{m, 2m} \tilde{L}_1 \tilde{r}_{2j-2} &= r^m \partial^{2m} \left(\frac{\gamma-1}{\Gamma+r} H^{ij} (r \partial_i \partial_j \tilde{r}_{2j-2} + \frac{1}{\gamma-1} \partial_i r \partial_j \tilde{r}_{2j-2} ) \right) \\
        &\simeq  \frac{\gamma-1}{\Gamma+r} H^{ij} r^m \partial^{2m} \left(r \partial_i \partial_j \tilde{r}_{2j-2} + \frac{1}{\gamma-1} \partial_i r \partial_j \tilde{r}_{2j-2} \right) + P_{\mathcal{O} \geq 1/2}\\
    \end{split}
\end{equation}
where we note that any terms where derivatives hit $\frac{\gamma-1}{\Gamma+r} H^{ij}$ get absorbed into $P_{\mathcal{O} \geq 1/2}$ (since these derivatives are not hitting either $r$ or $\tilde{r}_{2j-2}$ which are critical terms at worst). For analyzing the remaining terms, we continue
\begin{equation}
\begin{split}
     r^m\partial^{2m}\left(r \partial_i \partial_j \tilde{r}_{2j-2} + \frac{1}{\gamma-1} \partial_i r \partial_j \tilde{r}_{2j-2} \right) &\simeq r^m\sum_{l=0}^{2m} \partial^l r \partial^{2m-l} \partial_i \partial_j \tilde{r}_{2j-2} + \frac{1}{\gamma-1} \partial^l \partial_i r \partial^{2m-l} \partial_j \tilde{r}_{2j-2}
\end{split}
\end{equation}
where we absorbed combinatorial constants with $\simeq$. Now, we plan to separate out any of the terms that belong to $P_{\mathcal{O} \geq 1/2}$ from the critical terms which must be handled. When $\partial^{2m}$ hits $\tilde{L}_1$, we observe that the worst terms only occur when $2m$ derivatives hit $\tilde{r}_{2j-2}$, or $1$ derivative hits $r$ (in the first term) and $2m-1$ derivatives hit $\tilde{r}_{2j-2}$. The point is that after the first power of $r$ is removed, any subsequent derivatives that hit $\partial r$ do not actually contribute to the order. In fact, they are only taking away derivatives that could have hit $\tilde{r}_{2j-2}$. Indeed, we can check by computing the order. When $l=0$, we have that 
\begin{equation}
\begin{split}
     r^m r \partial^{2m} \partial_i \partial_j \tilde{r}_{2j-2}  & \simeq r^{m+1} \partial^{2m+2} (D_t^{2j-2} \tilde{r}) \text{ is $\mathcal{H}^{2k}$-critical ($\mathcal{O} = 0$)}  \\
     \frac{1}{\gamma-1}r^m \partial_i r \partial^{2m} \partial_j \tilde{r}_{2j-2} & \simeq r^m \partial^{2m+1}(D_t^{2j-2} \tilde{r}) \text{ is $\mathcal{H}^{2k}$-critical ($\mathcal{O} = 0$)}
\end{split}
\end{equation}
using Lemma \ref{L:Order_of_operators}.
When $l=1$, we see that
\begin{equation}
\begin{split}
     r^m \partial r \partial^{2m-1} \partial_i \partial_j \tilde{r}_{2j-2}  & \simeq r^m \partial^{2m+1} (D_t^{2j-2} \tilde{r}) \text{ is $\mathcal{H}^{2k}$-critical ($\mathcal{O} = 0$)}  \\
     \frac{1}{\gamma-1}r^m \partial \partial_i r \partial^{2m-1} \partial_j \tilde{r}_{2j-2} & \simeq r^m \partial^{2m}(D_t^{2j-2} \tilde{r}) \text{ is $\mathcal{H}^{2k}$-subcritical ($\mathcal{O} = 1$)}
\end{split}
\end{equation}
Thus, in the summation when $l\geq 2$, we will get
\begin{equation}
\begin{split}
     r^m \partial^l r \partial^{2m-l} \partial_i \partial_j \tilde{r}_{2j-2}  & \simeq r^m \partial^{\leq 2m} (D_t^{2j-2} \tilde{r}) \text{ is $\mathcal{H}^{2k}$-subcritical ($\mathcal{O} \geq 1$)}  \\
    \frac{1}{\gamma-1} r^m \partial^l \partial_i r \partial^{\leq 2m-l} \partial_j \tilde{r}_{2j-2} & \simeq r^m \partial^{\leq 2m-1}(D_t^{2j-2} \tilde{r}) \text{ is $\mathcal{H}^{2k}$-subcritical ($\mathcal{O} \geq 2$)}
\end{split}
\end{equation}
Summarizing our order analysis, we see that many of the terms in the summation will be absorbed into $P_{\mathcal{O} \geq 1/2}$, and we have:
\begin{equation}
\begin{split}
\label{E:L_1_Commutator_side_1}
     L^{m, 2m} \tilde{L}_1 \tilde{r}_{2j-2} & \simeq \frac{\gamma-1}{\Gamma+r} H^{ij} \left(\textcolor{blue}{r^{m+1} \partial^{2m} \partial_i \partial_j \tilde{r}_{2j-2}} + \textcolor{blue}{\frac{1}{\gamma-1}  r^m \partial_i r \partial^{2m} \partial_j \tilde{r}_{2j-2}} + r^m  \partial r \partial^{2m-1} \partial_i \partial_j \tilde{r}_{2j-2}   \right) \\
    & \hspace{7mm} + P_{\mathcal{O} \geq 1/2}
\end{split}
\end{equation}

For the other side of the commutator, we have

\begin{equation}
\label{E:L_1_Commutator_side_2}
\begin{split}
     \tilde{L}_1 \left(L^{m,2m} \tilde{r}_{2j-2} \right) &\simeq \frac{\gamma-1}{\Gamma+r} H^{ij} \left(r \partial_i \partial_j (r^m \partial^{2m} \tilde{r}_{2j-2})+ \frac{1}{\gamma-1} \partial_i r \partial_j(r^m \partial^{2m} \tilde{r}_{2j-2})  \right) \\
     &\simeq \frac{\gamma-1}{\Gamma+r} \big( H^{ij} \left(\textcolor{blue}{r^{m+1} \partial_i \partial_j \partial^{2m} \tilde{r}_{2j-2}} + 2r \partial_i(r^m) \partial_j \partial^{2m} \tilde{r}_{2j-2}  + \textcolor{purple}{r \partial_i \partial_j (r^m)\partial^{2m} \tilde{r}_{2j-2}}  \right) \\
     &\hspace{5mm} + \frac{1}{\Gamma+r} H^{ij} \partial_i r \left(\textcolor{blue}{r^m \partial_j \partial^{2m} \tilde{r}_{2j-2}} + \partial_j(r^m) \partial^{2m} \tilde{r}_{2j-2} \big) \right).
\end{split}
\end{equation}
Now, if we compute the order of the \textcolor{purple}{purple} term, we see that
\begin{equation*}
    \frac{\gamma-1}{\Gamma+r} H^{ij} r \partial_i \partial_j (r^m)\partial^{2m} \tilde{r}_{2j-2} \simeq r^{m-1} \partial^{2m} \tilde{r}_{2j-2} \text{ is $\mathcal{H}^{2k}$-subcritical ($\mathcal{O} \geq 2$)}
\end{equation*}
and thus, we can absorb it into $P_{\mathcal{O} \geq 1/2}$. Observing that the \textcolor{blue}{blue} terms cancel in \eqref{E:L_1_Commutator_side_1} and \eqref{E:L_1_Commutator_side_2}, we have
\begin{equation}
\begin{split}
     [L^{m,2m}, \tilde{L}_1] \tilde{r}_{2j-2}
     &= \frac{\gamma-1}{\Gamma+r} H^{ij} \left(  \partial r L^{m,2m-1} \partial_i \partial_j \tilde{r}_{2j-2} - 2m \partial_i r L^{m,2m} \partial_j \tilde{r}_{2j-2} - \frac{1}{\gamma-1} \partial_i r \partial_j r L^{m-1,2m} \tilde{r}_{2j-2} \right) \\
     & \hspace{7mm} + P_{\mathcal{O} \geq 1/2} \\
\end{split}
\end{equation}
where the first two terms take the form $L^{m, 2m+1} \tilde{r}_{2j-2}$.

Then, combining with \eqref{E:L_1_commutator_start}, and applying the $H^{0, \frac{2-\gamma}{2(\gamma-1)}}$ norm, we have 
\begin{equation}
\label{E:L_1_inequalities_starting_1}
\begin{split}
     \norm{\tilde{L}_1 L^{m,2m} \tilde{r}_{2j-2}}_{H^{0, \frac{2-\gamma}{2(\gamma-1)}}}  & \lesssim  \norm{L^{m,2m} \tilde{r}_{2j}}_{H^{0, \frac{2-\gamma}{2(\gamma-1)}}} +  \norm{L^{m,2m+1} \tilde{r}_{2j-2}}_{H^{0, \frac{2-\gamma}{2(\gamma-1)}}} \\
     &\hspace{5mm} + \norm{L^{m-1,2m} \tilde{r}_{2j-2}}_{H^{0, \frac{2-\gamma}{2(\gamma-1)}}}  + \norm{P_{\mathcal{O} \geq 1/2}}_{H^{0, \frac{2-\gamma}{2(\gamma-1)}}} \\
\end{split}
\end{equation}
For the LHS, we can apply Lemma \ref{L:Elliptic_estimates_tilde_r} to get
\begin{equation}
     \norm{\tilde{L}_1 L^{m,2m} \tilde{r}_{2j-2}}_{H^{0, \frac{2-\gamma}{2(\gamma-1)}}} \gtrsim  \norm{L^{m,2m} \tilde{r}_{2j-2}}_{H^{2, 1+ \frac{2-\gamma}{2(\gamma-1)}}} - \textcolor{blue}{ \norm{L^{m,2m} \tilde{r}_{2j-2}}_{H^{0, \frac{2-\gamma}{2(\gamma-1)}}}}
\end{equation}
Then, adding the \textcolor{blue}{blue} term to the other side of \eqref{E:L_1_inequalities_starting_1}, we have
\begin{equation}
\label{E:L_1_inequalities_starting_2}
\begin{split}
     \norm{ L^{m,2m} \tilde{r}_{2j-2}}_{H^{2, 1+\frac{2-\gamma}{2(\gamma-1)}}}  & \lesssim  \norm{L^{m,2m} \tilde{r}_{2j}}_{H^{0, \frac{2-\gamma}{2(\gamma-1)}}} +  \norm{L^{m,2m+1} \tilde{r}_{2j-2}}_{H^{0, \frac{2-\gamma}{2(\gamma-1)}}} \\
     &\hspace{5mm} + \norm{L^{m-1,2m} \tilde{r}_{2j-2}}_{H^{0, \frac{2-\gamma}{2(\gamma-1)}}} + \textcolor{blue}{\norm{L^{m,2m} \tilde{r}_{2j-2}}_{H^{0, \frac{2-\gamma}{2(\gamma-1)}}}} \\
     & \hspace{5mm} + \norm{P_{\mathcal{O} \geq 1/2}}_{H^{0, \frac{2-\gamma}{2(\gamma-1)}}} \\
\end{split}
\end{equation}

Now, by adding the following quantity on both sides of our inequality:
\begin{equation}
    \label{E:L_1_quanity added on both sides}
    \norm{\tilde{r}_{2j-2}}_{H^{0, \frac{2-\gamma}{2(\gamma-1)}}} + 
    \sum_{l=0}^{m-1} \norm{L^{l, m+l} \tilde{r}_{2j-2} }_{H^{2, 1+ \frac{2-\gamma}{2(\gamma-1)}}}
\end{equation}
we see using the help or Remark \ref{R:H^2k_norm_equivalence} that the LHS will be equivalent to
\begin{equation}
    \norm{\tilde{r}_{2j-2}}_{H^{2m+2,m+1 + \frac{2-\gamma}{2(\gamma-1)}}}
\end{equation}
For the RHS, we will get the following collection of terms which we number
\begin{equation}
\label{E:L_1_inequalities_starting_3}
\begin{split}
    & \norm{L^{m,2m} \tilde{r}_{2j}}_{H^{0, \frac{2-\gamma}{2(\gamma-1)}}} +  \norm{L^{m,2m+1} \tilde{r}_{2j-2}}_{H^{0, \frac{2-\gamma}{2(\gamma-1)}}} \\
     &\hspace{5mm} + \norm{L^{m-1,2m} \tilde{r}_{2j-2}}_{H^{0, \frac{2-\gamma}{2(\gamma-1)}}} + \textcolor{blue}{\norm{L^{m,2m} \tilde{r}_{2j-2}}_{H^{0, \frac{2-\gamma}{2(\gamma-1)}}}}  \\
     &\hspace{5mm} 
    \sum_{l=0}^{m-1} \norm{L^{l, m+l} \tilde{r}_{2j-2} }_{H^{2, 1+ \frac{2-\gamma}{2(\gamma-1)}}}  +  \norm{\tilde{r}_{2j-2}}_{H^{0, \frac{2-\gamma}{2(\gamma-1)}}} + \norm{P_{\mathcal{O} \geq 1/2}}_{H^{0, \frac{2-\gamma}{2(\gamma-1)}}} \\
    &:= \norm{L^{m,2m} \tilde{r}_{2j}}_{H^{0, \frac{2-\gamma}{2(\gamma-1)}}} + \sum_{a=0}^6 R_6
\end{split}
\end{equation}
using $R_1$ through $R_6$. 

For the terms $R_1, R_2,$ and $R_3$, we first observe that
\begin{equation}
\begin{split}
    \norm{L^{m,2m+1} \tilde{r}_{2j-2}}_{H^{0, \frac{2-\gamma}{2(\gamma-1)}}}  & + \norm{L^{m-1,2m} \tilde{r}_{2j-2}}_{H^{0, \frac{2-\gamma}{2(\gamma-1)}}} \\
    & + \textcolor{blue}{\norm{L^{m,2m} \tilde{r}_{2j-2}}_{H^{0, \frac{2-\gamma}{2(\gamma-1)}}}} \lesssim \norm{L^{m-1, 2m-1}\tilde{r}_{2j-2}}_{H^{2,1+ \frac{2-\gamma}{2(\gamma-1)}}} \\
    &\hspace{4.1cm} = \norm{L^{m-1, 2m-2} \partial \tilde{r}_{2j-2}}_{H^{2,1+ \frac{2-\gamma}{2(\gamma-1)}}} \\
\end{split}
\end{equation}
Then, we claim that the following inequality holds:
\begin{proposition}
\label{P:Proposition_Inequality_4_for_L_1}
\begin{equation}
\begin{split}
     \norm{L^{m-1, 2m-2} \partial \tilde{r}_{2j-2}}_{H^{2,1+ \frac{2-\gamma}{2(\gamma-1)}}} & \lesssim  \norm{L^{m-1, 2m-2} \partial \tilde{r}_{2j}}_{H^{0, \frac{2-\gamma}{2(\gamma-1)}}} \\
     & \hspace{5mm} +  \norm{L^{m-1, 2m-2} \partial \tilde{r}_{2j-2}}_{H^{0, \frac{2-\gamma}{2(\gamma-1)}}} + \norm{P_{\mathcal{O} \geq 1/2}}_{H^{0, \frac{2-\gamma}{2(\gamma-1)}}}
\end{split}
\end{equation}    
\begin{proof}
    To prove this proposition, we rely on the smallness of $\partial_3 r$ (similar to the work in \cite{DisconziIfrimTataru}) as well as a key step in which we we must absorb a term into $\tilde{L}_1$, creating the updated elliptic operator $\widehat{\tilde{L}_1}$. Similar to \cite{DisconziIfrimTataru}, we rely on localizing in the neighborhood of a boundary point such that 
    \begin{equation}
    \label{E:Localization_assumptions_L_1_and_more}
        |\partial' r| \lesssim A, \hspace{5mm} |\partial_3 r - 1| \lesssim A, \hspace{5mm} |H^{3, i'}| \lesssim A
    \end{equation}
    where $A \ll 1$ is small constant, and primed indices will range over $x^1$ and $x^2$ (i.e. $\partial' r \simeq \partial_1 r, \partial_2 r$). Note that a key assumption is present on the off-diagonal components of $H$ (and a similar simplification can also be found in \cite{DisconziIfrimTataru}). The proof is quite similar, and we only need to show the related inequalities
\begin{equation}
    \label{E:Inequality_4_and_6_combined}
         \begin{split}
        \norm{L^{m-1, 2m-2} \partial_3 \tilde{r}_{2j-2} }_{H^{2, 1+\frac{2-\gamma}{2(\gamma-1)}}} & \lesssim  \norm{L^{m-1, 2m-2} \partial_3 \tilde{r}_{2j} }_{H^{0, \frac{2-\gamma}{2(\gamma-1)}}} +  \norm{L^{m-1, 2m-2} \partial' \tilde{r}_{2j-2} }_{H^{2, 1+ \frac{2-\gamma}{2(\gamma-1)}}} \\
        & \hspace{5mm} + \norm{P_{\mathcal{O} \geq 1/2}}_{H^{0, \frac{2-\gamma}{2(\gamma-1)}}}
    \end{split}
    \end{equation}
    and 
     \begin{equation}
    \label{E:Inequality_5_key_step_in_proof}
        \begin{split}
        \norm{L^{m-1, 2m-2} \partial' \tilde{r}_{2j-2} }_{H^{2, 1+\frac{2-\gamma}{2(\gamma-1)}}} & \lesssim  \norm{L^{m-1, 2m-2} \partial' \tilde{r}_{2j} }_{H^{0, \frac{2-\gamma}{2(\gamma-1)}}} \\
        & \hspace{5mm} +\textcolor{blue}{  \norm{L^{m-1, 2m-2} \partial' \tilde{r}_{2j-2} }_{H^{0, \frac{2-\gamma}{2(\gamma-1)}}}} + \norm{P_{\mathcal{O} \geq 1/2}}_{H^{0, \frac{2-\gamma}{2(\gamma-1)}}}
    \end{split}
    \end{equation}

\end{proof}
\end{proposition}

\begin{corollary}[Corollary for Proposition \ref{P:Proposition_Inequality_4_for_L_1}]
\label{C:Corollary_estimates_for_partial_r_2j-2}
Let $l \in \mathbb{N}$ with $1 \leq l \leq m-1$. The following inequality also holds

\begin{equation}
\begin{split}
     \norm{L^{l, m+l-1} \partial \tilde{r}_{2j-2}}_{H^{2,1+ \frac{2-\gamma}{2(\gamma-1)}}} & \lesssim  \norm{L^{l, m+l-1} \partial \tilde{r}_{2j}}_{H^{0, \frac{2-\gamma}{2(\gamma-1)}}} \\
     & \hspace{5mm} +  \norm{L^{l, m+l-1} \partial \tilde{r}_{2j-2}}_{H^{0, \frac{2-\gamma}{2(\gamma-1)}}} + \norm{P_{\mathcal{O} \geq 1/2}}_{H^{0, \frac{2-\gamma}{2(\gamma-1)}}}
\end{split}
\end{equation}
\begin{proof}
    To prove this corollary, we simply need to follow the proof of Proposition \ref{P:Proposition_Inequality_4_for_L_1} combined with the last consequence in Remark \ref{R:Consequences_of_Inductive_Hyp_on_L_1}. Anywhere the inductive hypothesis is used on the commutator $L^{m-1,2m-2}$, we know that a similar fact holds for $L^{l, m+l-1}$ by Remark \ref{R:Consequences_of_Inductive_Hyp_on_L_1}.
\end{proof}

\end{corollary}

Now, we can return to \eqref{E:L_1_inequalities_starting_1} - \eqref{E:L_1_inequalities_starting_3} and summarize our results for what we have on the LHS and RHS now that we have estimated the terms $R_1, R_2$, and $R_3$ using Proposition \ref{P:Proposition_Inequality_4_for_L_1}. Simplifying $L^{m-1, 2m-2} \partial = L^{m-1, 2m-1} $, this yields
\begin{equation}
\label{E:main_inequality_tilde_r_reductions}
\begin{split}
     \norm{\tilde{r}_{2j-2}}_{H^{2m+2, m+1+ \frac{2-\gamma}{2(\gamma-1)}}} &\lesssim  \norm{L^{m,2m} \tilde{r}_{2j}}_{H^{0, \frac{2-\gamma}{2(\gamma-1)}}} + \norm{L^{m-1, 2m-1} \tilde{r}_{2j}}_{H^{0, \frac{2-\gamma}{2(\gamma-1)}}} \\
     &\hspace{5mm}+ \textcolor{red}{\norm{L^{m-1, 2m-1} \tilde{r}_{2j-2}}_{H^{0, \frac{2-\gamma}{2(\gamma-1)}}}}
    + \textcolor{red}{\sum_{l=0}^{m-1} \norm{L^{l, m+l} \tilde{r}_{2j-2} }_{H^{2, 1+ \frac{2-\gamma}{2(\gamma-1)}}}}  \\
    & \hspace{5mm} +  \norm{\tilde{r}_{2j-2}}_{H^{0, \frac{2-\gamma}{2(\gamma-1)}}} + \norm{P_{\mathcal{O} \geq 1/2}}_{H^{0, \frac{2-\gamma}{2(\gamma-1)}}}
\end{split}
\end{equation}
where only the \textcolor{red}{red} terms need to be handled, and we will be keeping $R_5$ and $R_6$ from before on the RHS. For the first red term, we have
\begin{equation}
\label{E:first_red_term_start}
\begin{split}
     \textcolor{red}{\norm{L^{m-1, 2m-1} \tilde{r}_{2j-2}}_{H^{0, \frac{2-\gamma}{2(\gamma-1)}}}} & \lesssim \norm{L^{m-1, 2m-2} \tilde{r}_{2j-2}}_{H^{1, \frac{2-\gamma}{2(\gamma-1)}}} \\
     & \lesssim \norm{ \tilde{L}_1 (L^{m-1, 2m-2} \tilde{r}_{2j-2})}_{H^{0, \frac{2-\gamma}{2(\gamma-1)}}}
\end{split}
\end{equation}
where we used the elliptic estimate for $\tilde{L}_1$ in the last line. Then, we can use \eqref{A:Inductive_hypothesis_for_L_1} to commute with $L^{m-1, 2(m-1)}$ which will only produce additional terms in $P_{\mathcal{O} \geq 1/2}$. Finally, we use \eqref{E:r_2j_equation_with_P} to express $\tilde{L}_1 \tilde{r}_{2j-2}$ in terms of $\tilde{r}_{2j}$ plus additional $P$ terms. We have
\begin{equation}
\label{E:first_red_term_end}
    \begin{split}
          \textcolor{red}{\norm{L^{m-1, 2m-1} \tilde{r}_{2j-2}}_{H^{0, \frac{2-\gamma}{2(\gamma-1)}}}} & \lesssim  \norm{ L^{m-1, 2m-2}(\tilde{L}_1  \tilde{r}_{2j-2}) + P_{\mathcal{O} \geq 1/2}}_{H^{0, \frac{2-\gamma}{2(\gamma-1)}}} \\
          &\lesssim  \norm{ L^{m-1, 2m-2}(  \tilde{r}_{2j} + P) }_{H^{0, \frac{2-\gamma}{2(\gamma-1)}}} + \norm{P_{\mathcal{O} \geq 1/2}}_{H^{0, \frac{2-\gamma}{2(\gamma-1)}}} \\
          &\lesssim \norm{ L^{m-1, 2m-2} \tilde{r}_{2j} }_{H^{0, \frac{2-\gamma}{2(\gamma-1)}}} + \norm{P_{\mathcal{O} \geq 1/2}}_{H^{0, \frac{2-\gamma}{2(\gamma-1)}}} \\
          &\lesssim \norm{\tilde{r}_{2j} }_{H^{2m,m+ \frac{2-\gamma}{2(\gamma-1)}}} + \norm{P_{\mathcal{O} \geq 1/2}}_{H^{0, \frac{2-\gamma}{2(\gamma-1)}}} \\
    \end{split}
\end{equation}
which is an important estimate for the first red term.

For the final \textcolor{red}{red} term, we start by rewriting and then applying Corollary \ref{C:Corollary_estimates_for_partial_r_2j-2}. 
\begin{equation}
\begin{split}
     \textcolor{red}{\sum_{l=0}^{m-1} \norm{L^{l, m+l-1} \partial \tilde{r}_{2j-2} }_{H^{2, 1+ \frac{2-\gamma}{2(\gamma-1)}}}} &\lesssim \sum_{l=0}^{m-1} \left( \norm{L^{l, m+l-1} \partial \tilde{r}_{2j} }_{H^{0, \frac{2-\gamma}{2(\gamma-1)}}} + \norm{L^{l, m+l-1} \partial \tilde{r}_{2j-2} }_{H^{0, \frac{2-\gamma}{2(\gamma-1)}}} \right) \\
     & \lesssim  \sum_{l=0}^{m-1} \left( \norm{L^{l, m+l} \tilde{r}_{2j} }_{H^{0, \frac{2-\gamma}{2(\gamma-1)}}} + \norm{L^{l, m+l} \tilde{r}_{2j-2} }_{H^{0, \frac{2-\gamma}{2(\gamma-1)}}} \right) \\
     &\lesssim \norm{\tilde{r}_{2j}}_{H^{2m, m+ \frac{2-\gamma}{2(\gamma-1)}}} +  \sum_{l=0}^{m-1} \norm{L^{l, m+l-1} \tilde{r}_{2j-2} }_{H^{1, \frac{2-\gamma}{2(\gamma-1)}}} \\
\end{split}
\end{equation}
where, in the last line, we used an inequality similar to \ref{E:first_red_term_start} for the $\tilde{r}_{2j-2}$ terms. For the remaining sum, it remains to apply the elliptic estimate for $\tilde{L}_1$ and then commute $\tilde{L}_1$ with $L^{l, m+l-1}$ using Remark \ref{R:Consequences_of_Inductive_Hyp_on_L_1}. We have
\begin{equation}
\label{E:main_inequality_tilde_r_reductions_conclusion}
\begin{split}
     \sum_{l=0}^{m-1} \left(\norm{L^{l, m+l-1} \tilde{r}_{2j-2} }_{H^{1, \frac{2-\gamma}{2(\gamma-1)}}} \right) &\lesssim \sum_{l=0}^{m-1} \norm{\tilde{L}_1 (L^{l, m+l-1} \tilde{r}_{2j-2}) }_{H^{0, \frac{2-\gamma}{2(\gamma-1)}}}  \\
     & \lesssim \norm{ P_{\mathcal{O} \geq 1/2} }_{H^{0, \frac{2-\gamma}{2(\gamma-1)}}} + \sum_{l=0}^{m-1} \norm{ L^{l, m+l-1}  (\tilde{r}_{2j} + P)}_{H^{0, \frac{2-\gamma}{2(\gamma-1)}}} \\
     &\lesssim \norm{\tilde{r}_{2j} }_{H^{2m,m+ \frac{2-\gamma}{2(\gamma-1)}}}+  \norm{ P_{\mathcal{O} \geq 1/2} }_{H^{0, \frac{2-\gamma}{2(\gamma-1)}}} \\
\end{split}
\end{equation}
where we followed the same argument as \eqref{E:first_red_term_end} with $L^{m-1, 2m-1}$ replaced by $L^{l, m+l-1}$.

Finally combining \eqref{E:main_inequality_tilde_r_reductions} - \eqref{E:main_inequality_tilde_r_reductions_conclusion}, we are now ready to prove our higher order elliptic estimates for the $\tilde{r}$ equation.

\begin{lemma}[Higher Order Elliptic estimates for $\tilde{r}$]
\label{L:Higher_order_elliptic_estimates_tilde_r}
Under the conditions for Lemma \ref{L:Elliptic_estimates_tilde_r}, the following inequality holds
\begin{equation}
     \norm{\tilde{r}}_{H^{2k, k+ \frac{2-\gamma}{2(\gamma-1)}}} \lesssim \sum_{j=0}^k \norm{D_t^{2j} \tilde{r}}_{H^{0, \frac{2-\gamma}{2(\gamma-1)}}} + C_{\hat{\varepsilon}} \norm{(\tilde{s}, \tilde{r}, \tilde{u})}_{\mathcal{H}^{2k}}
\end{equation}
where  $C_{\hat{\varepsilon}} \ll 1$ depends on $\hat{\varepsilon}$.

\begin{proof}
Combining \eqref{E:main_inequality_tilde_r_reductions} - \eqref{E:main_inequality_tilde_r_reductions_conclusion} and summarizing our work in this section, we get the following inequality:
\begin{equation}
\begin{split}
     \norm{\tilde{r}_{2j-2}}_{H^{2m+2, m+1+ \frac{2-\gamma}{2(\gamma-1)}}} &\lesssim  \norm{L^{m,2m} \tilde{r}_{2j}}_{H^{0, \frac{2-\gamma}{2(\gamma-1)}}} + \norm{L^{m-1, 2m-1} \tilde{r}_{2j}}_{H^{0, \frac{2-\gamma}{2(\gamma-1)}}} + \norm{\tilde{r}_{2j}}_{H^{2m, m + \frac{2-\gamma}{2(\gamma-1)}}}\\
    & \hspace{5mm} +  \norm{\tilde{r}_{2j-2}}_{H^{0, \frac{2-\gamma}{2(\gamma-1)}}} + \norm{P_{\mathcal{O} \geq 1/2}}_{H^{0, \frac{2-\gamma}{2(\gamma-1)}}} \\
    & \lesssim \norm{\tilde{r}_{2j}}_{H^{2m, m + \frac{2-\gamma}{2(\gamma-1)}}} + \norm{\tilde{r}_{2j-2}}_{H^{0, \frac{2-\gamma}{2(\gamma-1)}}} + \norm{P_{\mathcal{O} \geq 1/2}}_{H^{0, \frac{2-\gamma}{2(\gamma-1)}}} \\
    &\lesssim \norm{\tilde{r}_{2j}}_{H^{2m, m + \frac{2-\gamma}{2(\gamma-1)}}} + \norm{\tilde{r}_{2j-2}}_{H^{0, \frac{2-\gamma}{2(\gamma-1)}}} + \hat{\varepsilon} \norm{(\tilde{s}, \tilde{r}, \tilde{u})}_{\mathcal{H}^{2k}} \\
\end{split}
\end{equation}
which we showed using an induction argument on $m \geq 1$, and the $P_{\mathcal{O} \geq 1/2}$ terms are treated using \eqref{E:P_alpha_notation} and Remark \ref{R:order_free_boundary_term}. 
Then, we recall from \eqref{E:Notation_for_higher_order_commutators} that $m= k-j$, so when $m=k-1$, we have $j=1$ and
\begin{equation}
     \norm{\tilde{r}}_{H^{2k, k+ \frac{2-\gamma}{2(\gamma-1)}}} \lesssim \norm{D_t^2 \tilde{r}}_{H^{2k-2, k-1 + \frac{2-\gamma}{2(\gamma-1)}}} + \norm{\tilde{r}}_{H^{0, \frac{2-\gamma}{2(\gamma-1)}}} + \hat{\varepsilon} \norm{(\tilde{s}, \tilde{r}, \tilde{u})}_{\mathcal{H}^{2k}}
\end{equation}
Iterating, we obtain successive inequalities all the way up to $\norm{D_t^{2k} \tilde{r}}_{H^{0, \frac{2-\gamma}{2(\gamma-1)}}}$. Combining together, we have 
\begin{equation}
\label{E:tilde_r_coercivity_proof_finish}
     \norm{\tilde{r}}_{H^{2k, k+ \frac{2-\gamma}{2(\gamma-1)}}} \lesssim \sum_{j=0}^k \norm{D_t^{2j} \tilde{r}}_{H^{0, \frac{2-\gamma}{2(\gamma-1)}}} + C_{\hat{\varepsilon}} \norm{(\tilde{s}, \tilde{r}, \tilde{u})}_{\mathcal{H}^{2k}}
\end{equation}
where $C_{\hat{\varepsilon}} \ll 1$ depends on $\hat{\varepsilon}$. This completes the proof. 
\end{proof}
    
\end{lemma}

\begin{remark}
    \label{R:Connecting_to_E2k_wave_energy}
    We observe that the RHS of the inequality in Lemma \ref{L:Higher_order_elliptic_estimates_tilde_r} closely resembles the higher order wave energy $E^{2k}_{\text{wave}}$ found in \eqref{E:E^2k_wave}. After obtaining the matching estimates for $D_t^{2j} \tilde{u}$, we will have the full $E^{2k}_{\text{wave}}$ energy on the RHS and we will almost have the proof of Theorem \ref{Th:E^2k_equiv_H^2k} on the equivalence between our full energy \eqref{E:E^2k_full_linearized_energy} and the $\mathcal{H}^{2k}$ norm \eqref{E:Higher_order_H^2k_norm}.
\end{remark}

Next, we apply a similar argument for the $\tilde{u}$ equation and the operators $L_2$ with ``good'' spatial parts $\tilde{L}_2$ and $\tilde{L}_3$.  We observe that much of the analysis will be quite similar with $\tilde{L}_1$ replaced by $\tilde{L}_2$, and $H^{0,\frac{2-\gamma}{2(\gamma-1)}}$ replaced by $H^{0,\frac{1}{2}+\frac{2-\gamma}{2(\gamma-1)}} = H^{0,\frac{1}{2(\gamma-1)}} $. Following the arguments of Section \ref{S:Subsection_weighted_spatial_derivatives}, we will need to also incorporate the $\tilde{L}_3$ operator at each step in order to apply Lemma \ref{L:Elliptic_div_curl_estimates_tilde_u}. We will state the desired Lemma here for convenience:

\begin{lemma}[Higher order div-curl estimates for $\tilde{u}$]
\label{L:Higher_order_div_curl_estimates_tilde_u}

Under the conditions for Lemma \ref{L:Elliptic_div_curl_estimates_tilde_u}, the following inequality holds
\begin{equation}
  \norm{\tilde{u}}_{H^{2k, k+ \frac{1}{2(\gamma-1)}}} \lesssim \left(\sum_{j=0}^k \norm{D_t^{2j} \tilde{u}}_{H^{0, \frac{1}{2(\gamma-1)}}}\right)+ \norm{\overrightarrow{\text{curl}} \hspace{1mm} \tilde{u}}_{H^{2k-1, k+ \frac{1}{2(\gamma-1)}}} + D_{\hat{\varepsilon}} \norm{(\tilde{s}, \tilde{r}, \tilde{u})}_{\mathcal{H}^{2k}}
\end{equation}
where $D_{\hat{\varepsilon}} \ll 1$ depends on $\hat{\varepsilon}$.

\end{lemma}

\begin{remark}
    \label{R:Connecting_to_E2k_wave_energy_and_E2k_transport}
    We observe that the RHS of the inequality in Lemma \ref{L:Higher_order_div_curl_estimates_tilde_u} closely resembles the higher order wave energy $E^{2k}_{\text{wave}}$ found in \eqref{E:E^2k_wave} plus an additional term that depends on $\overrightarrow{\text{curl}} \hspace{1mm} \tilde{u}$. However, in view of Lemma \ref{L:relating_vorticity_spatial_curl}, this term is precisely bounded by $\norm{\hat{\omega}}_{H^{2k-1, k+ \frac{1}{2(\gamma-1)}}}$ which is the key vorticity piece of $E^{2k}_{\text{transport}}$. We note that Lemma \ref{L:Higher_order_div_curl_estimates_tilde_u} is used in conjuction with Lemma \ref{L:Higher_order_elliptic_estimates_tilde_r} to prove Theorem \ref{Th:E^2k_equiv_H^2k} on the equivalence between our full energy \eqref{E:E^2k_full_linearized_energy} and the $\mathcal{H}^{2k}$ norm \eqref{E:Higher_order_H^2k_norm}.
\end{remark}

\section{Energy Equivalence}
\label{S:Energy_equivalence}

\begin{theorem}[Equivalence between $E^{2k}$ and $\mathcal{H}^{2k}$ norms]
\label{Th:E^2k_equiv_H^2k}
Let $(\tilde{s}, \tilde{r}, \tilde{v})$ be smooth functions in $\overline{\Omega}$. If $r$ is positive and uniformly non-degenerate on $\Gamma$, then 
\begin{equation*}
    E^{2k}(\tilde{s}, \tilde{r}, \tilde{u}) \approx \norm{(\tilde{s}, \tilde{r}, \tilde{u})}_{\mathcal{H}^{2k}}^2
\end{equation*}
where the equivalence depends on $\gamma$ and up to $2k$ derivatives of $s,r,$ and $u$.
\begin{proof}
    Let's begin with the $\lesssim$ direction. Using our book-keeping scheme from Section \ref{S:Creating the Book-keeping scheme}, it will be quite straightforward to track the order of each of the terms as they appear. We plan to use Remark \ref{R:H^2k_norm_equivalence} in which the $\mathcal{H}^{2k}$ norm for $(\tilde{s}, \tilde{r}, \tilde{u})$ is shown to be equivalent to the $H^{2k, \frac{2-\gamma}{2(\gamma-1)}+ \frac{1}{2} + k} \times H^{2k, \frac{2-\gamma}{2(\gamma-1)} +k } \times H^{2k, \frac{2-\gamma}{2(\gamma-1)}+ \frac{1}{2} + k}$ norm.
    
    We start with $E^{2k}_{\text{transport}}$, which we recall as
    \begin{equation}
        \begin{split}
            E^{2k}_{\text{transport}}(\tilde{s}, \tilde{r}, \tilde{u}) &= \norm{\hat{\omega}}_{H^{2k-1,\frac{2-\gamma}{2(\gamma-1)}+ \frac{1}{2} + k }}^2 + \norm{\tilde{s}}_{H^{2k, \frac{2-\gamma}{2(\gamma-1)}+ \frac{1}{2} + k}}^2 \\
        \end{split}
    \end{equation}
    First, it is clear from Remark \ref{R:H^2k_norm_equivalence} that
    \begin{equation}
    \label{E:E2k_transport_entropy_estimate}
        \norm{\tilde{s}}_{H^{2k, \frac{2-\gamma}{2(\gamma-1)}+ \frac{1}{2} + k}}^2 \lesssim \norm{(\tilde{s}, \tilde{r}, \tilde{u})}_{\mathcal{H}^{2k}}^2
    \end{equation}
    For the vorticity part of $E^{2k}_{\text{transport}}$, we look at the order of each of the terms and observe that
    \begin{equation}
    \begin{split}
        \hat{\omega}_{\alpha \beta} &= \partial_\alpha (h \tilde{u}_\beta) -  \partial_\beta (h \tilde{u}_\alpha) \\
        &= h (\partial_\alpha \tilde{u}_\beta - \partial_\beta \tilde{u}_\alpha) +  \tilde{u}_\beta \partial_\alpha h  - \tilde{u}_\alpha \partial_\beta h  \\
        &\simeq \partial \tilde{u} + \tilde{u}
    \end{split}
    \end{equation}
    using our book-keeping scheme and recalling that $h = \frac{\Gamma+r}{\Gamma}$ which is $O(1)$ near the free boundary, as well as $\partial h$. We plan to estimate $\hat{\omega}$ using a similar argument as the estimate for $\overline{\omega}$ in \eqref{E:hat_omega_estimate_by_H2k}. Applying a multiindex $l$ with $|l|\leq 2k-1$ and counting the number of derivatives and powers of $r$, we have 
    \begin{equation}
        \partial^l \hat{\omega} \simeq \partial^{l+1} \tilde{u} + \partial^l \tilde{u}
    \end{equation}
    Applying the $L^2(r^{\frac{2-\gamma}{\gamma-1}+ 2k +1}) = H^{0,\frac{2-\gamma}{2(\gamma-1)}+ \frac{1}{2} + k }$ norm, we have
    \begin{equation}
\begin{split}
     \int_{\Omega_t} r^{\frac{2-\gamma}{\gamma-1}}|r^{k+\frac{1}{2}}\partial^l \hat{\omega}|_G^2 \hspace{1mm} dx & \lesssim \int_{\Omega_t} r^{\frac{2-\gamma}{\gamma-1}} |r^{k+\frac{1}{2}} \partial^{l+1} \tilde{u}|^2 \hspace{1mm} dx +  \int_{\Omega_t} r^{\frac{2-\gamma}{\gamma-1}} |r^{k+\frac{1}{2}} \partial^{l} \tilde{u}|^2 \hspace{1mm} dx \\
\end{split}
\end{equation}
    Summing over $l$, we obtain the $H^{2k-1,\frac{2-\gamma}{2(\gamma-1)}+ \frac{1}{2} + k }$ norm, and we will get terms that are only critical at worst, i.e. they will have the form $r^{k + \frac{1}{2}} \partial^{2k} \tilde{u}$ modulo coefficients depending on derivatives of $s$ and $r$, and the second intergral produces subcritical terms of the form $r^{k + \frac{1}{2}} \partial^{2k-1} \tilde{u}$ at worst. Each of these can easily be estimated via the $\mathcal{H}^{2k}$ norm, so we have
    \begin{equation}
        \norm{\hat{\omega}}_{H^{2k-1,\frac{2-\gamma}{2(\gamma-1)}+ \frac{1}{2} + k }}^2  \lesssim \norm{(\tilde{s}, \tilde{r}, \tilde{u})}_{\mathcal{H}^{2k}}^2.
    \end{equation}
    and combining with \eqref{E:E2k_transport_entropy_estimate} yields 
    \begin{equation}
    \label{E:E2k_transport_estimate_by_H2k}
          E^{2k}_{\text{transport}}(\tilde{s}, \tilde{r}, \tilde{u}) \lesssim \norm{(\tilde{s}, \tilde{r}, \tilde{u})}_{\mathcal{H}^{2k}}^2.
    \end{equation}

    For $E^{2k}_{\text{wave}}$, we first use Remark \ref{R:weights_in_wave_energy_tilde_H_norm} to simplfy
    \begin{equation}
    \label{E:E2k_wave_bound_from_above_start}
    \begin{split}
         E^{2k}_{\text{wave}}(\tilde{s}, \tilde{r}, \tilde{u}) &= \sum_{j=0}^{k} \norm{(D_t^{2j} \tilde{r}, D_t^{2j} \tilde{u}  )}_{\widetilde{\mathcal{H}}}^2 \\
         &\simeq \sum_{j=0}^{k} \left( \norm{D_t^{2j} \tilde{r}}_{H^{0,\frac{2-\gamma}{2(\gamma-1)} }}^2 + \norm{D_t^{2j} \tilde{u}}_{H^{0,\frac{2-\gamma}{2(\gamma-1)}+\frac{1}{2} }}^2 \right)
    \end{split}
    \end{equation}
    using the fact that norms for $\tilde{\mathcal{H}}, (L^2(r^{\frac{2-\gamma}{\gamma-1}}) \times L^2(r^{\frac{2-\gamma}{\gamma-1} +1} ))$, and $(H^{0,\frac{2-\gamma}{2(\gamma-1)}} \times H^{0,\frac{2-\gamma}{2(\gamma-1)}+\frac{1}{2} })$ are equivalent.
    Then, we make frequent use of Lemma \ref{L:Book-keeping} which allows us to greatly simplify convective derivatives. Taking $0\leq j \leq k$, we have
    \begin{equation}
         \norm{D_t^{2j} \tilde{r}}_{H^{0,\frac{2-\gamma}{2(\gamma-1)} }}^2 \simeq \nnorm{ \sum_{l=0}^j r^l \partial^{l+j} \tilde{r} }_{H^{0,\frac{2-\gamma}{2(\gamma-1)} }}^2 \lesssim \norm{\tilde{r}}_{H^{2j, \frac{2-\gamma}{2(\gamma-1)} +j }}^2
    \end{equation}
    noting that each term appearing in the summation will be critical at the $j$ level. Then, further applying our embedding lemmas, we get 
    \begin{equation}
    \begin{split}
        \norm{\tilde{r}}_{H^{2j, \frac{2-\gamma}{2(\gamma-1)} +j }}^2 &\lesssim  \norm{\tilde{r}}_{H^{2j+k-j, \frac{2-\gamma}{2(\gamma-1)} +j + k -j }}^2 \\
        &=  \norm{\tilde{r}}_{H^{k+j, \frac{2-\gamma}{2(\gamma-1)} +k }}^2 \\
        & \lesssim \norm{\tilde{r}}_{H^{2k, \frac{2-\gamma}{2(\gamma-1)} +k }}^2 \lesssim \norm{(\tilde{s}, \tilde{r}, \tilde{u})}_{\mathcal{H}^{2k}}^2
    \end{split}
    \end{equation}
    noting that $0 \leq j \leq k$ and we can apply Corollary \ref{C:Trading_deriv_for_weight} as needed when $j<k$. Thus, the $\tilde{r}$ part of $E^{2k}_{\text{wave}}$ is bounded by the desired $\mathcal{H}^{2k}$ norm. For the $\tilde{u}$ part, we can similarly invoke Lemma \ref{L:Book-keeping} again, where the only difference is that we now have terms involving both $\tilde{u}$ and $\tilde{r}$. However, the extra $r^{1/2}$ weight is exactly what we need to handle any of the terms that appear:
    \begin{equation}
    \begin{split}
    \label{E:E2k_wave_bound_from_above_finish}
         \norm{D_t^{2j} \tilde{u}}_{H^{0,\frac{2-\gamma}{2(\gamma-1)} + \frac{1}{2}}}^2 &\simeq \nnorm{ \sum_{l=0}^{j} r^{l} \partial^{l+j} \tilde{u} + \sum_{i=0}^{j -1} r^i \partial^{i+j} \tilde{r} }_{H^{0,\frac{2-\gamma}{2(\gamma-1)}+ \frac{1}{2} }}^2 \\
         & \lesssim \norm{(\tilde{s}, \tilde{r}, \tilde{u})}_{\mathcal{H}^{2k}}^2 \\
    \end{split}
    \end{equation}
    and we observe that the $\tilde{r}$ terms actually have an extra $r^{1/2}$ in terms of bounding by the $\mathcal{H}^{2k}$ norm. Combining \eqref{E:E2k_wave_bound_from_above_start}-\eqref{E:E2k_wave_bound_from_above_finish} produces
    \begin{equation}
        E^{2k}_{\text{wave}}(\tilde{s}, \tilde{r}, \tilde{u}) \lesssim \norm{(\tilde{s}, \tilde{r}, \tilde{u})}_{\mathcal{H}^{2k}}^2.
    \end{equation}
    Thus, we combine with \eqref{E:E2k_transport_estimate_by_H2k} and the $\lesssim$ direction for Theorem \ref{Th:E^2k_equiv_H^2k} is complete.

    For the $\gtrsim$ direction, we make use of many previously established lemmas. First, by the norm equivalence from Remark \ref{R:H^2k_norm_equivalence}, and Lemmas 
    \ref{L:Higher_order_elliptic_estimates_tilde_r}, and \ref{L:Higher_order_div_curl_estimates_tilde_u}, we have
    \begin{equation}
    \begin{split}
    \label{E:Energy_equivalence_key_step_1}
         \norm{(\tilde{s}, \tilde{r}, \tilde{u})}_{\mathcal{H}^{2k}}^2 & \approx  \norm{\tilde{r}}_{H^{2k, \frac{2-\gamma}{2(\gamma-1)}+k }}^2+ \norm{\tilde{u}}_{H^{2k, \frac{2-\gamma}{2(\gamma-1)}+ \frac{1}{2} +k }}^2  + \norm{\tilde{s}}_{H^{2k, \frac{2-\gamma}{2(\gamma-1)}+ \frac{1}{2} +k }}^2 \\
         & \lesssim \sum_{j=0}^k \left(\norm{D_t^{2j} \tilde{r}}^2_{H^{0,\frac{2-\gamma}{2(\gamma-1)}}} + \norm{D_t^{2j} \tilde{u}}^2_{H^{0,\frac{2-\gamma}{2(\gamma-1)}+ \frac{1}{2}}} \right) \\
         & \hspace{5mm }+ \norm{\overrightarrow{\text{curl}} \hspace{1mm} \tilde{u}}_{H^{2k-1, k+ \frac{1}{2(\gamma-1)}}}^2 + \norm{\tilde{s}}_{H^{2k, \frac{2-\gamma}{2(\gamma-1)}+ \frac{1}{2} +k }}^2\\
         & \hspace{5mm} + (C_{\hat{\varepsilon}}^2 +D_{\hat{\varepsilon}}^2 ) \norm{(\tilde{s}, \tilde{r}, \tilde{u})}_{\mathcal{H}^{2k}}^2 \\
    \end{split}
    \end{equation}
    where $C_{\hat{\varepsilon}}, D_{\hat{\varepsilon}} \ll 1$.
    Then, we observe by examining $E^{2k}_{\text{wave}}$ in \eqref{E:E^2k_wave} that
    \begin{equation}
    \label{E:E2k_wave_from_above_H_tilde_norm_weights}
        \sum_{j=0}^k \left(\norm{D_t^{2j} \tilde{r}}^2_{H^{0,\frac{2-\gamma}{2(\gamma-1)}}} + \norm{D_t^{2j} \tilde{u}}^2_{H^{0,\frac{2-\gamma}{2(\gamma-1)}+ \frac{1}{2}}} \right) \lesssim \sum_{j=0}^{k} \norm{(D_t^{2j} \tilde{r}, D_t^{2j} \tilde{u}  )}_{\widetilde{\mathcal{H}}}^2 = E^{2k}_{\text{wave}}(\tilde{s}, \tilde{r}, \tilde{u})
    \end{equation}
    as any weights involving $\Gamma$ can be removed in view of Remark \ref{R:weights_in_wave_energy_tilde_H_norm}. Then, combining \eqref{E:Energy_equivalence_key_step_1}, \eqref{E:E2k_wave_from_above_H_tilde_norm_weights}, and applying Lemma \ref{L:relating_vorticity_spatial_curl} for the term with $\overrightarrow{\text{curl}} \hspace{1mm} \tilde{u}$, we have
\begin{equation}
    \begin{split}
         \norm{(\tilde{s}, \tilde{r}, \tilde{u})}_{\mathcal{H}^{2k}}^2 &\lesssim E^{2k}_{\text{wave}}(\tilde{s}, \tilde{r}, \tilde{u}) +  \norm{\hat{\omega}}_{H^{2k-1,\frac{2-\gamma}{2(\gamma-1)}+ \frac{1}{2} + k }}^2  +   \norm{\tilde{s}}_{H^{2k, \frac{2-\gamma}{2(\gamma-1)}+ \frac{1}{2} +k }}^2 \\\
         &\hspace{5mm} + \hat{\varepsilon}^2\norm{\tilde{u}}_{H^{2k, \frac{2-\gamma}{2(\gamma-1)}+ \frac{1}{2} +k }}^2 + (C_{\hat{\varepsilon}}^2 +D_{\hat{\varepsilon}}^2) \norm{(\tilde{s}, \tilde{r}, \tilde{u})}_{\mathcal{H}^{2k}}^2 \\
         &\lesssim E^{2k}_{\text{wave}}(\tilde{s}, \tilde{r}, \tilde{u}) + E^{2k}_{\text{transport}}(\tilde{s}, \tilde{r}, \tilde{u}) + (C_{\hat{\varepsilon}}^2 +D_{\hat{\varepsilon}}^2 + \hat{\varepsilon}^2) \norm{(\tilde{s}, \tilde{r}, \tilde{u})}_{\mathcal{H}^{2k}}^2
    \end{split} 
\end{equation}
Absorbing the constants $(C_{\hat{\varepsilon}}^2 +D_{\hat{\varepsilon}}^2 + \hat{\varepsilon}^2)\ll 1$ into the LHS produces
\begin{equation}
    E^{2k}(\tilde{s}, \tilde{r}, \tilde{u})=E^{2k}_{\text{wave}}(\tilde{s}, \tilde{r}, \tilde{u}) + E^{2k}_{\text{transport}}(\tilde{s}, \tilde{r}, \tilde{u}) \gtrsim  \norm{(\tilde{s}, \tilde{r}, \tilde{u})}_{\mathcal{H}^{2k}}^2
\end{equation}
which completes the proof of the $\gtrsim$ direction. Thus, the proof of Theorem \ref{Th:E^2k_equiv_H^2k} is complete.
\end{proof} 

\end{theorem}

\section{Higher order wave energy estimates}
\label{S:Higher_order_energy_estimates}

Using the book-keeping scheme from Section \ref{S:Creating the Book-keeping scheme}, let's estimate each of the terms on the RHS of \eqref{E:Linearized_System_6} with the following Theorem. Note that the $B_{2k}$ equation will get multiplied by $ \frac{1}{\gamma-1} r^{\frac{2-\gamma}{\gamma-1}} D_t^{2k} \tilde{r}$, and the $C^\alpha_{2k}$ equation will get contracted with $ (\Gamma+r) r^{\frac{1}{\gamma-1}} G_{\alpha \beta} D_t^{2k} \tilde{u}^\beta$. 

We plan to integrate and estimate using the $\mathcal{H}^{2k}$ norm
which we recall:
\begin{equation*}
\begin{split}
     \norm{(\tilde{s}, \tilde{r}, \tilde{u})}_{\mathcal{H}^{2k}}^2
     &= \sum_{|\alpha| =0}^{2k} \sum_{\substack{a=0 \\ |\alpha|-a \leq k}}^k \norm{r^{\frac{2-\gamma}{2(\gamma -1)} + \frac{1}{2} +a } \partial^{\alpha } \tilde{s}}_{L^2}^2 +  \norm{r^{\frac{2-\gamma}{2(\gamma -1)} +a } \partial^{\alpha } \tilde{r}}_{L^2}^2 + \norm{r^{\frac{2-\gamma}{2(\gamma -1)} + \frac{1}{2} +a  } \partial^{\alpha } \tilde{u}}_{L^2}^2 \\
\end{split}
\end{equation*}

The following book-keeping remark is useful when taking convective derivatives of the many expressions that appear:
\begin{remark}
\label{R:Book_keeping_proj_Gamma_etc}
    Using the fact that $D_t (\Gamma(s)) = 0$ from \eqref{E:System_s_equation}, we have the following observations:
\begin{equation*}
\begin{split}
    & D_t \left( \frac{1}{\Gamma+r} \right) = - \frac{1}{(\Gamma+r)^2} D_t r\simeq \left(\frac{1}{\Gamma+r}\right)^2 r 
\end{split}
\end{equation*}
where we applied \eqref{E:D_t r equals r, convective derivative of r}. Thus, taking multiple convective derivatives of our weight $\frac{1}{\Gamma+r}$ leads to gaining additional powers of $r$, and thus makes the terms more subcritical. In view of Remarks \ref{D:definition_critical_subcritical_supercritical} and \ref{R:simeq_symbol_definition}, we will often ignore these terms and use the $\simeq$ symbol so that we can highlight the key critical/supercritical terms that cause difficulty when estimating.  
\end{remark}

\begin{theorem}[Higher Order energy estimates]
\label{Th:Higher_order_energy_estimates}
Let $(s,r,u)$ be a solution to \eqref{E:System} that exists on some time interval $[0,T]$. Assume that $s,r,u$ and their $2k+1$ order derivatives are bounded in the $L^\infty(\Omega_t)$ norm for each $t \in [0,T]$ and $r$ vanishes simply on the free boundary. Then, the following estimate holds for solutions $( \tilde{s}, \tilde{r}, \tilde{u})$ to the higher order linearized equations \eqref{E:Linearized_System_6}:
\begin{equation}
    \left|\frac{d}{dt} E^{2k}_{\text{wave}} (\tilde{s}, \tilde{r}, \tilde{u}) \right| \lesssim B \norm{(\tilde{s}, \tilde{r}, \tilde{u})}_{\mathcal{H}^{2k}}^2
\end{equation}
where $B$ is a function that depends on up to $2k+1$ derivatives of $s, r$, and $u$.

\begin{proof}
The proof relies on Section \ref{S:Basic_energy_estimate} and the basic energy estimate in Proposition \ref{P:Basic_Energy_Inequality}. Along the way, we will make use of Lemmas  \ref{L:Sum_orders}, \ref{L:Commutator}, \ref{L:Book-keeping}, and \ref{L:Order_of_operators}. 

We start multiplying \eqref{E:Linearized_System_6_r_equation} and \eqref{E:Linearized_System_6_v_equation} by their respective multipliers $\frac{1}{\gamma-1} r^{\frac{2-\gamma}{\gamma-1}} D_t^{2k} \tilde{r}$ and $(\Gamma+r) r^{\frac{1}{\gamma-1}} G_{\alpha \beta} D_t^{2k} \tilde{u}^\beta$. Then, we compare with \eqref{E:Basic_energy_estimate_tilde_r_equation} and \eqref{E:Basic_energy_estimate_tilde_u_equation} and observe that we will be following the Proof of Proposition \ref{P:Basic_Energy_Inequality} with $\tilde{r}$ and $\tilde{u}$ replaced by $D_t^{2k} \tilde{r}$ and $D_t^{2k} \tilde{u}$. Integrating over $\Omega_t$ and applying the moving domain formula \eqref{E:Moving_Domain_formula}, it remains to show that all of the terms on the RHS can be estimated using the $\mathcal{H}^{2k}$ norm. More specifically, we plan to show that 
\begin{subequations}
\label{E:higher_order_Wave_estimates_goal}
\begin{align}
\label{E:higher_order_Wave_estimates_1}
    \int_{\Omega_t} r^{\frac{2-\gamma}{\gamma-1}} \left|B_{2k} \frac{1}{\gamma-1}  D_t^{2k} \tilde{r}\right| \hspace{1mm} dx 
    &\lesssim \mathcal{C}_1 \norm{(\tilde{s}, \tilde{r}, \tilde{u})}_{\mathcal{H}^{2k}}^2 \\
\label{E:higher_order_Wave_estimates_2}
        \int_{\Omega_t} r^{\frac{2-\gamma}{\gamma-1}} 
    \left|C_{2k}^\alpha (\Gamma+r) r G_{\alpha \beta} D_t^{2k} \tilde{u}^\beta \right| \hspace{1mm} dx &\lesssim \mathcal{C}_2 \norm{(\tilde{s}, \tilde{r}, \tilde{u})}_{\mathcal{H}^{2k}}^2
\end{align}
\end{subequations}
where $B_{2k}$ and $C_{2k}^\alpha$ are long expressions defined in \eqref{E:G_and_H_definition}.
Recall from Remarks \ref{R:order_free_boundary_term} and \ref{R:Sum_orders} that when calculating the order of a given term where $\gamma>1$ is arbitrary, we simply ignore powers of $r$ coming from the weight $r^{\frac{2-\gamma}{\gamma-1}}$, since our estimates are built around the weighted space $L^2(r^{\frac{2-\gamma}{\gamma-1}}) = H^{0,\frac{2-\gamma}{2(\gamma-1)} }$. Also, recall Lemma \ref{L:Sum_orders} and Remark \ref{R:Sum_orders} where estimates involving $L^1(r^{\frac{2-\gamma}{\gamma-1}})$ are used extensively.

We will begin by showing \eqref{E:higher_order_Wave_estimates_1}. First, observe that the multiplier $D_t^{2k} \tilde{r}$ is $\mathcal{H}^{2k}$ critical since by Lemma \ref{L:Book-keeping}, we have
\begin{equation*}
    D_t^{2k} \tilde{r} \simeq (r^k \partial^{2k} \tilde{r} + r^{k-1} \partial^{2k-1} \tilde{r} +  ... + \partial^k \tilde{r}) + ... 
\end{equation*}
where one can quickly check the order of the terms listed as being $\mathcal{H}^{2k}$ critical, and all remaining terms are subcritical.

For the first term in \eqref{E:higher_order_Wave_estimates_1} after expanding $B_{2k}$ using \eqref{E:G_and_H_definition}, we get
\begin{equation}
\begin{split}
     D_t^{2k} \tilde{r} D_t^{2k} g &= D_t^{2k} \tilde{r} D_t^{2k} (\tilde{r} \partial_\mu u^\mu) \\
     &= D_t^{2k} \tilde{r} \sum_{i=0}^{2k} D_t^i \tilde{r} D_t^{2k-i}(\partial_\mu u^\mu) \\
\end{split}
\end{equation}
Then, by computing the order of terms in Lemma \ref{L:Book-keeping}, we see that $D_t^i \tilde{r}$ only contains $\mathcal{H}^{2k}$-critical terms when $i=2k$. After integrating in $L^1(r^{\frac{2-\gamma}{\gamma-1}})$, the product can estimated by making use of Lemma \ref{L:Sum_orders} since we have the product of two terms which are both critical:
\begin{equation}
      \int_{\Omega_t} r^{\frac{2-\gamma}{\gamma-1}}\left| D_t^{2k} \tilde{r} \sum_{i=0}^{2k} D_t^i \tilde{r} D_t^{2k-i}(\partial_\mu u^\mu) \right| \hspace{1mm} dx \lesssim C_1 \norm{(\tilde{s}, \tilde{r}, \tilde{u})}_{\mathcal{H}^{2k}}^2 
\end{equation}
where $C_1$ depends on the $L^\infty$ norms of up to $2k+1$ derivatives of $(s,r,u)$. Notice that in the coefficient expression $D_t^{2k-i}(\partial_\mu u^\mu)$, we can use Lemma \ref{L:Solving_for_time_derivatives} to solve for the time derivative of $u$ in terms of spatial derivatives. Additionally, we can solve for $D_t^{2k-i} (\partial u)$ in terms of spatial derivatives of $u$ with additional powers of $r$ in a similar way as Lemma \ref{L:Book-keeping} be returning to \eqref{E:System_v_equation}.

For the second term in \eqref{E:higher_order_Wave_estimates_1}, we get
\begin{equation}
\begin{split}
    D_t^{2k} \tilde{r} \sum_{i=0}^{2k-1} \left(D_t^i \tilde{u}^\mu \partial_\mu \left(D_t^{2k-i} r \right)- D_t^{2k-i} r  \partial_\mu \left(D_t^i  \tilde{u}^\mu \right) \right)
\end{split}
\end{equation}
where each piece will be handled separately. For the first piece with $D_t^i \tilde{u}$, we see by Lemma \ref{L:Commutator} that $D_t^i \tilde{u}$ only contains critical terms at worst when $i = 2k-1$ , and all other terms are subcritical. Thus, the product can be estimated using Lemma \ref{L:Sum_orders} since $D_t^{2k} \tilde{r}$ is also critical:
\begin{equation}
      \int_{\Omega_t} r^{\frac{2-\gamma}{\gamma-1}} \left|D_t^{2k} \tilde{r} \sum_{i=0}^{2k-1} D_t^i \tilde{u}^\mu \partial_\mu \left(D_t^{2k-i} r \right) \right| \hspace{1mm} dx  \lesssim  C_{2,1} \norm{(\tilde{s}, \tilde{r}, \tilde{u})}_{\mathcal{H}^{2k}}^2
\end{equation}
For the second piece, with $ \partial_\mu(D_t^i \tilde{u}^\mu)$ we see that $D_t^{2k-i} r$ contains one power of $r$ by \eqref{E:D_t r equals r, convective derivative of r}, and up to $2k-i$ derivatives of $r$ and $u$. Similar to the first piece, it suffices to check when $i=2k-1$ in which case $D_t^{2k-1} \tilde{u}$ contains terms that are $\mathcal{H}^{2k}$ critical, and terms that are subcritical with order $\alpha = \frac{1}{2}$. Simplifying the expression when $i=2k-1$ and solving for time derivatives in terms of spatial derivatives with Lemma \ref{L:Solving_for_time_derivatives}, we will get
\begin{equation}
\label{E:r_partial_(critical)}
\begin{split}
      D_t r \partial_\mu (D_t^{2k-1} \tilde{u}^\mu) &\simeq r \partial \left(\left(\text{terms with order $\alpha=0$}\right) + \left(\text{terms with order $\alpha=\frac{1}{2}$}\right) \right) 
      \end{split}
\end{equation}
Then, by applying Lemma \ref{L:Order_of_operators}, we see that applying $r \partial$ to any free boundary term will produce a new term with the same order, i.e. here we have $\alpha =0$ and $\alpha = \frac{1}{2}$ respectively. Thus, the second piece can be estimated using Lemma \ref{L:Sum_orders}
\begin{equation}
    \int_{\Omega_t} r^{\frac{2-\gamma}{\gamma-1}}\left|D_t^{2k} \tilde{r} \sum_{i=0}^{2k-1} D_t^{2k-i}r \partial_\mu \left(D_t^{i} \tilde{u} \right) \right| \hspace{1mm} dx \lesssim C_{2,2} \norm{(\tilde{s}, \tilde{r}, \tilde{u})}_{\mathcal{H}^{2k}}^2
\end{equation}

For the third term in \eqref{E:higher_order_Wave_estimates_1}, the estimate is quite similar to before as $D_t^i \tilde{u}$ only contains critical terms when $i = 2k-1$. By Lemma \ref{L:Commutator}, we get
\begin{equation}
    \begin{split}
        [\partial_\mu, D_t^{2k-i}] r &= C_{2k-i-1}^\nu \partial_\nu (D_t^{2k-i-1} r) +  C_{2k-i-2}^\nu \partial_\nu (D_t^{2k-i-2} r) \\
        & \hspace{5mm}+ ... \\
        & \hspace{5mm}+ C_1^\nu \partial_\nu (D_t r) \\
        & \hspace{5mm}+ C_0^\nu \partial_\nu r \\
    \end{split}
\end{equation}
which consists of derivatives of $u$ and $r$. Then, by Lemma \ref{L:Sum_orders}, the estimate will look like
\begin{equation}
\begin{split}
   \int_{\Omega_t} r^{\frac{2-\gamma}{\gamma-1}} \left| D_t^{2k} \tilde{r} \sum_{i=0}^{2k-1} D_t^i \tilde{u}^\mu [D_t^{2k-i}, \partial_\mu]r \right| \hspace{1mm} dx & \lesssim C_3 \norm{(\tilde{s}, \tilde{r}, \tilde{u})}_{\mathcal{H}^{2k}}^2
\end{split}
\end{equation}
since we have the product of terms that are both $\mathcal{H}^{2k}$-critical at worst.

To estimate the $G_{2k}$ fourth term, we will use the commutator identities from Lemma \ref{L:Commutator}:
\begin{equation}
\begin{split}
     D_t^{2k} \tilde{r} \sum_{i=1}^{2k} D_t^{2k-i} r [D_t^i, \partial_\mu] \tilde{u}^\mu &\simeq D_t^{2k} \tilde{r} \sum_{i=1}^{2k} D_t^{2k-i} r \left(\sum_{j=0}^{i-1} C^\nu_j \partial_\nu (D_t^j \tilde{u}^\mu) \right)
\end{split}
 \end{equation}
 
Now, the worst possible terms only occur when $i=2k$, in which case we will get $D_t^0 r = r$ and $D_t^{2k}\tilde{r} \left(r \partial (D_t^{2k-1} \tilde{u}) \right) $. This term can be handled in a similar way as \eqref{E:r_partial_(critical)} since we get the product of terms which are both critical. In fact, when $i \leq 2k-1$, we get that all terms are subcritical and easily estimated. Thus, we have
\begin{equation}
    \int_{\Omega_t} r^{\frac{2-\gamma}{\gamma-1}} \left|  D_t^{2k} \tilde{r} \sum_{i=1}^{2k} D_t^{2k-i} r [D_t^i, \partial_\mu] \tilde{u}^\mu \right| \hspace{1mm} dx \lesssim C_4 \norm{(\tilde{s}, \tilde{r}, \tilde{u})}_{\mathcal{H}^{2k}}^2
\end{equation}
Combining the previous inequalities proves \eqref{E:higher_order_Wave_estimates_1}.

Let's turn now to the $C_{2k}^\alpha$ terms occuring in \eqref{E:higher_order_Wave_estimates_2}. Recall from \eqref{E:higher_order_Wave_estimates_2} that the $C_{2k}^\alpha$ has the multiplier $(\Gamma + r )r G_{\alpha \beta }D_t^{2k} \tilde{u}^\beta$. In general, we observe using Lemma \ref{L:Book-keeping} that the multiplier $r D_t^{2k} \tilde{u}^\beta$ has $\mathcal{H}^{2k}$ subcritical terms with order $1/2$ at worst since
\begin{equation*}
\begin{split}
     D_t^{2k} \tilde{u} & \simeq \sum_{l=0}^{k} r^{l} \partial^{l+k} \tilde{u} + \sum_{j=0}^{k -1} r^j \partial^{j+k} \tilde{r} \\
    & \simeq \left(\text{$\alpha = -1/2$ supercritical terms} \right) + \left(\text{$\alpha = 0$ critical terms} \right) \\
    \implies r D_t^{2k} \tilde{u} & \simeq \sum_{l=0}^{k} r^{l+1} \partial^{l+k} \tilde{u} + \sum_{j=0}^{k -1} r^{j+1} \partial^{j+k} \tilde{r}\\
    &\simeq \left(\text{$\alpha = 1/2$ subcritical terms} \right) + \left(\text{$\alpha = 1$ subcritical terms} \right) \\
\end{split}  
\end{equation*}

For the first term in \eqref{E:higher_order_Wave_estimates_2}, we get
\begin{equation}
    \begin{split}
        (\Gamma + r )r G_{\alpha \beta}  D_t^{2k} \tilde{u}^\beta D_t^{2k} h^\alpha 
    \end{split}
\end{equation}
and we recall the from \eqref{E:Linearized_f_g_h} that
\begin{equation}
\begin{split}
     D_t^{2k} h^\alpha &= D_t^{2k} \left(-\tilde{u}^\mu \partial_\mu u^\alpha - \frac{1}{\Gamma +r}(\tilde{u}^\alpha u^\mu + u^\alpha \tilde{u}^\mu)\partial_\mu r + \frac{\Gamma'}{(\Gamma+r)^2} \tilde{s} \proj^{\alpha \mu}\partial_\mu r + \frac{1}{(\Gamma+r)^2} \tilde{r} \proj^{\alpha \mu} \partial_\mu r \right) \\
     &\simeq -\sum_{i=0}^{2k} D_t^i \tilde{u}^\mu D_t^{2k-i}(\partial_\mu u^\alpha) - \frac{1}{\Gamma +r} \sum_{i=0}^{2k} D_t^i (\tilde{u}^\alpha u^\mu + u^\alpha \tilde{u}^\mu) D_t^{2k-i}(\partial_\mu r) \\
     & \hspace{1cm} + \frac{\Gamma'}{(\Gamma +r)^2}\sum_{i=0}^{2k} D_t^i \tilde{s} D_t^{2k-i} \left(\proj^{\alpha \mu} \partial_\mu r \right)
     + \frac{1}{(\Gamma +r)^2}\sum_{i=0}^{2k} D_t^i \tilde{r} D_t^{2k-i} \left(\proj^{\alpha \mu} \partial_\mu r \right) \\
\end{split}
\end{equation}
where we recall that $D_t^{2k} \left(\frac{1}{\Gamma+r}\right) \simeq -\frac{1}{(\Gamma+r)^2} D_t^{2k} r$ using Remark \ref{R:Book_keeping_proj_Gamma_etc}, and this only contributes additional powers of $r$. After observing $D_t^{2k} h^\alpha$, we see that it contains terms of the form $D_t^{2k} (\tilde{u}, \tilde{s}, \tilde{r})$ which are order $-\frac{1}{2}$ supercritical at worst. However, those terms all get multiplied by $r D_t^{2k} \tilde{u}$ which has order $\frac{1}{2}$. Thus, by Lemma \ref{L:Sum_orders}, these terms are also estimated using the $\mathcal{H}^{2k}$ norm and we can solve for any expressions like $D_t^{2k-i}(\proj^{\alpha \mu} \partial_\mu r)$ in terms of spatial derivatives of $u, r,$ and $s$.

For the second term in \eqref{E:higher_order_Wave_estimates_2}, we get the following:
\begin{equation}
        (\Gamma +r) r G_{\alpha \beta} D_t^{2k} \tilde{u}^\beta \sum_{i=0}^{2k-1} D_t^{2k-i} \left(\frac{1}{\Gamma+r} \proj^{\alpha \mu} \right) \partial_\mu \left(D_t^i \tilde{r}\right) 
\end{equation}
Now, when $i=2k-1$, $D_t^{2k-1} \tilde{r}$ contains terms that have order $\frac{1}{2}$ at worst by Remarks \ref{D:definition_critical_subcritical_supercritical} and \ref{R:Order_of_operators}. This implies that $\partial (D_t^{2k-1} \tilde{r})$ will contain terms that are order $-\frac{1}{2}$ at worst by Lemma \ref{L:Order_of_operators}. However, the multiplier $rD_t^{2k} \tilde{u}$has order $\frac{1}{2}$, which is just enough to allow us to apply Lemma \ref{L:Sum_orders}. Thus, we can estimate this term using the $\mathcal{H}^{2k}$ norm. 

For the third and final term in \eqref{E:higher_order_Wave_estimates_2}, we can follow the analysis of the second term and use Lemma \ref{L:Commutator} to get
\begin{equation}
    \begin{split}
       & (\Gamma +r) r G_{\alpha \beta} D_t^{2k} \tilde{u}^\beta \sum_{i=1}^{2k} D_t^{2k-i} \left(\frac{1}{\Gamma+r} \proj^{\alpha \mu} \right) [\partial_\mu , D_t^i] \tilde{r} \\
         & \hspace{1cm} \simeq (\Gamma +r) G_{\alpha \beta} r D_t^{2k} \tilde{u}^\beta \sum_{i=1}^{2k} D_t^{2k-i} \left(\frac{1}{\Gamma+r} \proj^{\alpha \mu} \right) \left(\sum_{j=0}^{i-1} C^\nu_j \partial_\nu (D_t^j \tilde{r}) \right) \\
    \end{split}
\end{equation}

Since the worst terms appear when $i=2k$, we see that once again $D_t^{2k-1} \tilde{r}$ has terms with order $\frac{1}{2}$ at worst, thus, $\partial (D_t^{2k-1} \tilde{r})$ has terms with $-\frac{1}{2}$ order. However, the multiplier has order $\frac{1}{2}$, and this allows us to apply Lemma $\ref{L:Sum_orders}$ as usual. This completes estimate \eqref{E:higher_order_Wave_estimates_2}, and all higher order terms on the RHS in \eqref{E:Linearized_System_6} have been handled. 
\end{proof}
\end{theorem}


\section{Main Theorem}
\label{S:Main_Theorem}

Combining the previous sections, we have reached our final result:

\begin{theorem}[Estimates in $\mathcal{H}^{2k}$]
\label{Th:Main_Theorem_Estimates_in_H^2k}
Let $(s,r,u)$ be a smooth solution to \eqref{E:System} that exists on some time interval $[0,T]$, and for which the physical vacuum boundary condition \eqref{E:Physical_boundary_ideal_gas} holds. Let $(\tilde{s}_0, \tilde{r}_0, \tilde{u}_0)$ be initial data to system \eqref{E:Linearized_System_2}. Then, there exists a constant $\mathcal{C}$ depending only on $s, r, u,$ and $T$ such that, if $(\tilde{s}, \tilde{r}, \tilde{u}) \in C^{\infty}(\overline{\mathscr{D}})$ is a solution to \eqref{E:Linearized_System_2} on $[0,T]$, then
\begin{equation}
    \norm{(\tilde{s}, \tilde{r}, \tilde{u})}_{\mathcal{H}^{2k}(\Omega_t)} \lesssim \mathcal{C} \norm{(\tilde{s}_0, \tilde{r}_0, \tilde{u}_0)}_{\mathcal{H}^{2k}(\Omega_0)}
\end{equation}
where $\mathcal{H}^{2k}(\Omega_t)$ is defined in \eqref{E:Higher_order_H^2k_norm}.
\begin{proof}
Combining Proposition \ref{P:Entropy Estimates} and Theorem \ref{Th:Higher_order_energy_estimates}, we get the following
\begin{equation}
    \frac{d}{dt} \left(\norm{\tilde{s}}^2_{H^{2k, k + \frac{2-\gamma}{2(\gamma-1)}+ \frac{1}{2} }}+  E^{2k}_{\text{wave}}(\tilde{s}, \tilde{r}, \tilde{u}) \right) \lesssim \mathcal{C}_1 \norm{(\tilde{s}, \tilde{r}, \tilde{u})}^2_{\mathcal{H}^{2k}} .
\end{equation}
where $\mathcal{C}_1$ is a constant depending on $L^\infty(\Omega_t)$ norm of up to $2k+1$ derivatives of $s,r$, and $u$. 
Then, after adding $\frac{d}{dt} \norm{\hat{\omega}}_{H^{2k-1, k+ \frac{2-\gamma}{2(\gamma-1)}+ \frac{1}{2}}}^2$ to both sides and recalling \eqref{E:E^2k_transport} and \eqref{E:E^2k_full_linearized_energy}, we have
\begin{equation}
    \frac{d}{dt} E^{2k}_{\text{total}}(\tilde{s}, \tilde{r}, \tilde{u}) \lesssim \mathcal{C}_1 \norm{(\tilde{s}, \tilde{r}, \tilde{u})}^2_{\mathcal{H}^{2k}} + \frac{d}{dt} \norm{\hat{\omega}}_{H^{2k-1, k+ \frac{2-\gamma}{2(\gamma-1)}+ \frac{1}{2}}}^2.
\end{equation}
Integrating both sides and using Theorem \ref{Th:E^2k_equiv_H^2k} on the energy equivalence, we have
\begin{equation}
     \norm{(\tilde{s}, \tilde{r}, \tilde{u})}^2_{\mathcal{H}^{2k}(\Omega_t)} -  \norm{(\tilde{s}_0, \tilde{r}_0, \tilde{u}_0)}^2_{\mathcal{H}^{2k}(\Omega_0)} \lesssim \int_0^t \mathcal{C}_1 \norm{(\tilde{s}, \tilde{r}, \tilde{u})}^2_{\mathcal{H}^{2k}(\Omega_\tau)} \hspace{1mm} d\tau + \norm{\hat{\omega}}_{H^{2k-1, k+ \frac{2-\gamma}{2(\gamma-1)}+ \frac{1}{2}}(\Omega_t)}^2 .
\end{equation}
Applying Theorem \ref{Th:Transport_Energy_estimates} for the $\hat{\omega}$ term and rearranging, we get
\begin{equation}
\begin{split}
     \norm{(\tilde{s}, \tilde{r}, \tilde{u})}^2_{\mathcal{H}^{2k}(\Omega_t)} &\lesssim  \norm{(\tilde{s}_0, \tilde{r}_0, \tilde{u}_0)}^2_{\mathcal{H}^{2k}(\Omega_0)} + \int_0^t \mathcal{C}_1 \norm{(\tilde{s}, \tilde{r}, \tilde{u})}^2_{\mathcal{H}^{2k}(\Omega_\tau)} \hspace{1mm} d\tau \\
     & \hspace{5mm} + \norm{\tilde{\omega}_0}_{H^{2k-1, k+ \frac{1}{2(\gamma-1)}} (\Omega_0)}^2 + \hat{\varepsilon} \norm{(\tilde{s}, \tilde{r}, \tilde{u})}_{\mathcal{H}^{2k}(\Omega_t) }^2 + \int_0^t \mathcal{C}_2 \norm{(\tilde{s}, \tilde{r}, \tilde{u})}_{\mathcal{H}^{2k}(\Omega_\tau)}^2 \hspace{1mm} d \tau \\
\end{split}
\end{equation}
Then, combining the initial data $\tilde{\omega}_0$ with the initial data term $ \norm{(\tilde{s}_0, \tilde{r}_0, \tilde{u}_0)}^2_{\mathcal{H}^{2k}(\Omega_0)}$, and soaking the term with $\hat{\varepsilon}$ to the LHS, we arrive at
\begin{equation}
     \norm{(\tilde{s}, \tilde{r}, \tilde{u})}^2_{\mathcal{H}^{2k}(\Omega_t)} \lesssim \frac{1}{1- \hat{\varepsilon}} \left( \norm{(\tilde{s}_0, \tilde{r}_0, \tilde{u}_0)}^2_{\mathcal{H}^{2k}(\Omega_0)} + \int_0^t \mathcal{C}_3 \norm{(\tilde{s}, \tilde{r}, \tilde{u})}_{\mathcal{H}^{2k}(\Omega_\tau)}^2 \hspace{1mm} d \tau  \right)
\end{equation}
Finally, the desired result is a straightforward application of Gronwall's inequality.

\end{proof}
\end{theorem}

\newpage

\section{Appendix for Lemma \ref{L:Solving_for_time_derivatives}}
\label{S:Appendix}

Here, we record some of the computations for proving Lemma \ref{L:Solving_for_time_derivatives}.
Using \eqref{E:System} directly, we have
\begin{equation}
\label{E:start_solving_for_time}
    \begin{split}
        \partial_t s &= - \frac{u^i}{u^0} \partial_i s \\
        \partial_t r &= - \frac{u^i}{u^0} \partial_i r - \frac{\gamma-1}{u^0} r \partial_t u^0 - \frac{\gamma-1}{u^0} r \partial_i u^i \\
        \partial_t u^\alpha  &= - \frac{u^i}{u^0} \partial_i u^\alpha - \frac{1}{(\Gamma+r)u^0} \proj^{\alpha 0} \partial_t r -  \frac{1}{(\Gamma+r)u^0} \proj^{\alpha i} \partial_i r
    \end{split}
\end{equation}

Starting from \eqref{E:start_solving_for_time},  we can further solve for $\partial_t u^0$ by plugging in $\alpha=0$ into the last equation. If we set $a_1 = u^0 - \frac{\gamma-1}{(\Gamma+r)u^0}r \proj^{00}$, then 
\begin{equation}
\label{E:partial_t u^0}
\begin{split}
     \partial_t u^0 &= \frac{1}{a_1} \left(-u^i \partial_i u^0 + \frac{\gamma-1}{(\Gamma+r)u^0}r \proj^{00} \partial_i u^i + \frac{1}{(\Gamma+r)u^0}\proj^{00} u^i \partial_i r - \frac{1}{\Gamma+r} \proj^{0i} \partial_i r \right) \\
     &= -\frac{1}{a_1} u^i \partial_i u^0 - \frac{1}{a_1(\Gamma+r)u^0} u^i \partial_i r + \frac{(\gamma-1)((u^0)^2-1)}{a_1(\Gamma+r)u^0}r \partial_i u^i \\
     &:= C_1^i \partial_i u^0 + C_2^i \partial_i r + C_3 r \partial_i u^i \\
     &\simeq \partial u^0 + \partial r + r \partial u \\
     \implies \partial_t \tilde{u}^0 &= C_1^i \partial_i \tilde{u}^0+ C_2^i \partial_i \tilde{r} + C_3 r \partial_i \tilde{u}^i + C_3\tilde{r} \partial_i u^i + \widetilde{C_1^i} \partial_i u^0  + \widetilde{C_2^i} \partial_i r  + \widetilde{C_3} r \partial_i u^i\\
\end{split}
\end{equation}
We have the computations
\begin{align}
    a_1 &= u^0 - \frac{(\gamma-1)((u^0)^2-1)}{(\Gamma+r)u^0}r  \\
    \implies \tilde{a}_1 &= \tilde{u}^0 - \frac{(\gamma-1)((u^0)^2-1)}{(\Gamma+r)u^0} \tilde{r} \\
    & \hspace{5mm}-  r \left( \frac{2(\Gamma+r)(\gamma-1)u^0 \tilde{u}^0 -(\gamma-1)((u^0)^2-1)[(\Gamma' \tilde{s} + \tilde{r})u^0 + (\Gamma +r)\tilde{u}^0]  }{(\Gamma+r)^2 (u^0)^2} \right)
\end{align}
and
\begin{align}
    C_1^i := -\frac{u^i}{a_1} &\implies \widetilde{C_1^i} = \frac{u^i \tilde{a}_1 - a_1 \tilde{u}^i}{(a_1)^2}\\ 
    C_2^i := \frac{1}{(\Gamma+r)u^0} C_1^i  &\implies \widetilde{C_2^i} = \frac{1}{(\Gamma+r)u^0} \widetilde{C_1^i} - \frac{[(\Gamma' \tilde{s} + \tilde{r})u^0 + (\Gamma +r)\tilde{u}^0]}{(\Gamma+r)^2 (u^0)^2} C_1^i  \\ 
    C_3 := \frac{(\gamma-1)((u^0)^2-1)}{u^i} C_2^i  &\implies \widetilde{C_3} = \frac{(\gamma-1)((u^0)^2-1)}{u^i} \widetilde{C_2^i} \\
    \hspace{2mm} + \frac{u^i 2(\gamma-1) \tilde{u}^0 -(\gamma-1)((u^0)^2-1) \tilde{u}^i  }{(u^i)^2} C_2^i.
\end{align}

Note that after analyzing all of the terms (with $k=1$ here), we see that a large number are subcritical, and we can simplify:
\begin{align}
    \partial_t \tilde{u}^0 &\simeq \partial \tilde{u}^0 + \partial \tilde{r} + r \partial \tilde{u} + \tilde{r} \partial u \\
     &\simeq ( \text{$\mathcal{H}^2$ super-critical of order $-\frac{1}{2}$ }) + (\text{$\mathcal{H}^2$ critical}) + (\text{$\mathcal{H}^2$ sub-critical of order $\frac{1}{2}$})
\end{align}
Similiary, we can simplify $\partial_t r$ and $\partial_t \tilde{r}$ in the following way:
\begin{equation}
\begin{split}
    \partial_t r &= - \frac{u^i}{u^0} \partial_i r - \frac{\gamma-1}{u^0}r \partial_i u^i -  \frac{\gamma-1}{ u^0} r \left(C_1^i \partial_i u^0 + C_2^i \partial_i r + C_3 r \partial_i u^i  \right) \\
    \implies \partial_t \tilde{r} &= - \frac{u^i}{u^0} \partial_i \tilde{r} - \frac{\gamma-1}{u^0}r \partial_i \tilde{u}^i -  \frac{\gamma-1}{ u^0} r \left(\partial_t \tilde{u}^0 \right) \\
    & \hspace{9mm} + \frac{u^i \tilde{u}^0 - u^0 \tilde{u}^i }{(u^0)^2} \partial_i r + \left(\frac{(\gamma-1)}{(u^0)^2} \tilde{u}^0 r - \frac{\gamma-1}{u^0} \tilde{r}\right) \partial_i u^i + \left(\frac{(\gamma-1)}{(u^0)^2} \tilde{u}^0 r - \frac{\gamma-1}{u^0} \tilde{r}\right) \partial_t u^0
    \\
    & \simeq   \partial \tilde{r} + r \partial \tilde{u} + r \partial \tilde{u}^0 + r \partial \tilde{r} + r^2 \partial \tilde{u} + (\text{additional subcritical terms}) \\          
     \partial_t u^j &= - \frac{1}{(\Gamma+r)u^0} \proj^{j 0} \left[\partial_t r\right]- \frac{u^i}{u^0} \partial_i u^j -  \frac{1}{(\Gamma+r)u^0} \proj^{j i} \partial_i r \\
     & \simeq -[\partial_t r] - \partial u - \partial r
\end{split}
\end{equation}
The expressions for $\partial_t s$ and $\partial_t \tilde{s}$ take the easiest form, as 
\begin{equation}
    \begin{split}
          \partial_t s &= - \frac{u^i}{u^0} \partial_i s \\ 
          \partial_t \tilde{s} &=  - \frac{u^i}{u^0} \partial_i \tilde{s} + \frac{u^i \tilde{u}^0 - u^0 \tilde{u}^i }{(u^0)^2} \partial_i s
    \end{split}
\end{equation}

We also see that $(u^0)^2 = 1 + u^j u_j$, so
\begin{equation}
\label{E:solve_for_partial_i_u^0}
\begin{split}
        \tilde{u}^0 &= \frac{u_j}{u^0} \tilde{u}^j \\
     \partial_i \tilde{u}^0 &= \frac{u_j}{u^0} \partial_i \tilde{u}^j + \tilde{u}^j \partial_i \frac{u_j}{u_0}\\
\end{split}
\end{equation}

Thus, linearizing our expression for $\partial_t u^j$, we get

\begin{align}
    \partial_t \tilde{u}^j & \simeq [\partial_t \tilde{r}] + \partial \tilde{u} + \partial \tilde{r} \\
    &\simeq \partial \tilde{u} + \partial \tilde{r} + (r \partial \tilde{u} + r \partial \tilde{r} + r^2 \partial \tilde{u} + (\text{additional subcritical terms}))
\end{align}

Starting from \eqref{E:start_solving_for_time},  we can further solve for $\partial_t u^0$ by plugging in $\alpha=0$ into the last equation. If we set $a_1 = u^0 - \frac{\gamma-1}{(\Gamma+r)u^0}r \proj^{00}$, then 
\begin{equation}
\begin{split}
     \partial_t u^0 &= \frac{1}{a_1} \left(-u^i \partial_i u^0 + \frac{\gamma-1}{(\Gamma+r)u^0}r \proj^{00} \partial_i u^i + \frac{1}{(\Gamma+r)u^0}\proj^{00} u^i \partial_i r - \frac{1}{\Gamma+r} \proj^{0i} \partial_i r \right) \\
     &= -\frac{1}{a_1} u^i \partial_i u^0 - \frac{1}{a_1(\Gamma+r)u^0} u^i \partial_i r + \frac{(\gamma-1)((u^0)^2-1)}{a_1(\Gamma+r)u^0}r \partial_i u^i \\
     &:= C_1^i \partial_i u^0 + C_2^i \partial_i r + C_3 r \partial_i u^i \\
     &\simeq \partial u^0 + \partial r + r \partial u \\
     \implies \partial_t \tilde{u}^0 &= C_1^i \partial_i \tilde{u}^0+ C_2^i \partial_i \tilde{r} + C_3 r \partial_i \tilde{u}^i + C_3\tilde{r} \partial_i u^i + \widetilde{C_1^i} \partial_i u^0  + \widetilde{C_2^i} \partial_i r  + \widetilde{C_3} r \partial_i u^i\\
\end{split}
\end{equation}
We have the computations
\begin{align}
    a_1 &= u^0 - \frac{(\gamma-1)((u^0)^2-1)}{(\Gamma+r)u^0}r  \\
    \implies \tilde{a}_1 &= \tilde{u}^0 - \frac{(\gamma-1)((u^0)^2-1)}{(\Gamma+r)u^0} \tilde{r} \\
    & \hspace{5mm}-  r \left( \frac{2(\Gamma+r)(\gamma-1)u^0 \tilde{u}^0 -(\gamma-1)((u^0)^2-1)[(\Gamma' \tilde{s} + \tilde{r})u^0 + (\Gamma +r)\tilde{u}^0]  }{(\Gamma+r)^2 (u^0)^2} \right)
\end{align}
and
\begin{align}
    C_1^i := -\frac{u^i}{a_1} &\implies \widetilde{C_1^i} = \frac{u^i \tilde{a}_1 - a_1 \tilde{u}^i}{(a_1)^2}\\ 
    C_2^i := \frac{1}{(\Gamma+r)u^0} C_1^i  &\implies \widetilde{C_2^i} = \frac{1}{(\Gamma+r)u^0} \widetilde{C_1^i} - \frac{[(\Gamma' \tilde{s} + \tilde{r})u^0 + (\Gamma +r)\tilde{u}^0]}{(\Gamma+r)^2 (u^0)^2} C_1^i  \\ 
    C_3 := \frac{(\gamma-1)((u^0)^2-1)}{u^i} C_2^i  &\implies \widetilde{C_3} = \frac{(\gamma-1)((u^0)^2-1)}{u^i} \widetilde{C_2^i} \\
    \hspace{2mm} + \frac{u^i 2(\gamma-1) \tilde{u}^0 -(\gamma-1)((u^0)^2-1) \tilde{u}^i  }{(u^i)^2} C_2^i.
\end{align}

Note that after analyzing all of the terms (with $k=1$ here), we see that a large number are subcritical, and we can simplify:
\begin{align}
    \partial_t \tilde{u}^0 &\simeq \partial \tilde{u}^0 + \partial \tilde{r} + r \partial \tilde{u} + \tilde{r} \partial u \\
     &\simeq ( \text{$\mathcal{H}^2$ super-critical of order $-\frac{1}{2}$ }) + (\text{$\mathcal{H}^2$ critical}) + (\text{$\mathcal{H}^2$ sub-critical of order $\frac{1}{2}$})
\end{align}
Similiary, we can simplify $\partial_t r$ and $\partial_t \tilde{r}$ in the following way:
\begin{equation}
\begin{split}
    \partial_t r &= - \frac{u^i}{u^0} \partial_i r - \frac{\gamma-1}{u^0}r \partial_i u^i -  \frac{\gamma-1}{ u^0} r \left(C_1^i \partial_i u^0 + C_2^i \partial_i r + C_3 r \partial_i u^i  \right) \\
    \implies \partial_t \tilde{r} &= - \frac{u^i}{u^0} \partial_i \tilde{r} - \frac{\gamma-1}{u^0}r \partial_i \tilde{u}^i -  \frac{\gamma-1}{ u^0} r \left(\partial_t \tilde{u}^0 \right) \\
    & \hspace{9mm} + \frac{u^i \tilde{u}^0 - u^0 \tilde{u}^i }{(u^0)^2} \partial_i r + \left(\frac{(\gamma-1)}{(u^0)^2} \tilde{u}^0 r - \frac{\gamma-1}{u^0} \tilde{r}\right) \partial_i u^i + \left(\frac{(\gamma-1)}{(u^0)^2} \tilde{u}^0 r - \frac{\gamma-1}{u^0} \tilde{r}\right) \partial_t u^0
    \\
    & \simeq   \partial \tilde{r} + r \partial \tilde{u} + r \partial \tilde{u}^0 + r \partial \tilde{r} + r^2 \partial \tilde{u} + (\text{additional subcritical terms}) \\          
     \partial_t u^j &= - \frac{1}{(\Gamma+r)u^0} \proj^{j 0} \left[\partial_t r\right]- \frac{u^i}{u^0} \partial_i u^j -  \frac{1}{(\Gamma+r)u^0} \proj^{j i} \partial_i r \\
     & \simeq -[\partial_t r] - \partial u - \partial r
\end{split}
\end{equation}
The expressions for $\partial_t s$ and $\partial_t \tilde{s}$ take the easiest form, as 
\begin{equation}
    \begin{split}
          \partial_t s &= - \frac{u^i}{u^0} \partial_i s \\ 
          \partial_t \tilde{s} &=  - \frac{u^i}{u^0} \partial_i \tilde{s} + \frac{u^i \tilde{u}^0 - u^0 \tilde{u}^i }{(u^0)^2} \partial_i s
    \end{split}
\end{equation}

We also see that $(u^0)^2 = 1 + u^j u_j$, so
\begin{equation}
\begin{split}
        \tilde{u}^0 &= \frac{u_j}{u^0} \tilde{u}^j \\
     \partial_i \tilde{u}^0 &= \frac{u_j}{u^0} \partial_i \tilde{u}^j + \tilde{u}^j \partial_i \frac{u_j}{u_0}\\
\end{split}
\end{equation}

Thus, linearizing our expression for $\partial_t u^j$, we get
\begin{align}
    \partial_t \tilde{u}^j & \simeq [\partial_t \tilde{r}] + \partial \tilde{u} + \partial \tilde{r} \\
    &\simeq \partial \tilde{u} + \partial \tilde{r} + (r \partial \tilde{u} + r \partial \tilde{r} + r^2 \partial \tilde{u} + (\text{additional subcritical terms}))
\end{align}

\bibliographystyle{plain}
\bibliography{references.bib}

\end{document}